\pdfoutput=1
\documentclass[11pt]{amsart}
\usepackage{amsmath,amssymb,amsthm}
\usepackage{mathtools}
\usepackage[margin=1.2in]{geometry}
\usepackage{booktabs}
\usepackage{array}
\usepackage{adjustbox,longtable,needspace}
\usepackage{flafter}
\usepackage{xcolor}
\usepackage{tikz}
\usepackage[colorlinks=true,linkcolor=blue,citecolor=blue,urlcolor=blue]{hyperref}
\hypersetup{pdfauthor={Bernd Johannes Wuebben},
  pdftitle={Smoothing Calabi-Yau threefolds with nodes and del Pezzo cone points}}

\newtheorem{theorem}{Theorem}[section]
\newtheorem{proposition}[theorem]{Proposition}
\newtheorem{lemma}[theorem]{Lemma}
\newtheorem{corollary}[theorem]{Corollary}
\newtheorem{recollection}[theorem]{Recollection}

\theoremstyle{definition}
\newtheorem{definition}[theorem]{Definition}

\newtheorem{convention}[theorem]{Convention}
\theoremstyle{remark}
\newtheorem{remark}[theorem]{Remark}

\theoremstyle{plain}
\newtheorem{thmletter}{Theorem}

\newcommand{\ZZ}{\mathbb{Z}}
\newcommand{\QQ}{\mathbb{Q}}
\newcommand{\RR}{\mathbb{R}}
\newcommand{\CC}{\mathbb{C}}
\newcommand{\PP}{\mathbb{P}}
\newcommand{\AAA}{\mathbb{A}}
\newcommand{\cO}{\mathcal{O}}
\newcommand{\Rsm}{\mathcal{R}}   % smoothing subspace of the link homology
\newcommand{\Mil}{\mathcal{M}}  % Milnor fibre
\newcommand{\cX}{\mathcal{X}}
\newcommand{\cP}{\mathcal{P}}
\newcommand{\cD}{\mathcal{D}}
\newcommand{\cL}{\mathcal{L}}
\newcommand{\cA}{\mathcal{A}}
\newcommand{\cS}{\mathcal{S}}
\newcommand{\conv}{\operatorname{conv}}
\newcommand{\Cone}{\operatorname{Cone}}

\newcommand{\Sing}{\operatorname{Sing}}
\newcommand{\tail}{\operatorname{tail}}

\newcommand{\Def}{\operatorname{Def}}
\newcommand{\rk}{\operatorname{rank}}

\newcommand{\dP}{\mathrm{dP}}
\newcommand{\TV}{\operatorname{TV}}
\newcommand{\Hom}{\operatorname{Hom}}

\newcommand{\Xh}{\widehat{X}}

\title[Nodes and del Pezzo cone points]{Smoothing Calabi--Yau threefolds\\ with
nodes and del Pezzo cone points}
\author{Bernd Johannes Wuebben}

\date{7 September 2026}
\subjclass[2020]{14J32 (primary); 14B07, 14M25, 14D15, 52B20 (secondary)}
\keywords{Calabi--Yau threefolds, smoothing of singularities, del Pezzo cones,
ordinary double points, Minkowski decomposition, $T$-varieties, reflexive
polytopes}

\begin{document}

\begin{abstract}
We give a necessary and sufficient homological criterion for smoothing
projective Calabi--Yau threefolds with nodes and anticanonical cones over
smooth del Pezzo surfaces of degrees five, six and seven, or over a smooth
quadric surface. For some choice of local smoothing components, the
criterion requires a relation on a crepant resolution among link classes
whose associated local parameters avoid every discriminant. The
complexified relation space describes first-order deformations modulo
locally trivial ones. Necessity follows from local period and vanishing-cycle
arguments; sufficiency applies Friedman's unobstructedness theorem to a
crepant partial resolution. Whenever the criterion holds, the
deformation space on the chosen components is smooth. A nonsmoothable
example, whose deformation space is reduced with four smooth components,
shows why nonvanishing at every singular point is insufficient. We also
construct a smoothable Batyrev hypersurface admitting no smoothing induced
by proper ambient families obtained from Minkowski decompositions at a
single lattice degree.
\end{abstract}

\maketitle

\section{Introduction}\label{sec:intro}

A projective Calabi--Yau threefold can have only locally smoothable
singularities and nevertheless admit no smoothing.  For ordinary double
points, the obstruction is expressed by the exceptional curves of a small
resolution: a smoothing exists precisely when their classes satisfy a
relation
\begin{equation}\label{eq:friedman}
 \sum_j\lambda_j[C_j]=0,\qquad \lambda_j\ne0\quad\text{for every }j.
\end{equation}
This is Friedman's criterion, together with the unobstructedness and
necessity results discussed in \cite{Friedman86,Namikawa94,RT09};
\cite[Thm.~1.3]{RT09} states the smoothing criterion itself.
We prove an analogous criterion when nodes occur together with
anticanonical cones over smooth del Pezzo surfaces of degrees five, six and seven,
and over $\PP^1\times\PP^1$, of degree eight.
The degree-six cone has a two-dimensional smoothing component, and a
nonzero vector on that component need not be a smoothing direction.
Three separate nonvanishing conditions are necessary.  An explicit
nonsmoothable threefold shows that they cannot be replaced by a single
nonzero-vector condition, even if the local components may be chosen anew.
For nodes, degree-seven cones and degree-six cones smoothed through their
one-dimensional component, the necessity direction is the vanishing-cycle
obstruction of the companion paper~\cite{PaperIV}; the present paper proves
the converse and treats the remaining local components.

\subsection{The smoothing criterion}
Let $X$ be a connected normal projective complex threefold with
$\omega_X\cong\cO_X$ and $H^1(X,\cO_X)=0$.  Suppose every singularity is
an ordinary double point or is analytically isomorphic to the
anticanonical cone over a smooth del Pezzo surface of degree five, six or seven,
or over $\PP^1\times\PP^1$.
The assumption concerns these exact analytic germs.
Let $Y_0\to X$ be the smooth crepant analytic resolution obtained from
chosen small resolutions of the nodes and the standard divisorial
resolutions of the cones.  It is compact and Moishezon.

A \emph{branch profile} $S$ chooses a reduced local deformation component
at each singular point.  At a node or a $\PP^1\times\PP^1$ cone the base
is a line; at a degree-five point it is a smooth four-dimensional germ;
at a degree-seven
point there is one reduced line; at a degree-six point there are a line
and a plane, denoted $A_1$ and $A_2$.  These names agree with the root
subspaces in the del Pezzo lattice.  For the link $L_p$ of a cone point,
the circle-bundle projection identifies
\[
 H_2(L_p;\QQ)\cong K_{E_p}^{\perp}\subset H_2(E_p;\QQ).
\]
Write $R_{p,S_p}$ for the nodal line or the selected root subspace, and set
\begin{equation}\label{eq:relationkernel}
 K_S=\ker\left\{\bigoplus_pR_{p,S_p}
                         \longrightarrow H_2(Y_0;\QQ)\right\}.
\end{equation}
At a $\PP^1\times\PP^1$ cone the selected subspace is the entire
link-homology line $\QQ(f_1-f_2)$, where $f_1,f_2$ are the ruling
classes on the exceptional surface.
At a degree-six point, mark $E_p=\mathrm{Bl}_3\PP^2$ by the line class
$h$ and exceptional classes $e_1,e_2,e_3$.  Its $A_2$ subspace is
\[
R_{p,A_2}=\left\{\mu_1e_1+\mu_2e_2+\mu_3e_3:
                                  \mu_1+\mu_2+\mu_3=0\right\}.
\]
At a degree-five point mark $E_p=\operatorname{Bl}_4\PP^2$
by $h,e_1,\ldots,e_4$.  The selected subspace is the entire
$K_{E_p}^\perp$, and its five conic-pencil classes are
\[
 f_i=h-e_i\ (1\le i\le4),\qquad
 f_5=2h-e_1-e_2-e_3-e_4.
\]
Put $\lambda_i(\kappa_p)=\kappa_p\cdot f_i$.  Their sum is
zero because $\sum_i f_i=-2K_{E_p}$.

\begin{thmletter}[smoothing criterion]\label{thm:criterionintro}
The threefold $X$ has a convergent one-parameter smoothing inducing $S$
if and only if $K_S$ contains a vector nonzero at every rank-one summand
and satisfying $\mu_1\mu_2\mu_3\ne0$ at every $A_2$ summand
and $\prod_{i=1}^5\lambda_i\ne0$ at every degree-five point.
\end{thmletter}

Equivalently, none of the stated scalar functionals vanishes identically
on $K_S$.  Thus the criterion is a finite linear-algebra test for each
of the finitely many profiles.  When neither degree-five nor degree-six
points occur, its form is the same as the nodal criterion.  At an $A_2$
point the three excluded lines are precisely the local discriminant.
Necessity holds for arbitrary smoothing arcs, including arcs tangent to
that discriminant to higher order.
At degree five, the five excluded hyperplanes are likewise the
exact local discriminant, and necessity holds for arbitrary arcs.

Fix analytic local classifying maps for a semiuniversal deformation of
$X$, and let $\Def^S(X)$ be the scheme-theoretic inverse image of the
chosen local branches.  Allowing smoothings over arbitrary complex analytic
base germs does not change the criterion.  Indeed, take a small
representative of $B=\Def^S(X)_{\mathrm{red}}$ and restrict the
proper semiuniversal family to $B$.  The locus $D\subset B$ of singular
fibres is closed analytic by properness.  If smooth fibres occur
arbitrarily close to the origin, complex curve selection
\cite[Lem.~2.36]{LeMassey05} gives a holomorphic map
$\gamma:(\Delta,0)\to(B,0)$ with
$\gamma(\Delta\setminus\{0\})\subset B\setminus D$ after shrinking
the disc.  Pullback gives a convergent proper one-parameter smoothing
inducing $S$.  Conversely, a disc is itself an analytic base germ.
For an unrestricted smoothing, first apply the same argument in
$\Def(X)_{\mathrm{red}}$; the irreducibility of the disc forces each
local classifying map to factor through a reduced local component,
and hence determines a profile.  Thus $X$ is smoothable if and only if
the criterion holds for at least one profile $S$.  This argument is
independent of the singularity types, and applies in particular to
degree-five cones and anticanonical $\PP^1\times\PP^1$ cones;
compare \cite[proof of Cor.~6.7]{PaperIV}.

The integrability statement is stronger than the existence assertion.
Denote the tangent sheaf by $\Theta_X$.

\begin{thmletter}[smooth selected deformation spaces]\label{thm:smoothintro}
If $K_S\to R_{p,S_p}$ is nonzero at every degree-six point, then
$\Def^S(X)$ is smooth, of dimension
\[
 h^1(X,\Theta_X)+\dim_{\QQ}K_S.
\]
Every tangent vector to this space is the tangent of a convergent
one-parameter deformation on the selected branches.
\end{thmletter}

No nonzero-projection hypothesis is required at nodes, degree-five or degree-seven
points, or $\PP^1\times\PP^1$ cones in this theorem.  The hypothesis
is automatic when the smoothing
criterion holds.  Smoothness of a selected base does not, by itself,
assert the existence of a smooth fibre.

\subsection{Scope among del Pezzo cones}\label{sec01:degree-scope}

The degree-seven and degree-six cases arise from the pentagonal and
hexagonal faces in the toric examples.  Both associated cones are smoothable and are
not complete intersections.  Their local deformation spaces exhibit
the two phenomena treated here: distinct smoothing components in
degree six, and a nonreduced structure transverse to the smoothing
line in degree seven.  Theorems~\ref{thm:criterionintro}
and~\ref{thm:smoothintro} also allow the degree-five cone and the
anticanonical cone over $\PP^1\times\PP^1$, of degree eight.
These hypotheses specify the
local comparisons established below; smoothable del Pezzo cones also
occur in the other degrees.
For every degree covered by the criterion, the curve-selection argument
following Theorem~\ref{thm:criterionintro} identifies smoothability over
complex analytic base germs with the stated one-parameter condition.

For a smooth del Pezzo surface $E$, the anticanonical cone means
\[
 \operatorname{Spec}\bigoplus_{m\geq0}H^0(E,-mK_E).
\]
This convention includes the weighted
anticanonical models in degrees one and two.  The classical local
facts in Table~\ref{tab01:degree-scope} are recalled in
\cite[\S\S3.3, 5.5 and Example~2.19]{Gross97rev}; the branch structures
in degrees six and seven follow from \cite{Altmann97}.
Proposition~\ref{prop03:degree5family} proves the full local-base
and discriminant statements in degree five.

% Canonical degree-scope table.  Keep its final column synchronized with
% the theorem statements and the degree-five/degree-eight research TODO.
% Local smoothability and an established global criterion are distinct.
\begin{table}[htbp]
\centering
\caption{Anticanonical cones over smooth del Pezzo surfaces and the
scope of the present results.  The local assertions alone do not imply
smoothability of a compact threefold containing these germs.}
\label{tab01:degree-scope}
\small
\renewcommand{\arraystretch}{1.14}
\begin{adjustbox}{max width=\textwidth, center}
\begin{tabular}{@{}>{\raggedright\arraybackslash}p{0.17\textwidth}
                  >{\raggedright\arraybackslash}p{0.41\textwidth}
                  >{\raggedright\arraybackslash}p{0.32\textwidth}@{}}
\toprule
Degree / surface & Local deformation behavior & Role in the present results \\
\midrule
$1$--$3$
  & Hypersurfaces; smoothable, with smooth local bases.
  & Classical complete-intersection cases; outside the stated mixed criterion. \\
\addlinespace
$4$
  & Intersection of two quadrics; smoothable, with smooth local base.
  & As in degrees one to three. \\
\addlinespace
$5$
  & Smooth four-dimensional base; five discriminant hyperplanes.
  & Covered by Theorems~\ref{thm:criterionintro}
    and~\ref{thm:smoothintro}; five nonzero conic-pencil pairings. \\
\addlinespace
$6$
  & A smoothing line and a smoothing plane; three discriminant lines in the plane.
  & Both components covered by Theorems~\ref{thm:criterionintro}
    and~\ref{thm:smoothintro}. \\
\addlinespace
$7$
  & One reduced smoothing line, with an obstructed transverse tangent.
  & Covered by Theorems~\ref{thm:criterionintro}
    and~\ref{thm:smoothintro}. \\
\addlinespace
$8$, $\PP^1\!\times\!\PP^1$
  & Smooth one-dimensional local base; noncentral fibres smooth.
  & Covered by Theorems~\ref{thm:criterionintro}
    and~\ref{thm:smoothintro}; coefficient on the ruling difference. \\
\addlinespace
$8$, $\mathbb F_1$
  & No smoothing; local base $\operatorname{Spec}\CC[t]/(t^2)$.
  & Excluded by local nonsmoothability. \\
\addlinespace
$9$, $\PP^2$
  & Rigid.
  & Excluded by local nonsmoothability. \\
\bottomrule
\end{tabular}
\end{adjustbox}
\end{table}

The degree-eight cone over $\PP^1\times\PP^1$ already occurs as the
partial resolution of a degree-seven point.
Lemma~\ref{lem03:quadric} proves smoothness of its full local base and
identifies its tangent line with the difference of the ruling classes.
When such points occur on the original threefold, the partial
modification is the identity near them.  Their Milnor fibres retract
to $\mathbb{RP}^3$; the descended nodal period supplies the necessity
argument in Lemma~\ref{lem04:periodnecessity}.  Thus they require one
nonzero coefficient on the ruling difference, with no additional
hypothesis in the selected-base smoothness theorem.

The degree-five surface is not toric.  Totaro's Bott vanishing
theorem \cite{totaro2020} supplies its local Hodge comparison.
The four-parameter family of Grassmannian sections and its
boundary periods identify the five conic-pencil conditions and
prove necessity at every order.  Degree-five points also remain
unchanged on the partial model.  Thus the present criterion covers
all locally smoothable anticanonical cones over smooth del Pezzo
surfaces that are not complete intersections.
For degrees one to four, the existing complete-intersection theory
does not by itself give the stated criterion for mixtures with the
non-complete-intersection cones considered in this paper.

The restriction is not a failure of local smoothability.  For a compact
threefold with only isolated canonical local complete intersection
singularities, trivial canonical bundle, $H^1(\cO)=0$, and a K\"ahler
resolution, Namikawa's theorem gives unobstructedness.  With a K\"ahler
crepant resolution, Gross's theorem gives a deformation smoothing the
non-nodal points while remaining locally trivial at the existing nodes
\cite[Theorems~0.3 and~0.5]{FL22a}.  These hypotheses do not cover a
mixture in which the non-complete-intersection cone points remain.  Gross
also smooths non-complete-intersection points in two settings: when $X$ is
$\QQ$-factorial and all its singular points are good in his sense
\cite[Thm.~4.3]{Gross97rev}, and for the target of a primitive contraction of
a divisor to a point, unless that divisor is $\PP^2$ or $\mathbb F_1$
\cite[Thm.~5.8]{Gross97rev}.  Both settings concern $\QQ$-factorial threefolds,
whereas the threefolds considered here may carry nodes and several cone points
without being $\QQ$-factorial.

More recently, Imagi proves a global smoothing theorem for
$\QQ$-factorial projective Calabi--Yau threefolds with isolated canonical
singularities whose full local deformation bases have dense smoothing
loci \cite[Theorems~1.8--1.9 and Corollary~1.10]{Imagi26}.
This allows non-complete-intersection singularities and singular local
bases.  The present results give a criterion on each selected profile
without assuming $\QQ$-factoriality, together with smoothness of the
selected base under the stated projection hypotheses.

There is also a change in the local comparison used here.  At the cone
over the smooth Fermat cubic surface, put $f=\sum_{i=0}^3x_i^3$.
The hypersurface cotangent presentation gives
\[
 T^1_{\{f=0\},0}
 =\frac{\CC[x_0,x_1,x_2,x_3]}{(x_0^2,x_1^2,x_2^2,x_3^2)},
 \qquad \dim T^1_{\{f=0\},0}=16.
\]
The sixteen squarefree monomials form a basis.  The link $L$ is the
circle bundle of the canonical bundle of the cubic surface, so the
Gysin sequence gives $\dim H_2(L;\CC)=7-1=6$.  Hence the full-tangent
identification of Lemma~\ref{lem:localhodge} cannot hold for this cone.
An extension to mixtures including degrees one to four must account for
deformation directions not represented by link homology.  This dimension
comparison rules out that identification, not a different mixed
smoothing theorem.

\Needspace{18\baselineskip}
\subsection{A nonsmoothable example and its deformation space}
The distinction between a nonzero local vector and a smoothing direction
cannot be removed by changing the branch profile.
For two degree-six points, a mixed profile selects the line component
at one point and the plane component at the other.  A uniform profile
selects the same component type at both points.

\begin{samepage}
\begin{thmletter}\label{thm:counterintro}
There is a projective Calabi--Yau threefold with thirty nodes and two
degree-six cone points such that a relation in $K_S$ is nonzero at every
singular point for either mixed profile, but the threefold has no
smoothing on any profile.  Its semiuniversal analytic local ring is
\[
 \CC\{z_1,\ldots,z_{32},x,y,u,v,b\}/(xy,xb,uv,ub).
\]
It is reduced, has tangent dimension $37$, and has four smooth
irreducible components of dimensions $34,34,34,35$.  All their
scheme-theoretic intersections are smooth.
\end{thmletter}
\end{samepage}

The example is a general anticanonical hypersurface in the toric
fourfold associated with an explicit reflexive polytope with $25$
vertices.  Three divisor identities prove nonsmoothability independently
of the sufficiency argument.  They force an old nodal coefficient to
vanish on each uniform profile and one of the three $A_2$ coefficients
to vanish on each mixed profile.  The full deformation ring then follows
from Theorem~\ref{thm:smoothintro}, the local branch equations, and a
convergent change of coordinates.  A prescribed tangent integrates
precisely when the four displayed quadratic expressions vanish.

\subsection{The geometric argument}
The first-order comparison identifies the image of the global tangent
space in the sum of the local tangent spaces with the kernel of the
link-homology map.  Contraction with one global volume form, local
cohomology with supports, and the $\partial\bar\partial$-lemma give this
comparison for the full local tangent spaces, including the transverse
infinitesimal direction at a degree-seven cone.

To integrate the selected directions, we replace the cone points by
proper crepant partial resolutions.  The degree-seven cone is replaced
by the anticanonical cone over $\PP^1\times\PP^1$; the $A_1$ and $A_2$
choices at a degree-six cone are replaced by two and three nodes,
respectively.  Degree-five and original $\PP^1\times\PP^1$
cones remain unchanged.
Each resulting singularity has a smooth full local
deformation base.  Friedman's generalized $T^1$ lifting theorem
\cite[Thm.~3.10]{friedman2026}, combined with this local smoothness, gives the
$T^1$-lifting property and hence unobstructedness of the partial model
\cite[Thm.~3.3]{friedman2026}.  An analytic blow-down lemma and a lift of locally
trivial global tangents transfer this smoothness to the selected
deformation space of $X$.  The partial model need not be projective;
compact Moishezon resolutions supply the required Hodge theory.

For necessity at an $A_2$ point, an explicit simultaneous partial
resolution of the rank-one matrix family is an isomorphism over every
noncentral fibre.  Pulling it back along an arbitrary smoothing arc
replaces the central cone by three nodes without changing the smooth
fibres.  Nodal necessity then supplies three nonzero coefficients;
intersection with the exceptional del Pezzo surface makes their sum
zero.  This proves the stronger $A_2$ condition at every order.
At a $\PP^1\times\PP^1$ cone, the vanishing cycle is the quotient
of the nodal sphere by the antipodal involution; its nonzero period
gives the corresponding rank-one necessity condition.
At degree five, the four-parameter Grassmannian section family
has five discriminant hyperplanes.  Its boundary carries every
absolute Milnor cycle.  The resulting holomorphic boundary periods
identify the five conditions and prove necessity along arbitrary arcs.

\subsection{Toric constructions and their limits}
The general theorems also distinguish abstract deformations from
deformations induced by a toric ambient space.  For Batyrev
hypersurfaces \cite{Batyrev94}, we study the ambient families obtained
from an admissible Minkowski decomposition at one primitive lattice
degree, using Altmann's local theory and the global construction of
Ilten--Vollmert \cite{Altmann97,IV12}.  Definition~\ref{def:singledegree}
specifies these families.  Ambient deformations of the toric pair built from
Minkowski decompositions of two-dimensional faces with long dual edges are
constructed in \cite{PaperII}; the local obstruction of \cite{PaperI} and the
doubly isolated mirror pairs of \cite{PaperIII} supply the toric examples
considered here.

The hypersurface $X_{19}$ of a polytope from \cite[\S5.3]{PaperIII} has
fourteen nodes and one degree-seven cone point.  Corollary~\ref{cor:x19smooth} proves that its full deformation
space is smooth of dimension $30$ and that it is smoothable.
Nevertheless, Theorem~\ref{thm:A} excludes a smoothing by every
single-degree ambient family through the relative anticanonical system.
Its proof combines Altmann's grading calculation with a propagation
rule on the edges of the two facets containing the pentagonal face.
Definition~\ref{def:locked} calls a facet \emph{locked} when this rule
forces all its edge-dilation parameters to agree.  The obstruction is
uniform over the entire degree lattice.

The complementary construction is $X_9$, with two nodes and one
degree-seven cone point.  Theorem~\ref{thm:B} gives an explicit ambient
smoothing at degree $(0,1,-1,-1)$.  Theorem~\ref{thm:C} classifies all
$77$ admissible polytopes with at most nine vertices and a pentagonal
face: $76$ give nonsmoothable hypersurfaces, and the remaining one is
$X_9$.  The larger edge-propagation census describes the frequency of
the ambient obstruction, without asserting that its separate
degree-coverage condition holds throughout that census.

Sections~\ref{sec:deformation}--\ref{sec04:smoothing} prove the general
theorems, and Section~\ref{sec05:counterexample} gives the counterexample
and its full deformation ring.  Sections~\ref{sec:germs}--\ref{sec:thmB}
develop the toric applications.  Appendix~\ref{app:cert} records the
exact computations and the data from which they can be reproduced.
The results leave open smoothness of selected bases whose degree-six
projection vanishes, and the structure of unrestricted deformation
bases in general, including possible degree-seven nilpotents.
For the toric hypersurfaces of Definition~\ref{def:framework}, the companion
paper~\cite{PaperVI} studies the relation lattices on the mirror side: for
nodal crepant partial resolutions, one for each profile, it identifies the
relations among the exceptional curves with the relations among the vanishing
three-spheres of the Batyrev mirror family.

\subsection*{Acknowledgements}
I am grateful to Mark Gross for first drawing my attention, many years ago, to the problem of smoothing canonical singularities on Calabi--Yau threefolds.
The author used LLMs during the development of this work for literature exploration, editorial and coding assistance and as an adversarial reader.

\section{Local branches and global first-order deformations}
\label{sec:deformation}

\subsection{The analytic setting and the local lattices}
Throughout Sections~\ref{sec:deformation}--\ref{sec04:smoothing}, $X$
satisfies the projective Calabi--Yau hypotheses of the introduction.
All its singularities are isolated nodes or the exact anticanonical
cones over $E=\dP_5$, $\dP_6$, $\dP_7$, or $\PP^1\times\PP^1$.
The degree-five, degree-six and degree-seven surfaces are unique up
to isomorphism over $\CC$.  The degree-five surface is not toric;
the other three surfaces are toric.  Fix a nowhere-vanishing
section $\Omega$ of $\omega_X$.

At a node choose a small resolution; at a cone use the zero-section
contraction $\operatorname{Tot}(K_E)\to\Cone(E)$.  These proper analytic
maps are isomorphisms off the singular points and glue to a smooth
crepant analytic resolution $f_0:Y_0\to X$.  The gluing is Hausdorff
because two points over the same singular point lie in one proper
local model; points with distinct images can be separated on $X$.
Properness is local on the target.  Thus $Y_0$ is compact, and it is
Moishezon because it is bimeromorphic to the projective variety $X$.
The $\partial\bar\partial$-lemma holds on $Y_0$
\cite[Thm.~5.22 and Cor.~5.23]{dgms1975}.

Let $L_p$ be a small link of a singularity.  At a cone point the
circle-bundle Gysin sequence identifies
\[
 H_2(L_p;\QQ)=K_{E_p}^{\perp}\subset H_2(E_p;\QQ).
\]
For $E=\mathrm{Bl}_r\PP^2$, use the usual marking $h,e_1,\ldots,e_r$.
Its intersection form and canonical class are
\[
 \operatorname{diag}(1,-1,\ldots,-1),\qquad K_E=-3h+\sum_i e_i.
\]
At degree seven the selected line is
$\QQ(e_1-e_2)$.  For $E=\PP^1\times\PP^1$, the ruling classes
satisfy $f_1^2=f_2^2=0$, $f_1\cdot f_2=1$, and
$K_E=-2f_1-2f_2$; the selected line is the entire space
$K_E^\perp=\QQ(f_1-f_2)$.  At degree six set
\begin{equation}\label{eq:rootcoordinates}
 \ell=-h+e_1+e_2+e_3,\qquad A=-e_1+e_3,\qquad B=e_2-e_3.
\end{equation}
Then $K_E^\perp=\QQ\ell\oplus\langle A,B\rangle_\QQ$, the orthogonal
decomposition into the $A_1$ and $A_2$ root subspaces.  In the plane,
\begin{equation}\label{eq:a2coefficients}
 aA+bB=-a e_1+b e_2+(a-b)e_3.
\end{equation}
At a node the chosen resolution identifies its link-homology line
with the span of the exceptional curve.  Changing the small resolution
changes its generator by a sign and does not change any nonvanishing test.

At degree five the selected subspace is all of $K_E^\perp$.
For $E=\operatorname{Bl}_4\PP^2$ its five conic-pencil classes are
\begin{equation}\label{eq:degree5pencils}
 f_i=h-e_i\quad(1\le i\le4),\qquad
 f_5=2h-e_1-e_2-e_3-e_4,\qquad \sum_{i=1}^5f_i=-2K_E.
\end{equation}
For $\gamma=xh+\sum_i y_i e_i\in K_E^\perp$, define
\begin{equation}\label{eq:degree5coordinates}
 \lambda_i(\gamma)=\gamma\cdot f_i=x+y_i\quad(i\le4),
 \qquad \lambda_5(\gamma)=-x,\qquad \sum_i\lambda_i=0.
\end{equation}
These coordinates identify the integral perpendicular lattice with
the negative $A_4$ lattice: the inverse is
$x=-\lambda_5$, $y_i=\lambda_i+\lambda_5$, and
$\gamma^2=-\sum_i\lambda_i^2$.

\begin{definition}\label{def:branchprofile}
A branch profile $S$ selects the entire line at each node and each
$\PP^1\times\PP^1$ cone, the entire four-dimensional base at
each degree-five point, the unique reduced
degree-seven branch, and one of the two reduced branches at every
degree-six point.  Write $R_{p,S_p}$ for the corresponding subspace of
$H_2(L_p;\QQ)$ described above, and $K_S$ for \eqref{eq:relationkernel}.
We also write
\begin{equation}\label{eq:fullrelationkernel}
 K=\ker\left\{\bigoplus_pH_2(L_p;\QQ)\longrightarrow H_2(Y_0;\QQ)\right\}.
\end{equation}
\end{definition}

These subspaces are also
$R_{p,S_p}=\ker\{H_2(L_p;\QQ)\to H_2(\Mil_{p,S_p};\QQ)\}$, where
$\Mil_{p,S_p}$ is a smooth Milnor fibre on the chosen component
\cite[Prop.~2.1 and Lem.~4.2]{PaperIV}.  For the
$\PP^1\times\PP^1$ and degree-five cones the same assertion follows
from $H_2(\Mil_{p,S_p};\QQ)=0$, proved in
Lemmas~\ref{lem03:quadricmilnor} and~\ref{lem03:degree5milnor}.
When a profile is fixed, we
occasionally abbreviate $R_{p,S_p}$ to $\Rsm_p$.

The scheme-theoretic local bases are, in analytic coordinates,
\begin{equation}\label{eq:localbases}
 \begin{array}{c|c}
 \text{germ}&\text{local base ring}\\ \hline
 \text{node}&\CC\{t\}\\
 \Cone(\PP^1\times\PP^1)&\CC\{t\}\\
 \Cone(\dP_5)&\CC\{b_1,b_2,b_3,b_4\}\\
 \Cone(\dP_7)&\CC\{s,w\}/(w^2,sw)\\
 \Cone(\dP_6)&\CC\{s,u,v\}/(su,sv).
 \end{array}
\end{equation}
The degree-five entry is proved in
Proposition~\ref{prop03:degree5family}; the other cone entries are
Altmann's calculations.  At degree six the $A_1$ branch is $u=v=0$ and the $A_2$ branch is
$s=0$ \cite[(2.5), (9.1)(v), \S\S5--8]{Altmann97}.  Fix analytic classifying maps from
a semiuniversal base $\Def(X)$ to these local bases.  Their
scheme-theoretic inverse image of the selected components is denoted
$\Def^S(X)$; the chosen maps form part of this notation.  Write
$V_S=T_0\Def^S(X)$.  The line and plane are smooth, but the plane's
smoothing locus excludes the three lines
$ab(a-b)=0$ in the marking \eqref{eq:a2coefficients}.  In particular,
the condition of being a nonzero tangent in the plane is weaker than
being a first-order smoothing direction.
At degree five the smoothing locus excludes the five hyperplanes
$\lambda_i=0$, as proved in Lemma~\ref{lem03:degree5marking}.

\subsection{The full local Hodge comparison}
Write $\mathbb T^i_X=\operatorname{Ext}^i(L_X,\cO_X)$, where $L_X$ is
the cotangent complex.  The following comparison permits the use of
K\"ahler differentials in the low-degree local-to-global sequence.

\begin{lemma}\label{lem:omegaL}
Let $X$ be a normal Gorenstein threefold whose non-lci locus is finite.  Then
\[
  \mathrm{Ext}^i(\Omega^1_X,\cO_X)\;\cong\;\mathrm{Ext}^i(L_X,\cO_X)
  \qquad(i\le2),
\]
both as sheaves and globally, compatibly with the local-to-global spectral
sequence.
\end{lemma}

\begin{proof}
The cohomology sheaves $\mathcal H^{-q}(L_X)$ for $q\ge1$ vanish where $X$ is a
local complete intersection, so they are supported on a finite set and have
finite length.  A finite-length module on a Cohen--Macaulay threefold has grade
$3$, so $\mathcal Ext^p(-,\cO_X)$ and $\mathrm{Ext}^p(-,\cO_X)$ vanish on it for
$p<3$.  In the spectral sequence
$E_2^{p,q}=\mathrm{Ext}^p(\mathcal H^{-q}(L_X),\cO_X)\Rightarrow
\mathrm{Ext}^{p+q}(L_X,\cO_X)$ a term with $q\ge1$ contributing in total degree
$\le2$ has $p\le1<3$, hence vanishes.  Only $q=0$ survives in degrees $\le2$, and
$\mathcal H^0(L_X)=\Omega^1_X$.
\end{proof}

The local-to-global sequence begins
\[
 0\longrightarrow H^1(X,\Theta_X)\longrightarrow\mathbb T^1_X
 \xrightarrow{\alpha}\bigoplus_pT^1_{X,p}
 \xrightarrow{\mathrm{ob}}H^2(X,\Theta_X).
\]
We use contraction with the restriction of the fixed global volume
form $\Omega$ in every local comparison below. Thus numerical local
parameters are Hodge-normalized; independently chosen equation
coordinates may differ by nonzero scalar factors.

\begin{lemma}[local Hodge comparison]\label{lem:localhodge}
Let $(Y,p)$ be a three-dimensional ordinary double point or the anticanonical
cone over $E=\dP_5$, $\dP_6$, $\dP_7$, or $\PP^1\times\PP^1$.  Let
$\pi_p\colon\widehat Y_p\to Y$ be the small resolution at a node and
$\widehat Y_p=\operatorname{Tot}(K_E)$ at a cone point, let $Z_p$ be its
exceptional locus, and put $U_p=\widehat Y_p\smallsetminus Z_p$.  Then the
connecting homomorphism gives an isomorphism
\begin{equation}\label{eq:localhodge}
 T^1_{Y,p}\;\xrightarrow{\ \sim\ }\;
 K_p^{\mathrm{cr}}:=
 \ker\!\left\{H^2_{Z_p}(\widehat Y_p;\Omega^2_{\widehat Y_p})
 \longrightarrow H^2(\widehat Y_p;\Omega^2_{\widehat Y_p})\right\}.
\end{equation}
There is also a natural isomorphism
\begin{equation}\label{eq:linklocalhodge}
 H_2(L_p;\CC)\cong H^3(L_p;\CC)
 \xrightarrow{\ \sim\ }K_p^{\mathrm{cr}},
\end{equation}
and the two isomorphisms are compatible with the local Gysin map to any global
crepant resolution containing $\widehat Y_p$.
\end{lemma}

\begin{proof}
The identifications
$T^1_{Y,p}\cong H^1(U_p;\Theta_{U_p})\cong H^1(U_p;\Omega^2_{U_p})$
are Schlessinger's comparison \cite{Schlessinger71} and contraction with the prescribed crepant volume
form.  The comparison applies to the non-lci cones: the hypothesis in
\cite[Lem.~1.17 and Lem.~4.1]{FL22a} is weakened in
\cite[Rem.~4.2]{FL22a} to an isolated singularity of depth at least $3$.

For a node,
$\widehat Y_p=\operatorname{Tot}_{\PP^1}(\cO(-1)\oplus\cO(-1))$.
The projection is affine, and the tangent sequence together with
$H^1(\PP^1,\cO(m))=0$ for $m\ge-1$ gives
$H^1(\widehat Y_p;\Theta_{\widehat Y_p})=0$.  The local cohomology sequence
therefore identifies the one-dimensional space $T^1_{Y,p}$ with
$K_p^{\mathrm{cr}}$.  The residue calculation for the node identifies this
class with the generator of $H^3(L_p)$; equivalently, this is the local
Gysin identification in Friedman's construction \cite{Friedman86}.

Now suppose $\widehat Y_p=\operatorname{Tot}(K_E)$, with projection
$\rho\colon\widehat Y_p\to E$ and zero section again denoted $E$.  Contraction
with the crepant volume form identifies
\[
 \Omega^2_{\widehat Y_p}(\log E)(-E)
 \cong \Theta_{\widehat Y_p}(-\log E),
\]
and the logarithmic tangent sequence is
\[
 0\longrightarrow\cO_{\widehat Y_p}\longrightarrow
 \Theta_{\widehat Y_p}(-\log E)\longrightarrow\rho^*\Theta_E
 \longrightarrow0.
\]
Since $\rho$ is affine and
$\rho_*\cO_{\widehat Y_p}=\bigoplus_{k\ge0}\cO_E(-kK_E)$, Kodaira vanishing
and Bott vanishing give
\begin{align*}
 H^1(\widehat Y_p;\cO_{\widehat Y_p})
   &=\bigoplus_{k\ge0}H^1(E;\cO_E(-kK_E))=0,\\
 H^1(\widehat Y_p;\rho^*\Theta_E)
   &=\bigoplus_{k\ge0}H^1
      \bigl(E;\Omega^1_E\otimes\cO_E(-(k+1)K_E)\bigr)=0.
\end{align*}
For the degree-five surface, the Bott vanishing used here is
Totaro's theorem \cite[Theorem~2.1]{totaro2020}; the twist
$-(k+1)K_E$ is ample for every $k\ge0$.  For the other surfaces
it is the usual toric Bott vanishing.
Thus
$H^1(\widehat Y_p;\Omega^2(\log E)(-E))=0$.  The exact sequence of
\cite[Thm.~2.1(v)]{FL22a}, applied to the good equivariant resolution
$\operatorname{Tot}(K_E)$, whose theorem explicitly allows every
three-dimensional isolated canonical Gorenstein germ, now gives
\eqref{eq:localhodge}.  Moreover \cite[Thm.~2.1(vi)]{FL22a} gives a natural
inclusion
\[
 H^3(L_p;\CC)=\operatorname{Gr}^2_FH^3(L_p;\CC)
 \lhook\joinrel\longrightarrow K_p^{\mathrm{cr}}
\]
and a decomposition of $K_p^{\mathrm{cr}}$ with this as one summand.
If $E$ is toric and its boundary polygon has $n$ vertices, Altmann's calculation gives
$\dim T^1_{Y,p}=n-3$.  The circle-bundle Gysin sequence gives
$\dim H^3(L_p)=9-\deg E=n-3$.  Hence the other summand is zero and the displayed
inclusion is an isomorphism.  At degree five the same dimension
comparison uses $\dim T^1=4$ \cite[\S5.5]{Gross97rev} and
$\dim H^3(L_p)=9-5=4$.  Its construction in local cohomology is
compatible with the Gysin map, which proves the final assertion.
\end{proof}

The degree hypotheses in Lemma~\ref{lem:localhodge} are essential to
this full-tangent comparison.  The cubic-cone calculation in
Section~\ref{sec01:degree-scope} has $\dim T^1=16$ but
$\dim H_2(L;\CC)=6$; local complete-intersection unobstructedness does
not supply the missing isomorphism.

\begin{lemma}[the local branches]\label{lem:branchhodge}
Under the isomorphism
$T^1_{Y,p}\cong H_2(L_p;\CC)$ of Lemma~\ref{lem:localhodge}, the tangent space
to each reduced local component is exactly the corresponding
subspace $\Rsm_p$: the whole line at a node or a
$\PP^1\times\PP^1$ cone, the entire $K_E^\perp$ at a degree-five
cone, the unique root line at a
$\dP_7$ cone, and respectively the $A_1$ line and the $A_2$ plane at a
$\dP_6$ cone.
\end{lemma}

\begin{proof}
At $E=\PP^1\times\PP^1$ both spaces are one-dimensional, and
$K_E^\perp=\CC(f_1-f_2)$ gives the assertion.  At degree five
the assertion follows from the full four-dimensional
comparison and the smooth irreducible base of
Proposition~\ref{prop03:degree5family}.  For the degree-six and
degree-seven points, consider a unit-edge polygon
$P=(v_0,\ldots,v_{n-1})$, put $e_i=v_{i+1}-v_i$.  Altmann identifies the
degree of the versal tangent space used here with
\begin{equation}\label{eq:altmanntangent}
 \mathcal M(P)=
 \left\{(t_i)\in\CC^n:\ \sum_i t_i e_i=0\right\}
 \big/\CC(1,\ldots,1),
\end{equation}
and a lattice Minkowski decomposition gives the class of its edge-dilation
vector $(t_i)$.  The Hodge isomorphism of Lemma~\ref{lem:localhodge} is natural
under lattice automorphisms of $P$, with one necessary character: contraction
with the crepant volume form twists equivariance by the determinant of the
induced automorphism of the three-dimensional cone lattice.
This equivariance is computed with the standard toric volume form.
The restriction of $\Omega$ is $g$ times that form, for an analytic
unit $g$.  Altmann's concentration of $T^1$ in the single Gorenstein
degree implies $\mathfrak m_pT^1_{X,p}=0$.  Multiplication by $g$
therefore acts on $T^1_{X,p}$ through the scalar $g(p)$.
The comparison determined by $\Omega$ differs only by this nonzero
scalar, so the invariant subspaces identified below are unchanged.

There is one unit-edge reflexive pentagon up to $\operatorname{GL}_2(\ZZ)$.  Take
\[
 P_7=\bigl((-1,-1),(0,-1),(1,0),(0,1),(-1,0)\bigr).
\]
Its nontrivial decomposition has dilation vector
$t=(0,1,0,1,0)$, up to replacing $t$ by $1-t$, and spans the reduced versal
base.  The reflection $s(x,y)=(y,x)$ fixes $[t]\in\mathcal M(P_7)$ and has
determinant $-1$.  It acts trivially on the tangent line of the reduced
branch, not merely on $[t]$ up to a character: it preserves the segment and the
triangle of the decomposition, hence both Cayley characters, and so fixes the
parameter of Altmann's family.  On
$H_2(L_p)=K_E^\perp$ it acts by $-1$ on the unique root line and by $+1$ on
the complementary line.  The determinant twist therefore forces the Hodge
image of $[t]$ to be the root line.  Independently, the smooth fibre of this
branch is the complement $M$ of a smooth member of $|-\tfrac12K|$ in
$\operatorname{Bl}_{\mathrm{pt}}\PP^3$, and the kernel of
$H_2(L_p;\QQ)\to H_2(M;\QQ)$ is the root line \cite[Prop.~2.1]{PaperIV}.

Likewise there is one unit-edge reflexive hexagon.  Take
\[
 P_6=\bigl((1,0),(1,1),(0,1),(-1,0),(-1,-1),(0,-1)\bigr).
\]
The determinant-one rotation $r(x,y)=(x-y,x)$ cyclically permutes its vertices.
On $\mathcal M(P_6)$ its characteristic polynomial is
$(\lambda+1)(\lambda^2+\lambda+1)$.  The $(-1)$-eigenspace is spanned by the
alternating dilation $(0,1,0,1,0,1)$, the triangle-plus-triangle decomposition
which gives Altmann's one-dimensional smoothing component; the complementary
plane is spanned by the coarsenings of the three-segment decomposition and is
the two-dimensional component.  On
$K_E^\perp\cong A_1\oplus A_2$, the same rotation is $-1$ on $A_1$ and a
Coxeter element, with characteristic polynomial $\lambda^2+\lambda+1$, on
$A_2$.  Equivariance now identifies the line with $A_1$ and the plane with
$A_2$.  The boundary classes and edge-dilation vectors above give finite
intersection and representation calculations; Appendix~\ref{app:cert}
records their exact verification.
\end{proof}

\subsection{The image of the global tangent space}
The use of the full local spaces in Lemma~\ref{lem:localhodge} is
essential: the comparison also accounts for tangent directions that
do not belong to a reduced local component.

\begin{theorem}\label{thm:D}
Under the Hodge identifications determined by $\Omega$, localization
gives exact sequences
\begin{align}
 0&\longrightarrow H^1(X,\Theta_X)\longrightarrow\mathbb T^1_X
                  \longrightarrow K\otimes_\QQ\CC\longrightarrow0,
                  \label{eq:fulltangent}\\
 0&\longrightarrow H^1(X,\Theta_X)\longrightarrow V_S
                  \longrightarrow K_S\otimes_\QQ\CC\longrightarrow0.
                  \label{eq:selectedtangent}
\end{align}
\end{theorem}

\begin{proof}
Let $E=\operatorname{Exc}(f_0)$ and $U=Y_0\setminus E=X_{\mathrm{reg}}$.
Excision and the local cohomology sequence give
\begin{equation}\label{eq:globalsupport}
 H^1(U;\Omega^2_U)\longrightarrow
 \bigoplus_pH^2_{E_p}(\widehat X_p;\Omega^2_{\widehat X_p})
 \longrightarrow H^2(Y_0;\Omega^2_{Y_0}).
\end{equation}
Here $\widehat X_p$ denotes the restriction of $Y_0$ over a sufficiently
small neighborhood of $p$.  The algebraic cone calculations in
Lemma~\ref{lem:localhodge} compute these analytic local groups as well:
the vanishing statements are equivalent to the corresponding
vanishing of completed higher direct images, by formal functions.
The analytic local-support comparison itself is that of
\cite[Thm.~2.1]{FL22a}.

Under Lemma~\ref{lem:localhodge}, the map from
$\bigoplus_pK_p^{\mathrm{cr}}$ to $H^2(Y_0;\Omega^2_{Y_0})$, followed by
the de Rham map, is the topological link map followed by Poincar\'e
duality.  In particular the square
\[
\begin{array}{ccc}
 \displaystyle\bigoplus_pH_2(L_p;\CC)&\longrightarrow&H_2(Y_0;\CC)\\
 \downarrow\cong&&\downarrow\operatorname{PD}\\
 \displaystyle\bigoplus_pK_p^{\mathrm{cr}}&\longrightarrow&H^4(Y_0;\CC)
\end{array}
\]
commutes.  The $\partial\bar\partial$-lemma makes
$H^2(Y_0;\Omega^2_{Y_0})\to H^4(Y_0;\CC)$ injective.  Hence an element
of the sum of local tangent spaces has zero image in
$H^2(Y_0;\Omega^2_{Y_0})$ if and only if its link class lies in
$K\otimes_\QQ\CC$.

Contraction with $\Omega$ and Schlessinger's depth-three comparison
identify
\[
 H^1(U;\Omega^2_U)\cong H^1(U;\Theta_U)\cong\mathbb T^1_X;
\]
see \cite[Rem.~4.2]{FL22a}.  Exactness of \eqref{eq:globalsupport}
therefore gives every element of $K\otimes\CC$ as the localization
of a global tangent.  Conversely, the localization of any global
tangent has zero supported Gysin image, hence lies in that kernel.
The kernel of localization is $H^1(X,\Theta_X)$ by the local-to-global
sequence.  This proves \eqref{eq:fulltangent}.  Taking the inverse
image of the chosen local tangent subspaces and applying
Lemma~\ref{lem:branchhodge} proves \eqref{eq:selectedtangent}.
\end{proof}

At an $A_2$ point, \eqref{eq:selectedtangent} can supply a nonzero
local vector on a discriminant line.  It asserts the existence of a
global first-order deformation with those local components; it makes
no smoothing assertion for such a vector.  The next two sections
determine which directions integrate and which profiles contain
smooth fibres.

\section{Local partial resolutions and simultaneous families}\label{sec03:partial}

We construct the local modifications used to compare deformation spaces.
Write $C_d$ for the anticanonical cone over $\dP_d$, and write $Q$ for the
anticanonical cone over $\PP^1\times\PP^1$.  The latter is different from
the ordinary double point: its polarization is $\cO(2,2)$, rather than
$\cO(1,1)$.  All morphisms below are proper over a sufficiently small
representative of the cone germ.  The constructions are algebraic, so
they also define modifications of analytic germs.

Fix a nowhere vanishing volume form on the regular locus of the global
threefold, and use its pullback on every crepant modification.  All
identifications of tangent classes with homology classes in this section
use this single form and Lemma~\ref{lem:localhodge}.  In particular,
coordinates introduced in polynomial models are initially algebraic
coordinates.  Their conversion into Hodge-normalized coordinates will
be stated separately.

\subsection{The degree-five discriminant and boundary periods}\label{sec03:degree5}

The degree-five cone has a smooth local deformation base, so it will
remain unchanged on the global partial resolution.  Its discriminant
requires five separate nonvanishing conditions.

\begin{proposition}\label{prop03:degree5family}
The full semiuniversal deformation base of $C_5$ is a smooth
four-dimensional germ.  It admits linear coordinates in which the
reduced discriminant is the union of five distinct hyperplanes.
Every noncentral singular fibre has only ordinary double points,
one for each discriminant hyperplane containing its parameter.
\end{proposition}

\begin{proof}
Let $V=\CC^5$, $A=\bigwedge^2V$, and
\[
 D=\{p\in A:p\wedge p=0\}=\Cone(\operatorname{Gr}(2,V)).
\]
Choose a general six-dimensional subspace $W\subset A$.
The surface $E=\operatorname{Gr}(2,V)\cap\PP(W)$ is smooth by
Bertini.  Adjunction gives $K_E=-H$, and its degree is five.
All smooth degree-five del Pezzo surfaces over $\CC$ are isomorphic,
so its anticanonical cone is the required germ.  Put
$U=W^\perp\subset A^*$ and $B=A/W=U^*$.  The quotient map
\begin{equation}\label{eq03:degree5family}
 L:D\longrightarrow B
\end{equation}
varies the four cutting linear forms by constants and has central
fibre $C_5$.

The Grassmannian cone is Cohen--Macaulay of dimension seven.
For example, its five quadratic Pfaffians have the graded resolution
\[
 0\longrightarrow S(-5)\longrightarrow S(-3)^5
 \longrightarrow S(-2)^5\longrightarrow S
 \longrightarrow\cO_D\longrightarrow0,
 \qquad S=\operatorname{Sym}(A^*).
\]
The four components of $L$ cut a three-dimensional central fibre
and hence form a regular sequence at the vertex.  Miracle flatness
over the smooth four-dimensional base proves flatness there.
Homogeneity then proves flatness everywhere: scaling transports
each point into this neighborhood and acts compatibly on $B$.

We check the full Kodaira--Spencer map.  A constant vector $c\in A$
gives the perturbation $p\wedge c$ of the Pfaffian equations on $W$,
up to the common factor two.  If this abstract first-order class is
zero, the cotangent presentation of the embedded germ shows that
its homogeneous degree-$-1$ part comes from a constant translation
$w\in W$.  The quadratic ideal has no linear elements, so
\[
 p\wedge(c-w)=0\quad\hbox{for every }p\in W.
\]
Write $a=c-w$.  If $a$ has rank four, the map
$a\wedge-:\bigwedge^2V\to\bigwedge^4V$ has rank five; its
kernel cannot contain the six-dimensional space $W$.
If $a$ has rank two with support $P$, its kernel is
$P\wedge V$, of dimension seven.  The decomposable bivectors in
$\PP(P\wedge V)$ form the four-dimensional Schubert variety of
two-planes meeting $P$: choose a line in $P$ and then a second
direction modulo that line.  A hyperplane $\PP(W)$ in this
$\PP^6$ would meet the Schubert variety in dimension at least
three, contradicting $\dim E=2$.  Thus $a=0$.
The Kodaira--Spencer map $B\to T^1_{C_5}$ is injective and is
an isomorphism because $\dim T^1_{C_5}=4$
\cite[\S5.5]{Gross97rev}.  This also shows that all of $T^1_{C_5}$
has degree $-1$.

Embed a semiuniversal base minimally with ring
$\CC\{z_1,\ldots,z_4\}/I$, $I\subset(z_1,\ldots,z_4)^2$.
The classifying map of~\eqref{eq03:degree5family} pulls the $z_i$
back to four functions with independent linear terms.  They are
analytic coordinates on $B$.  Hence the map from the ambient
convergent power-series ring is injective, forcing $I=0$.
The family is therefore semiuniversal for the full germ, including
its scheme structure; local prorepresentability is not required.

For a nonzero decomposable $p$ with support $P_p$, one has
$T_pD=P_p\wedge V$.  An alternating form $\alpha\in A^*$
annihilates this tangent space precisely when
$P_p\subset\ker\alpha$.  Such a nonzero form has rank two.
Thus the dual Grassmannian is the rank-two locus in $\PP(A^*)$.
It meets $\PP(U)$ transversely in five points
$[\alpha_1],\ldots,[\alpha_5]$.  Indeed, for a rank-two
$\alpha\in U$, failure of transversality is equivalent to
\[
 W\cap\bigwedge^2\ker\alpha\ne0.
\]
Every nonzero bivector in this intersection is decomposable and
would be a nonvertex singular point of $C_5$, contrary to the
smoothness of $E$.  The same argument excludes a positive-dimensional
intersection.  Its length is the degree five of the Grassmannian.

Set $K_i=\ker\alpha_i$ and $H_i=\ker(\alpha_i:B\to\CC)$.
The critical locus of $L$ away from the vertex is the union of
the punctured three-dimensional vector spaces $\bigwedge^2K_i$.
The maps
\[
 L:\bigwedge^2K_i\longrightarrow H_i
\]
are isomorphisms: their kernels are the intersections just excluded,
and both dimensions are three.  Thus a point $b\in H_i\setminus0$
has one associated critical point $p_i(b)$, linear in $b$.
Two such points cannot coincide for $b\ne0$.  Otherwise every
form in the pencil spanned by $\alpha_i,\alpha_j$ would vanish
on the same two-plane and hence have rank at most two, giving a
line in the finite dual intersection.  The complete reduced
discriminant is consequently $\bigcup_iH_i$.

Finally, near a nonzero $p_i(b)$ choose a Grassmannian chart in
which $\alpha_i=e_4^*\wedge e_5^*$ and the varying plane is
spanned by
\[
 e_1+x_3e_3+x_4e_4+x_5e_5,\qquad
 e_2+y_3e_3+y_4e_4+y_5e_5.
\]
With the nonzero cone coordinate $r$, the function $\alpha_i$
is $r(x_4y_5-x_5y_4)$.  Three other components of $L$ restrict
to coordinates on its critical locus, by the preceding isomorphism.
Eliminate them.  The remaining equation has nondegenerate quadratic
part in $x_4,x_5,y_4,y_5$, so the holomorphic Morse lemma gives
an ordinary double point with transverse smoothing parameter
$\alpha_i(b)$, up to a unit.
\end{proof}

\begin{lemma}\label{lem03:degree5marking}
Under the local Hodge comparison and the conic-pencil coordinates
\eqref{eq:degree5coordinates}, the discriminant of
Proposition~\ref{prop03:degree5family} is
\begin{equation}\label{eq03:degree5discriminant}
 \prod_{i=1}^5\lambda_i=0,
 \qquad \sum_{i=1}^5\lambda_i=0.
\end{equation}
\end{lemma}

\begin{proof}
Permuting the four blown-up points gives the $S_4$ action on the
first four conic pencils.  The quadratic Cremona involution at
the first three points lifts to an automorphism of $E$ and
interchanges the fourth and fifth pencils, fixing the first three.
These automorphisms generate $S_5$.  On
$\operatorname{Tot}(K_E)$ their canonical lifts preserve the
tautological two-form $\eta$ and the volume form $d\eta$.
The local Hodge comparison is therefore equivariant and realizes
$T^1_{C_5}$ as the standard sum-zero representation on the five
coordinates $\lambda_i$.

The tangent cone of the discriminant is intrinsic to the
deformation functor.  An automorphism of the central germ induces
a classifying map on a semiuniversal base; its invertible derivative
is the natural action on $T^1$.  It preserves the discriminant
because it preserves the isomorphism classes of the fibres.
Hence $S_5$ permutes the five hyperplanes in
Proposition~\ref{prop03:degree5family}.  This permutation action
is faithful.  A nontrivial kernel would contain $A_5$ and make
each normal line invariant under $A_5$.  As $A_5$ has no
nontrivial characters and no invariants in the standard
four-dimensional representation, this is impossible.

The image of this faithful action has order $120$, so it is the
full symmetric group on the five hyperplanes.  The stabilizer
of one hyperplane has order $24$.  In its action
on the five conic pencils it must fix a pencil: without a fixed
point its orbits would have sizes $2+3$, giving order at most
$2!3!=12$, while a size-five orbit cannot divide $24$.
It is thus the $S_4$ stabilizing one pencil.  The restriction
of the standard representation to this subgroup is the sum
of a trivial line and its three-dimensional standard representation.
The unique invariant normal line is generated by the corresponding
$\lambda_i$.  This identifies all five hyperplanes.

For a global volume form $g\,d\eta$, multiplication by the
analytic unit $g$ acts on $T^1_{C_5}$ by $g(0)$: its degree-$-1$
concentration gives $\mathfrak m_0T^1_{C_5}=0$.
Thus the fixed global normalization changes the comparison by
one nonzero scalar and preserves the five hyperplanes.
\end{proof}

\begin{lemma}\label{lem03:degree5milnor}
For a smooth Milnor fibre $M$ of $C_5$ with boundary link $L$,
\[
 H_2(M;\QQ)=0,\qquad \dim H_3(M;\QQ)=4.
\]
Both maps
\begin{equation}\label{eq03:degree5boundary}
 H_3(L;\QQ)\longrightarrow H_3(M;\QQ),\qquad
 H_3(M,L;\QQ)\longrightarrow H_2(L;\QQ)
\end{equation}
are isomorphisms.
\end{lemma}

\begin{proof}
For $b$ outside the discriminant the affine fibre compactifies as
\[
 V_b=\operatorname{Gr}(2,V)\cap\PP(L^{-1}(\CC b)),
 \qquad X_b=V_b\setminus E.
\]
This is a smooth three-hyperplane section of the Grassmannian.
It is smooth at $E$ because the four-hyperplane section $E$ is
smooth, and away from $E$ by Proposition~\ref{prop03:degree5family}.
Near the compact smooth divisor $E$, the squared affine norm is
$F/|s|^2$, where $s$ is a local equation for $E$ and $F$ is
smooth and positive.  Its normal derivative does not vanish
sufficiently close to $E$.  Thus every sufficiently large
Euclidean truncation of $X_b$ removes only a collar at infinity
and has its homotopy type.  Scaling by $\varepsilon$ identifies
$X_b\cap\{\|p\|\le R\}$ with
$X_{\varepsilon b}\cap\{\|p\|\le\varepsilon R\}$.
Fix the latter radius and take $R$ large.  This gives a small
Milnor fibre; transport in the smooth local family gives the
assertion for every nearby smooth parameter.

Lefschetz gives $b_1(V_b)=0$ and $b_2(V_b)=1$.
For completeness the middle Betti number follows from the Chern
calculation.  Write $H=c_1(\mathcal S^*)$ and
$P=c_2(\mathcal S^*)$ for the tautological bundle on the
Grassmannian.  Its tangent bundle is $\mathcal S^*\otimes\mathcal Q$,
so
\[
 c_1=5H,\qquad c_2=11H^2+P,\qquad c_3=15H^3.
\]
Dividing by $(1+H)^3$ gives $c_3(TV_b)=2H^3-3HP$.
Pieri's rule gives
$\int_{V_b}H^3=5$ and $\int_{V_b}HP=2$; hence
$\chi(V_b)=4$ and $b_3(V_b)=0$.
The Gysin sequence for $E\subset V_b$ gives
$H^2(X_b;\QQ)=0$ and
\[
 H^3(X_b;\QQ)\cong
 \ker\{H^2(E;\QQ)\longrightarrow H^4(V_b;\QQ)\},
\]
of dimension four.  The link circle-bundle sequence likewise
gives $b_2(L)=b_3(L)=4$.  Lefschetz duality and the pair sequence
make the second map in~\eqref{eq03:degree5boundary} a surjection
between four-dimensional spaces.  It is an isomorphism.  Exactness
then makes the first map surjective and hence an isomorphism as well.
\end{proof}

\begin{lemma}[degree-five boundary periods]\label{lem03:degree5periods}
There are rational classes $\beta_1,\ldots,\beta_5\in H_3(L;\QQ)$
whose pairings with the local Hodge classes are nonzero multiples
of the five functionals $\lambda_i$.
For any convergent smoothing arc of $C_5$, equipped with a relative
volume form generating its dualizing sheaf, the periods of the
images of all five classes in the smooth fibre are nonzero for
every sufficiently small noncentral parameter.
\end{lemma}

\begin{proof}
Use the semiuniversal family~\eqref{eq03:degree5family}.
The boundary family of a small representative is a smooth
proper submersion over the entire small base $B$, including $0$;
its homology local system is therefore trivial.
By~\eqref{eq03:degree5boundary}, every absolute Milnor cycle
has a unique boundary class.  Near a general point of the
discriminant hyperplane $H_i$, choose the class $\beta_i$
whose image is the nodal vanishing sphere, and transport it in
the boundary family over $B$.

The Pfaffian resolution above shows that the Grassmannian cone
has a homogeneous dualizing generator of weight five.  Dividing
by the four base differentials gives a nowhere vanishing
relative generator $\Omega_b^0$ of weight one.
Integration on transported compact boundary cycles gives
holomorphic functions on the whole base,
\[
 P_i(b)=\int_{\beta_{i,b}}\Omega_b^0.
\]
Holomorphic dependence uses only the smooth locus near the
boundary, so it holds also at parameters with a singular interior.
Scaling changes the boundary radius; compare the two boundaries
through their annular collars, whose central radial isotopy fixes
the homology marking.  The connected scaling action then preserves
transported classes and gives $P_i(cb)=cP_i(b)$.  Thus $P_i$ is linear.
Near a general point of $H_i$, it equals the nodal vanishing
period and has a simple zero along $H_i$.  Consequently
\begin{equation}\label{eq03:degree5linearperiod}
 P_i(b)=c_i\alpha_i(b),\qquad c_i\ne0.
\end{equation}

We verify that its derivative is the claimed Hodge functional.
Choose first-order product trivializations on the punctured
central germ, with Kodaira--Spencer cocycle $\theta_{jk}$.
Differentiating the relative three-form gives
\[
 \dot\Omega_k-\dot\Omega_j
       =d(\iota_{\theta_{jk}}\Omega_0^0).
\]
The \v{C}ech--de Rham representative for the derivative of a
period therefore has overlap term
$\iota_{\theta_{jk}}\Omega_0^0$, precisely the contraction
cocycle defining Lemma~\ref{lem:localhodge}.
In the logarithmic de Rham model of the circle bundle, this is
the natural comparison underlying \cite[Theorem~2.1(vi)]{FL22a}.
The link has $H^3(L;\CC)$ of pure Hodge type $(2,2)$, so its
$F^3$ is zero and this comparison determines the entire class.
Thus $dP_i|_0$ pairs $\beta_i$ with the local Hodge class,
up to the one common convention for Gysin maps and orientation.
There is no ambiguity from the chosen extension of the volume
form: a difference with the same overlap cocycle is a global
form $a\Omega_0^0$.  Normality extends $a$ across the vertex,
and contraction with the Euler vector field makes this form
exact, since all its homogeneous weights are strictly positive.
The same identity applies to a convergent series after shrinking.
Lemma~\ref{lem03:degree5marking} and
\eqref{eq03:degree5linearperiod} identify its kernel with
$\lambda_i=0$.

Finally let $b=b(t)$ be an arbitrary smoothing arc.  Its actual
relative volume form is $g(p,t)\Omega^0$, where $g$ is an
analytic unit on the pulled-back family.  Lift $g$ to a holomorphic
function of the ambient coordinates $(p,t)$; after shrinking it
is a unit and defines the same multiplier over $B\times\Delta$.
The boundary periods are holomorphic functions $P_i^g(b,t)$.
They vanish on $H_i\times\Delta$, by the nodal vanishing-cycle
calculation at its general points and analytic continuation.
Hence
\[
 P_i^g(b,t)=\alpha_i(b)u_i(b,t).
\]
To compute its linear term without expanding $g$ term by term,
fix a smooth direction $b_*$ and a compact cycle $\delta$
representing $\beta_i$ in $X_{b_*}$.  Scaling this cycle gives
\[
 P_i^g(sb_*,t)
 =s\int_\delta g(sp,t)\Omega_{b_*}^0
 =s\,g(0,t)c_i\alpha_i(b_*)+O(s^2).
\]
The estimate is uniform on the fixed compact cycle.  Smooth
directions form a dense open set, so
$d_bP_i^g(0,t)=g(0,t)c_i\alpha_i$.  Thus
$u_i(0,0)=g(0,0)c_i\ne0$.  Restricting to $b(t)$ proves
nonvanishing, whatever its contact order with $H_i$.
\end{proof}

\subsection{The quadric surface cone}

\begin{lemma}\label{lem03:quadric}
The full semiuniversal deformation base of $Q$ is a smooth
one-dimensional germ.  Its general fibre is smooth.  On the crepant
resolution $\operatorname{Tot}(K_{\PP^1\times\PP^1})$, the local Hodge
class of its tangent line is the difference of the two ruling classes.
\end{lemma}

\begin{proof}
The height-one polygon of $Q$, translated to its centre, has vertices
\[
 (1,0),\quad(0,1),\quad(-1,0),\quad(0,-1).
\]
Its oriented edges are $u,v,-u,-v$,
where $u,v$ are linearly independent over $\QQ$.  The polygon has
primitive edges, and the associated threefold singularity is isolated.
Altmann's equations for the versal base are
\[
 \sum_{i=1}^4 t_i^k d_i=0\quad(k\geq1),
 \qquad(d_1,d_2,d_3,d_4)=(u,v,-u,-v),
\]
modulo common translation of the $t_i$.  The equations with $k=1$
are $t_1=t_3$ and $t_2=t_4$.  All higher equations belong to the ideal
generated by these two linear equations.  Consequently the quotient
base is a smooth affine line, scheme-theoretically.  The full versality
theorem, including the obstruction calculation in
\cite[Sections~6--7]{Altmann97}, applies here; the assertion concerns
the entire base, including its possible nilpotent structure.

A smoothing can also be seen directly as
\begin{equation}\label{eq03:quadfamily}
 \{xy-zw=s\}/\boldsymbol\mu_2\longrightarrow\AAA^1_s,
\end{equation}
where the involution negates all four coordinates and fixes $s$.
Taking invariants is exact in characteristic zero, so this quotient
family is flat.  Its central coordinate ring is the second Veronese
subring of the ordinary quadric cone, hence is the section ring of
$\cO_{\PP^1\times\PP^1}(2,2)$.  For $s\ne0$ the quadric is smooth
and the involution acts freely.  Thus the family is a smoothing.

This family is semiuniversal, including its parameter normalization.
Translate the displayed polygon by $(1,0)$ to obtain
$[0,(1,1)]+[0,(1,-1)]$.  Its Cayley cone has rays
\[
 (0,0,1,0),\quad(1,1,1,0),\quad
 (0,0,0,1),\quad(1,-1,0,1).
\]
They span an index-two sublattice of $\ZZ^4$, and half their sum
is integral.  The corresponding toric fourfold is therefore
$\AAA^4/\boldsymbol\mu_2$, with the involution negating all four
coordinates.  The two projection characters pull back to the
products of the first two and the last two coordinates.
Altmann's component family \cite[\S8.1]{Altmann97} is their
difference, precisely~\eqref{eq03:quadfamily} after relabeling.
Its Kodaira--Spencer class is the nonzero difference of the two
segment-dilation parameters.  Since the full base is a smooth line,
every sufficiently small analytic deformation of $Q$ is locally a
pullback of this family.

Finally, the proof of Lemma~\ref{lem:localhodge} applies without change
to $E=\PP^1\times\PP^1$: the same toric vanishing proves the local
cohomology comparison.  The circle-bundle sequence identifies its link
homology with
\[
 K_E^\perp\subset H_2(E;\CC)
   =\CC f_1\oplus\CC f_2,
 \qquad K_E=-2f_1-2f_2.
\]
This perpendicular space is $\CC(f_1-f_2)$ and has dimension one,
which equals the tangent dimension just calculated.  The comparison
therefore identifies the tangent line with this ruling difference.
\end{proof}

\begin{lemma}\label{lem03:quadricmilnor}
Let $M$ be a smooth Milnor fibre of $Q$, with boundary $L$.
Then $M$ is homotopy equivalent to $\mathbb{RP}^3$ and
$\pi_1(L)=\ZZ/2$.  In particular,
\[
 H_2(M;\QQ)=0,\qquad H_3(M;\QQ)=\QQ,\qquad H^1(L;\ZZ)=0.
\]
The boundary map
$H_3(M,L;\QQ)\to H_2(L;\QQ)$ is an isomorphism; under the
circle-bundle identification its image is $\QQ(f_1-f_2)$.
\end{lemma}

\begin{proof}
After a linear change of coordinates the double cover of a
noncentral fibre of~\eqref{eq03:quadfamily} is
$\sum_{i=1}^4z_i^2=s$.  Its Milnor fibre retracts equivariantly
onto the vanishing sphere $\sqrt{s}\,S^3$, and the involution
acts antipodally.  Passing to the free quotient gives the asserted
retraction onto $\mathbb{RP}^3$ and its rational homology.
The link is the circle bundle with Euler class
$K_{\PP^1\times\PP^1}=-2f_1-2f_2$.  In its homotopy sequence
the image of $\pi_2(\PP^1\times\PP^1)\to\ZZ$ is $2\ZZ$,
so $\pi_1(L)=\ZZ/2$ and $H^1(L;\ZZ)=0$.
The rational circle-bundle sequence gives
$H_2(L;\QQ)=\QQ(f_1-f_2)$.  Lefschetz duality gives
$\dim H_3(M,L;\QQ)=\dim H^3(M;\QQ)=1$; the pair sequence
and $H_2(M;\QQ)=0$ make its boundary map surjective onto the
one-dimensional link group.  It is therefore an isomorphism.
\end{proof}

Smoothness in Lemma~\ref{lem03:quadric} does not assert vanishing of
cotangent cohomology in degree two.  In fact the local calculation gives
$\dim T^1_Q=1$ and $\dim T^2_Q=6$.\footnote{The latter value is reproduced in
\texttt{certificates/d19\_quadric\_partial\_resolution.py}.}
The degree-two value is not needed in the proof: the explicit equations
and full versality show directly
that the obstruction maps vanish.  In particular no complete-intersection
hypothesis on $Q$ is being used.

\subsection{Central fibres and their curve classes}

For the degree-six surface use the marking
\[
 E=\operatorname{Bl}_{p_1,p_2,p_3}\PP^2,
 \qquad h^2=1,\quad e_i^2=-1,\quad
 h\cdot e_i=e_i\cdot e_j=0\ (i\ne j),
\]
with the three points noncollinear.  Its toric boundary, in cyclic
order, is
\begin{equation}\label{eq03:boundary}
 e_1,\quad h-e_1-e_2,\quad e_2,\quad h-e_2-e_3,
 \quad e_3,\quad h-e_1-e_3.
\end{equation}
We retain the root conventions
\begin{equation}\label{eq03:roots}
 \ell=-h+e_1+e_2+e_3,\qquad
 A=-e_1+e_3,\qquad B=e_2-e_3.
\end{equation}
Thus the $A_1$ summand is $\QQ\ell$, and the $A_2$ summand is
$\QQ A\oplus\QQ B$.

\begin{proposition}\label{prop03:localmodels}
There are proper crepant partial resolutions with the following
properties.
\begin{enumerate}
\item A small morphism $W_7\to C_7$ has precisely one singular point,
of type $Q$.  A smooth crepant resolution through $W_7$ differs from
$\operatorname{Tot}(K_{\dP_7})$ by the flop of $h-e_1-e_2$.
Under this comparison the ruling difference of $Q$ is the degree-seven
root, up to the choice of its sign.
\item A morphism $W_{6,1}\to C_6$ is obtained from
$\operatorname{Tot}(K_E)$ by contracting
$C_0=e_1$ and $C_3=h-e_2-e_3$.  It has exactly two ordinary double
points.  The image of $E$ is $\PP^1\times\PP^1$.
\item A morphism $W_{6,2}\to C_6$ is obtained by contracting
$e_1,e_2,e_3$ in $\operatorname{Tot}(K_E)$.  It has exactly three
ordinary double points.  The image of $E$ is $\PP^2$.
\end{enumerate}
The last two morphisms to $C_6$ are divisorial.  With Hodge-normalized
nodal coefficients, the maps into the selected root subspaces are
\begin{equation}\label{eq03:nodaltransfer}
 \lambda\ell=\lambda[C_0]-\lambda[C_3],\qquad
 aA+bB=-a e_1+b e_2+(a-b)e_3.
\end{equation}
In particular the two coefficients over an $A_1$ point are
$(\lambda,-\lambda)$, and the three coefficients over an $A_2$ point
are $(\mu_1,\mu_2,\mu_3)=(-a,b,a-b)$.
\end{proposition}

\begin{proof}
For degree seven use the height-one pentagon
\[
 (0,0),(1,0),(2,1),(1,2),(0,1).
\]
Subdivide it into the triangle with vertices $(0,0),(1,0),(0,1)$
and the quadrilateral with vertices $(1,0),(2,1),(1,2),(0,1)$.
The triangle is unimodular.  Translation of the quadrilateral by
$(-1,-1)$ gives the polygon in Lemma~\ref{lem03:quadric}.
The subdivision preserves the support and introduces no rays;
the induced toric morphism is therefore proper and small.
Every ray has height one, so the morphism is crepant.  The two maximal
cones give one smooth chart and one $Q$ chart, proving the assertion
about its singularities.

Inserting the interior lattice point of the quadrilateral gives a
smooth resolution through this model.  Its triangulation and the star
triangulation of the pentagon differ by one diagonal change.  The
corresponding curve on the original degree-seven surface is
$c=h-e_1-e_2$.  It is a $(-1)$-curve and $K_E\cdot c=-1$; its normal
bundle in $\operatorname{Tot}(K_E)$ is
$\cO(-1)\oplus\cO(-1)$.  The diagonal change is its ordinary flop.
The surface after the flop is the contraction of $c$, namely
$\PP^1\times\PP^1$.  The classes $h-e_1$ and $h-e_2$ have
intersection zero with $c$, square zero, and mutual intersection one.
They descend to its two rulings.  Their difference is $e_2-e_1$,
the degree-seven root.  This also identifies their classes on the
common complement of the flopping curves.  Indeed representatives of
these two pencils can be chosen disjoint from $c$; their differences
give the same class under the homology comparison across the flop.

For degree six use the polygon
\[
 v_0=(1,0),\ v_1=(1,1),\ v_2=(0,1),\ v_3=(-1,0),
 \ v_4=(-1,-1),\ v_5=(0,-1)
\]
and its centre $o=(0,0)$.  The six triangles
$[o,v_i,v_{i+1}]$ are unimodular; the star fan gives
$\operatorname{Tot}(K_E)$ with boundary marking
\eqref{eq03:boundary}.  To contract the boundary curve at $v_i$,
merge its two adjacent triangles into
$[o,v_{i-1},v_i,v_{i+1}]$.  For $i=0,3$, these mergers are disjoint
in their interiors and leave two unimodular triangles.  For $i=0,2,4$
they give three quadrilaterals covering the polygon.
Each merged cell is a unit parallelogram: its four lifted rays have
the relation with coefficients $(1,-1,1,-1)$ in cyclic order, and
three adjacent lifted rays have determinant $\pm1$.  Its affine
toric chart is consequently $xy-zw=0$, with no finite quotient.
No boundary cone changes, and the subdivisions cover the original
cone, proving properness, crepancy, and the claimed singularity counts.

The same contractions can be checked intrinsically.  Every curve
in \eqref{eq03:boundary} has self-intersection and canonical degree
$-1$.  Its normal sequence in the total canonical bundle splits,
since its extension class belongs to $H^1(\PP^1,\cO)=0$.
The selected curves are disjoint, so their contractions are independent
ordinary nodal contractions.  Contracting $e_1,e_2,e_3$ on $E$ gives
$\PP^2$.  Contracting $e_1,h-e_2-e_3$ gives the smooth toric
degree-eight surface with four boundary curves of square zero,
namely $\PP^1\times\PP^1$.  These surfaces are still contracted
by the subsequent morphism to $C_6$, explaining its divisorial nature.

Finally, \eqref{eq03:nodaltransfer} follows by substitution in
\eqref{eq03:roots}.  These are also the deformation-theoretic transfer
maps: inclusion of the nodal exceptional supports in the full
exceptional surface sends a tuple of local classes to the sum of
their curve classes.  Naturality of the connecting maps in local
cohomology and contraction with the same volume form make this map
commute with the Hodge comparison.  In the degree-seven case the
comparison is made on the common regular open set across the flop.
There is thus no independent scalar choice at each newly introduced
node in \eqref{eq03:nodaltransfer}.
\end{proof}

\subsection{A simultaneous modification of the one-dimensional branch}

The two-node model extends over the entire $A_1$ branch.  Here it is
useful to give the extension explicitly, since its central-fibre
description alone would not justify lifting a family.
Altmann's triangle-plus-triangle family is the affine cone over the
Segre embedding of $(\PP^1)^3$, equipped with the function
\begin{equation}\label{eq03:a1function}
 u_0v_0w_0-u_1v_1w_1=s;
\end{equation}
see \cite[Section~8.2(i)]{Altmann97}.  The products in this equation
are tensor coordinates on the cone, rather than functions on the
product of projective spaces.

Retain the first two projective factors.  The resulting incidence
space is
\[
 V=\operatorname{Tot}_{\PP^1\times\PP^1}
       \bigl(\cO(-1,-1)\oplus\cO(-1,-1)\bigr).
\]
Its morphism to the tensor cone sends a pair of fibre vectors to the
tensor $u\otimes v\otimes w$.  It is proper, since the incidence
description is closed in the product of the cone with
$\PP^1\times\PP^1$.  It is an isomorphism away from the vertex:
a nonzero decomposable tensor determines its first two factor lines
uniquely.  On $V$, equation \eqref{eq03:a1function} defines a morphism
to the $s$-line and hence a simultaneous modification of that family.

There are exactly two critical points in the central fibre.  To see
this, write the fibre coordinates locally as $A,B$.  Differentiating
with respect to them forces $u_0v_0=u_1v_1=0$.  The two possible
base points are $([1:0],[0:1])$ and $([0:1],[1:0])$.  On their
respective affine charts the equations are exactly
\begin{equation}\label{eq03:a1charts}
 \beta A-\alpha B=s,
 \qquad \gamma C-\delta D=s.
\end{equation}
Their differentials vanish only at the origin of each four-dimensional
chart, and their quadratic forms are nondegenerate.  On the other
two projective charts one coefficient of a fibre coordinate is a unit,
so the map is smooth there.  These equations also prove flatness:
near a critical point the morphism is the standard nodal deformation,
and elsewhere it is smooth.  Equivalently, the local rings are
torsion-free over the regular one-dimensional base.

The zero section belongs only to $s=0$.  Thus the modification is an
isomorphism over every $s\ne0$ and away from the vertex of the central
fibre.  Its central fibre is $W_{6,1}$: resolving the two nodes on
the side which blows up the corresponding points of the zero section
gives the marked degree-six surface and the two curves $C_0,C_3$.
This can equally be read from the first degree-six subdivision in
Proposition~\ref{prop03:localmodels}.

The two displayed equations have the same algebraic parameter $s$.
Their Hodge-normalized parameters have opposite signs when their
exceptional curves are oriented by $[C_0],[C_3]$.  More precisely,
the image is $(\lambda,-\lambda)$ with $\lambda=c_1s$ to first
order for a nonzero constant $c_1$.  This follows either from the
local-cohomology comparison above or by pairing with $E$:
$E\cdot C_0=E\cdot C_3=-1$, while the branch class lies in
$K_E^\perp$.  Neither nodal derivative vanishes, by
\eqref{eq03:a1charts}.  The normalized branch coordinate is chosen
using this $c_1$, rather than by equating two independently chosen
residue normalizations.

\subsection{The two-dimensional branch and its discriminant}

\begin{proposition}\label{prop03:a2family}
The $A_2$ branch of $C_6$ admits a proper simultaneous partial
resolution whose central fibre is $W_{6,2}$.  The morphism is an
isomorphism over every nonzero point of the two-dimensional base,
including its discriminant.  In linear coordinates $(a,b)$ compatible
with the Hodge normalization, the three nodal tangent parameters are
\[
 (\mu_1,\mu_2,\mu_3)=(-a,b,a-b).
\]
The discriminant is the union of the three lines
$ab(a-b)=0$.  At a nonzero point of any one of these lines the fibre
has precisely one ordinary double point; off these lines it is smooth.
\end{proposition}

\begin{proof}
We first work in algebraic parameters before imposing the Hodge
normalization.  Write $D=\CC^3/\CC(1,1,1)$ and represent its points
by $(d_1,d_2,d_3)$.  The branch family is
\begin{equation}\label{eq03:matrixfamily}
 \mathcal M_d=\{M:\ \rk M\leq1,\quad M_{ii}=z+d_i\ (1\leq i\leq3)\}.
\end{equation}
Common translation of the $d_i$ is absorbed in $z$.  This is the
actual restricted versal family, not just a first-order model.
For an explicit comparison with \cite[Section~5.6]{Altmann97}, take
edge parameters $(a,b,c,a,b,c)$ and the matrix
\[
 \begin{pmatrix}
 z+c&z_1&z_2\\ z_4&z+b&z_3\\ z_5&z_6&z+a
 \end{pmatrix}.
\]
Its nine $2\times2$ minors are the six equations
$z_it_i-z_{i-1}z_{i+1}=0$ (cyclic indices), with
$(t_1,\ldots,t_6)=(z+a,z+b,z+c,z+a,z+b,z+c)$,
and the three equations
$t_2t_3-z_1z_4=t_3t_4-z_2z_5=t_1t_2-z_3z_6=0$.
They generate exactly the branch ideal.  This also follows from
\cite[Section~8.2(ii)]{Altmann97}, which describes its total space
as the cone over $\PP^2\times\PP^2$ with the two diagonal-difference
functions.

Resolve that total space by retaining its column line:
\begin{equation}\label{eq03:incidence}
 \widetilde{\mathcal M}
 =\{(M,L):\operatorname{im}M\subseteq L\}
 =\operatorname{Tot}_{\PP^2}(\cO(-1)^{\oplus3})
 \longrightarrow\{M:\rk M\leq1\}.
\end{equation}
Write $M=xy^{\mathsf T}$, where $x\ne0$ and
$(x,y)\sim(tx,t^{-1}y)$.  The incidence relation proves properness.
A nonzero rank-one matrix has a unique column line, and on an open
set $M_{k\ell}\ne0$ this line is represented by its $\ell$th column.
These formulas give regular inverses, proving that
\eqref{eq03:incidence} is an isomorphism away from the zero matrix.

Compose \eqref{eq03:incidence} with the diagonal-difference map to
$D$.  On $x_k=1$, let $i,j$ be the two other indices.  Elimination of
$y_k,z$ gives the single equation
\begin{equation}\label{eq03:a2chart}
 x_i y_i-x_j y_j=d_i-d_j,
 \qquad y_k=x_i y_i-d_i+d_k.
\end{equation}
The two independent base coordinates can be taken as $d_i-d_j$
and $d_k-d_i$.  Thus each chart is the product of the standard nodal
deformation with a smooth one-dimensional base.  This proves
flatness over the full plane, including its origin and discriminant.
In the central fibre its only singular point in this chart has all
four coordinates zero.  The three resulting points lie over the
three distinct coordinate points of $\PP^2$ and give exactly three
nodes.  The zero section is the exceptional $\PP^2$.  Resolving the
nodes by blowing up their points on this surface recovers
$\operatorname{Bl}_3\PP^2$ and the curves $e_1,e_2,e_3$.
The star fan is the second subdivision of
Proposition~\ref{prop03:localmodels}; in particular the central
morphism is crepant and is $W_{6,2}\to C_6$.

The zero matrix has all diagonal differences zero.  It therefore
occurs only over $0\in D$.  The isomorphism away from the zero
matrix consequently restricts to an isomorphism on every fibre with
$d\ne0$, with no exception along a discriminant line.  The chart
equations show that the three singularity parameters are the cyclic
differences
\[
 d_2-d_3,\quad d_3-d_1,\quad d_1-d_2.
\]
Their sum is zero.  A nonzero base point lies on at most one of their
zero lines, and the corresponding chart has one node by its
nondegenerate quadratic equation.  All remaining charts are smooth.
This proves both the discriminant statement and the assertion about
its noncentral fibres.

To identify parameters with the fixed marking, label the three
coordinate points, and choose signs in the nodal equations, so that
their differences in the order $e_1,e_2,e_3$ are
$d_3-d_1,d_2-d_3,d_1-d_2$.  Taking
$(d_1,d_2,d_3)=(a,b,0)$ gives $(-a,b,a-b)$ algebraically.
The determinant-twisted equivariance in
Lemma~\ref{lem:branchhodge} identifies the two-dimensional parameter
representation with the $A_2$ root representation, uniquely up to a
nonzero scalar: rotation and reflection act as the usual irreducible
two-dimensional representation of the triangle symmetry group.
Thus, after matching this marking, the Hodge comparison multiplies
these three algebraic differences by one common nonzero scalar.
Rescale $a,b$ simultaneously by that scalar.  Naturality under
inclusion of supports now gives exactly
\[
 aA+bB=-a e_1+b e_2+(a-b)e_3
\]
with the prescribed global volume form.  The linear rescaling leaves
the three discriminant lines unchanged.  This proves the asserted
normalized parameter map.
\end{proof}

The constants relating algebraic parameters to Hodge-normalized
parameters may differ between singular points, and between the
$A_1$ and $A_2$ summands at one degree-six point.  They are fixed by
the one global volume form, not by unrelated choices of local
coordinates.  Equations~\eqref{eq03:nodaltransfer} are identities
after this normalization.  Their vanishing conditions are independent
of the remaining common nonzero factors.

The simultaneous constructions also distinguish a nonzero tangent
on a smoothing component from a tangent transverse to its
discriminant.  Every nonzero $A_1$ parameter smooths the germ, whereas
a nonzero vector in the $A_2$ plane may leave one of the three nodes
unsmoothed.  Because the simultaneous modification is an isomorphism
on every noncentral fibre, this distinction can be applied to an
arbitrary analytic classifying arc, including an arc tangent to a
discriminant line to high order.  It is this property that permits
the necessity argument in Section~\ref{sec04:smoothing}.

\section{The smoothing criterion}\label{sec04:smoothing}

Throughout this section, $X$ is a connected normal projective complex
threefold with $\omega_X\cong\cO_X$ and $H^1(X,\cO_X)=0$, whose
singularities are isolated ordinary double points and the specified analytic
anticanonical cones over $\dP_5$, $\dP_6$, $\dP_7$, and
$\PP^1\times\PP^1$.  Fix a local component $S_p$
at every cone point, and write $S=(S_p)$.  We retain the entire nodal base
and the entire degree-five and $\PP^1\times\PP^1$ cone bases, and use the reduced
component at a degree-seven point.  The smooth crepant
resolution $Y_0\to X$, the link subspaces $R_{p,S_p}$, and their relation
space
\begin{equation}\label{eq04:relations}
 K_S=\ker\left(\bigoplus_p R_{p,S_p}
                   \longrightarrow H_2(Y_0;\QQ)\right)
\end{equation}
are as in Section~\ref{sec:deformation}.  All tangent identifications below
use contraction with one fixed generator of $H^0(X,\omega_X)$.  In
particular, coefficients in~\eqref{eq04:relations} are the resulting
homological coordinates; they are not independently normalized local
equation parameters.

Choose an analytic semiuniversal deformation of $X$, analytic local
semiuniversal deformations of its singularities, and analytic classifying
maps from the first base to the latter bases.  We denote by $\Def^S(X)$
the scheme-theoretic inverse image of the product of the selected local
components.  These choices specify the analytic subgerm and its completed
local ring.  We use $\mathbf{Def}_X$ for the functor of marked deformations
over local Artin $\CC$-algebras, and $\Def(X)$ for an analytic
semiuniversal base.

\subsection{Global partial resolutions}

\begin{proposition}\label{prop04:globalmodel}
The local morphisms of Proposition~\ref{prop03:localmodels}, together
with the identity near degree-five and original
$\PP^1\times\PP^1$ cones, determine a
proper crepant analytic modification $\pi:Z\to X$.  Its singularities are
the original nodes, two or three nodes over each degree-six point according
to its selected component, and one anticanonical
$\PP^1\times\PP^1$ cone over each degree-seven point, together with
the original degree-five and $\PP^1\times\PP^1$ cones.  Moreover,
\begin{equation}\label{eq04:cy}
 R\pi_*\cO_Z\simeq\cO_X,\qquad
 \omega_Z\cong\cO_Z,\qquad H^1(Z,\cO_Z)=0.
\end{equation}
Every smooth resolution of $Z$ is Moishezon and satisfies the
$\partial\bar\partial$-lemma.
\end{proposition}

\begin{proof}
Choose mutually disjoint sufficiently small neighborhoods $U_p$ of the
degree-six and degree-seven cone points, and restrict the local morphisms to proper morphisms
$W_p\to U_p$.  Glue the $W_p$ to the unchanged complement of these
degree-six and degree-seven points using their specified isomorphisms
over $U_p\setminus\{p\}$.  This complement includes every original
degree-five and $\PP^1\times\PP^1$ cone.
The images of these charts are open.  Two points with different images in
$X$ are separated by inverse images of disjoint neighborhoods in $X$;
two distinct points over the same cone point are separated in the
Hausdorff space $W_p$.  Thus the glued space is Hausdorff.  The map is
proper over every $U_p$ and over the complement of these points,
and properness is local on the target.  This proves properness globally,
and hence compactness of $Z$.

The local models are Gorenstein and canonical.  In particular, $X$ and
$Z$ have rational singularities.  If $g:\widetilde Z\to Z$ is a
resolution, then $\pi g$ is a resolution of $X$ and
$Rg_*\cO_{\widetilde Z}\simeq\cO_Z$,
$R(\pi g)_*\cO_{\widetilde Z}\simeq\cO_X$.  Composition gives the
first assertion of~\eqref{eq04:cy}, and Leray gives its last assertion.
Local crepancy means that the pullback of the chosen volume form extends
to a generator of the dualizing sheaf on every local model.  These
generators agree on the overlaps, giving $\omega_Z\cong\cO_Z$.

A smooth resolution is compact and bimeromorphic to projective $X$, hence
Moishezon.  A smooth projective modification of it exists, and the
$\partial\bar\partial$-lemma descends under a proper bimeromorphic map
between compact complex manifolds
\cite[Theorem~5.22 and Corollary~5.23]{dgms1975}.
\end{proof}

Let $Y\to Z$ be the smooth crepant resolution obtained by resolving the
original nodes as in $Y_0$, the
new nodes on the sides specified in Proposition~\ref{prop03:localmodels},
and the degree-five and quadric cones by their vertex blowups.
At a degree-six
point this recovers the original neighborhood in $\operatorname{Tot}K_E$.
At a degree-seven point it differs from $Y_0$ by the described flop.
Near a degree-five or original $\PP^1\times\PP^1$ cone the two
resolutions agree.
Removing a smooth curve from a smooth threefold does not change $H_2$:
the relative groups in degrees two and three vanish by the real
codimension-four Thom isomorphism.  The common complements therefore
identify $H_2(Y_0;\QQ)$ and $H_2(Y;\QQ)$.  The degree-seven root is
represented on that complement by the difference of the two ruling
classes, so the identifications preserve the local substitutions used
in~\eqref{eq04:relations}.

\subsection{Blowing down Artin deformations analytically}

The following lemma supplies a morphism of deformation functors even
when $Z$ is not algebraic.  Its sheaf-theoretic part is the construction
in \cite[Proposition~2.12 and Corollary~2.13]{sano2017}; we include the
analytic realization.

\begin{lemma}[analytic blow-down]\label{lem04:blowdown}
Let $f:Z\to X$ be a proper holomorphic morphism of complex spaces with
$f_*\cO_Z=\cO_X$ and $R^1f_*\cO_Z=0$.  A flat deformation $Z_A$ over
a local Artin $\CC$-algebra $A$ determines a flat analytic deformation
\[
 X_A=(|X|,f_*\cO_{Z_A})
\]
and a proper morphism $f_A:Z_A\to X_A$ deforming $f$.  This construction
is natural in marked deformations and commutes with every Artin base
change.
\end{lemma}

\begin{proof}
The marking identifies the underlying topological space of $Z_A$ with
$|Z|$, so the displayed pushforward is defined as a sheaf of $A$-algebras.
Let $U\subset X$ be a sufficiently small Stein open subset and put
$V=f^{-1}(U)$, $\mathcal F=\cO_{Z_A}|_V$.  Leray and Cartan's theorem~B
give $H^1(V,\cO_Z)=0$.  For every finite $A$-module $M$ we claim
\begin{equation}\label{eq04:modulevanish}
 H^1(V,\mathcal F\otimes_A M)=0.
\end{equation}
Indeed, a composition series of $M$ has successive quotients $\CC$.
Flatness of $\mathcal F$ preserves its short exact sequences under
tensor product.  The long cohomology sequences, starting from
$\mathcal F\otimes_A\CC=\cO_Z|_V$, prove
\eqref{eq04:modulevanish} by induction.

It follows that $M\mapsto H^0(V,\mathcal F\otimes_A M)$ is exact on
finite $A$-modules.  Applying it to a finite presentation of $M$ by free
modules gives the natural isomorphism
\begin{equation}\label{eq04:modulebc}
 H^0(V,\mathcal F)\otimes_A M
       \xrightarrow{\ \sim\ }H^0(V,\mathcal F\otimes_A M).
\end{equation}
Consequently $H^0(V,\mathcal F)$ is $A$-flat: exactness on finite
modules suffices over the Noetherian ring $A$.  Passing to stalks on
$X$ proves flatness of $f_*\cO_{Z_A}$ and the corresponding sheaf
version of~\eqref{eq04:modulebc}.  In particular its reduction modulo
the maximal ideal $\mathfrak m_A$ is $\cO_X$.

To prove analyticity, take $U$ small enough to be a closed analytic
subspace of a polydisc $D$.  Equation~\eqref{eq04:modulebc}, with
$M=\CC$, lifts its embedding coordinates to global sections of
$\mathcal F$.  Together with the structure morphism to
$\operatorname{Spec}_{\mathrm{an}}A$ these sections define a holomorphic
map
\[
 g_A:V_A\longrightarrow D\times\operatorname{Spec}_{\mathrm{an}}A.
\]
Its underlying continuous map is the proper composite
$V\to U\hookrightarrow D$, hence $g_A$ is proper.  Grauert's direct
image theorem for complex spaces, including nonreduced spaces,
implies that $g_{A*}\cO_{V_A}$ is coherent on the polydisc thickening
\cite[\S6, Hauptsatz~I]{grauert1960}.  The natural algebra map
\[
 \cO_{D\times\operatorname{Spec}_{\mathrm{an}}A}
                   \longrightarrow g_{A*}\cO_{V_A}
\]
is surjective modulo $\mathfrak m_A$, since its reduction is
$\cO_D\to\cO_U$.  Its coherent cokernel $\mathcal C$ therefore
satisfies $\mathcal C=\mathfrak m_A\mathcal C$, and nilpotence gives
$\mathcal C=0$.  Its kernel is a coherent ideal.  The resulting closed
analytic subspace has underlying ringed space
$(U,f_*\cO_{Z_A}|_U)$.

This ringed space is independent of the lifted coordinates.  Thus the
local analytic realizations glue, and the adjunction counit induces
$f_A$.  Locally this is the holomorphic map to the closed subspace just
constructed; its central restriction is $f$ and its underlying map is
proper.  Finally, if $A\to A'$ is any homomorphism of local Artin
$\CC$-algebras, then $A'$ is a finite $A$-module.  The sheaf version
of~\eqref{eq04:modulebc} identifies the blow-down of
$Z_A\times_A A'$ with $X_A\times_A A'$.  These natural tensor
identifications respect composition.  An isomorphism of marked source
deformations pushes forward to an isomorphism of the constructed target
deformations, proving the functorial assertion.
\end{proof}

\subsection{Unobstructedness of the partial resolution}

\begin{lemma}\label{lem04:unobstructed}
The functor $\mathbf{Def}_Z$ is unobstructed.  Both
$\mathbf{Def}_Z$ and $\mathbf{Def}_X$ are prorepresentable, and
Lemma~\ref{lem04:blowdown} induces a morphism of their formal bases
\begin{equation}\label{eq04:beta}
 \beta:\widehat{\Def}(Z)\longrightarrow\widehat{\Def}(X).
\end{equation}
\end{lemma}

\begin{proof}
By Proposition~\ref{prop04:globalmodel}, $Z$ is compact, normal and
Gorenstein with trivial dualizing sheaf, $H^1(Z,\cO_Z)=0$, and isolated
rational, hence Du Bois, singularities.  A resolution satisfies the
$\partial\bar\partial$-lemma.  Friedman's theorem
\cite[Theorem~3.10]{friedman2026} gives the generalized $T^1$ lifting
property.  We explain why here it gives ordinary $T^1$ lifting.

Set $A_k=\CC[t]/(t^{k+1})$ and
$B_k=A_k[\epsilon]/(\epsilon^2)$.  For a deformation $Z_k/A_k$ let
$Z_{k-1}$ denote its restriction.  Relative first-order deformations,
with an identification over $\epsilon=0$, form $T^1(Z_k/A_k)$.
There is an exact sequence and its local counterpart
\[
 \begin{aligned}
 T^1(Z_k/A_k)&\longrightarrow T^1(Z_{k-1}/A_{k-1})
       \xrightarrow{\delta}\mathbb T^2_Z,\\
 T^1_{\mathrm{loc}}(Z_k/A_k)&\longrightarrow
 T^1_{\mathrm{loc}}(Z_{k-1}/A_{k-1})
       \xrightarrow{\delta_{\mathrm{loc}}}
       H^0(Z,\mathcal T^2_Z).
 \end{aligned}
\]
Here $\mathbb T^2_Z=\operatorname{Ext}^2(\Omega^1_Z,\cO_Z)$ and
$\mathcal T^2_Z=\mathcal{E}xt^2(\Omega^1_Z,\cO_Z)$ are the global
and local obstruction spaces used in the lifting theorem.
Localization $\ell:\mathbb T^2_Z\to H^0(Z,\mathcal T^2_Z)$
commutes with the connecting maps.  Generalized lifting means that
$\ell$ is injective on $\operatorname{im}\delta$ for every $k$ and
every $Z_k$.

The full local deformation bases of $Z$ are smooth: this is the usual
nodal calculation, Proposition~\ref{prop03:degree5family}, and
Lemma~\ref{lem03:quadric}.  This implies
surjectivity of the displayed local $T^1$ map with its markings, as
follows.  For a fixed local deformation over $A_k$ and a marked relative
deformation over $B_{k-1}$, use the specified identification over
$A_{k-1}$ to glue them over
\[
 D_k=A_k\times_{A_{k-1}}B_{k-1}.
\]
The deformation category of analytic spaces has the
Rim--Schlessinger gluing property: the local structure sheaves are
formed by the fibre product over the common restriction, and the
result comes with compatible identifications on both factors.  This
is gluing with a specified identification, not merely a choice of
isomorphism classes; see also the categorical formulation in
\cite[\S1]{Gross97rev}.  The natural map
$B_k\to D_k$ is surjective with kernel $(\epsilon t^k)$.
Smoothness of the local deformation functor lifts the glued object to
$B_k$.  Choose an isomorphism from its restriction to $D_k$ to that
object, and restrict this same isomorphism to the two factors.
The resulting identifications agree over $A_{k-1}$ and give exactly
the required relative first-order lift.  Infinitesimal automorphisms
can change these choices but do not prevent their existence; local
prorepresentability is not used.

Thus $\delta_{\mathrm{loc}}=0$, so $\ell\delta=0$ and injectivity
on $\operatorname{im}\delta$ gives $\delta=0$.  Ordinary $T^1$
lifting proves unobstructedness
\cite[Theorem~3.3]{friedman2026}.  In particular this argument concerns
the full base, including its nilpotent structure; it does not infer
smoothness from a statement about reduced supports.

For completeness, the vanishing of infinitesimal automorphisms follows
for both $T=X$ and $T=Z$ from the same argument.  Choose an equivariant
resolution $h:\widetilde T\to T$.  Canonical singularities allow the
volume form to pull back holomorphically, and contraction gives an
injection $\Theta_{\widetilde T}\to\Omega^2_{\widetilde T}$.
Rationality, Hodge symmetry on the Moishezon resolution, and Serre
duality give
\[
 h^0(\widetilde T,\Omega^2_{\widetilde T})
 =h^2(\widetilde T,\cO_{\widetilde T})
 =h^2(T,\cO_T)=h^1(T,\cO_T)=0.
\]
Since $h_*\Theta_{\widetilde T}=\Theta_T$, this proves
$H^0(T,\Theta_T)=0$, and the deformation functors are
prorepresentable; compare \cite[Lemma~3.2]{friedman2026}.
The natural Artin transformation supplied by
Lemma~\ref{lem04:blowdown} therefore gives~\eqref{eq04:beta}.
\end{proof}

\subsection{Lifting the entire selected tangent space}

\begin{lemma}\label{lem04:globaltangent}
Put $V_S=T_0\Def^S(X)$.  There is an exact sequence
\begin{equation}\label{eq04:tangent}
 0\longrightarrow H^1(X,\Theta_X)\longrightarrow V_S
       \longrightarrow K_S\otimes_{\QQ}\CC\longrightarrow0,
\end{equation}
and every vector in $V_S$ belongs to the image of $d\beta$.
\end{lemma}

\begin{proof}
The exact sequence is the restriction of Theorem~\ref{thm:D} to the
selected local tangent subspaces.  We construct lifts through $d\beta$,
including lifts of its locally trivial kernel.

Let $\kappa\in K_S\otimes\CC$.  At a degree-six $A_1$ point set
$C_0=e_1$, $C_3=h-e_2-e_3$.  Since the convention of
Section~\ref{sec:deformation} is $\ell=C_0-C_3$, replace the
coefficient $\lambda\ell$ by nodal coefficients
$(\lambda,-\lambda)$ at $(C_0,C_3)$.  At an $A_2$ point replace
$\sum_i\mu_i e_i$, $\sum_i\mu_i=0$, by the three nodal coefficients
$\mu_i$.  At a degree-seven point use the same multiple of the
quadric ruling difference.  Retain the old nodal coefficients and
the ruling-difference coefficients at original $\PP^1\times\PP^1$
cones, and retain the whole local vector at each degree-five point.  The
curve identifications of Proposition~\ref{prop03:localmodels} and the
common complements identify the sum of this tuple $\nu$ with the
sum of $\kappa$ in $H_2(Y;\CC)$; it is therefore zero.

Write $F$ for the exceptional locus of $Y\to Z$, with local pieces
$F_q\subset Y_q$.  The local Hodge comparisons identify the components
of $\nu$ with local tangent classes and with their images in
$H^2_{F_q}(Y_q,\Omega^2_{Y_q})$.  In the support sequence
\begin{equation}\label{eq04:support}
 H^1(Y\setminus F,\Omega^2)
 \longrightarrow\bigoplus_q H^2_{F_q}(Y_q,\Omega^2)
 \longrightarrow H^2(Y,\Omega^2_Y),
\end{equation}
their image in the last group is the class of the corresponding
curves.  Under the injection
$H^2(Y,\Omega^2_Y)\hookrightarrow H^4(Y;\CC)$ supplied by the
$\partial\bar\partial$-lemma it is the Poincar\'e dual of their
homological sum, hence zero.  Exactness gives a class on
$Y\setminus F$.  Contraction with the volume form and the
depth-three comparison of deformations with the regular locus give
a tangent $\eta\in\mathbb T^1_Z$ with local tuple $\nu$, as in the
proof of Theorem~\ref{thm:D}.

Let $E'$ be the full exceptional set of $Y\to X$.  Restriction from
$Y\setminus F$ to $Y\setminus E'$ commutes with the boundary maps
in~\eqref{eq04:support} and the inclusion of supports
$H^2_F(Y,\Omega^2_Y)\to H^2_{E'}(Y,\Omega^2_Y)$.  At a degree-six
point this sends the nodal tuple to its sum of curve classes in $E$;
at a degree-seven point it sends the ruling difference to the original
root; at a degree-five or original $\PP^1\times\PP^1$ cone
it is the identity.
These are exactly the local components of $\kappa$.

At a degree-seven point $p$ this comparison is made across the flop, and we
give the details.  Let $A=\operatorname{Tot}(K_E)$ be the neighborhood of
$p$ in $Y_0$ and $A'$ the corresponding neighborhood in $Y$, let
$c\subset E$ be the curve $h-e_1-e_2$ flopped in
Proposition~\ref{prop03:localmodels}, let $c'\subset A'$ be its flop, and put
$O=A\smallsetminus c=A'\smallsetminus c'$.  The open set $O$ contains the
punctured neighborhood $U_p$ of $p$ and the surface $E\smallsetminus c$.
For $x\in T^1_{X,p}\cong H^1(U_p;\Omega^2)$ the connecting images
$\delta_A(x)\in H^2_E(A;\Omega^2)$ and
$\delta_{A'}(x)\in H^2_{E'}(A';\Omega^2)$, where $E'$ is the exceptional set of
$A'$, both restrict to the connecting image
$\delta_O(x)\in H^2_{E\smallsetminus c}(O;\Omega^2)$.  The kernel of
$\delta_O$ is the image of $H^1(O;\Omega^2)$.  After contraction with the
toric volume form, Altmann's concentration of $T^1$ in the Gorenstein degree
places $T^1$ in torus weight zero; the restriction of $\Omega$ differs from
the toric form by a unit, which acts on $T^1$ through a nonzero scalar as in
the proof of Lemma~\ref{lem:branchhodge}.  The weight-zero part of
$H^1(O;\Omega^2)$ vanishes: on the affine toric charts of $O$, the
weight-zero regular $q$-forms on the chart of a smooth cone $\sigma$ are
$\wedge^q(\sigma^\perp\otimes\CC)$, and the resulting \v{C}ech complex is
exact in degree one (Appendix~\ref{app:cert}).  Hence $\delta_O$ is
injective on $T^1$, and restriction to $O$ is injective on
$K_p^{\mathrm{cr}}=\delta_A(T^1)$.  The class $\delta_{A'}(x)$ is the Hodge
class of the ruling difference $\lambda(f_1-f_2)$ supplied by the
construction above.  The two rulings are represented by members of the
pencils $|h-e_1|$ and $|h-e_2|$ disjoint from $c$, which are the same curves
in $O$.  Thus $\delta_O(x)$ is also the restriction of the Hodge class of
$\lambda(e_2-e_1)\in K_p^{\mathrm{cr}}$, and injectivity gives
$\delta_A(x)=\lambda(e_2-e_1)$, the prescribed root component.

The construction of blow-down agrees with restriction on the common
regular locus.  Thus $d\beta(\eta)$ has the prescribed local
tangent on $X$, including at points where its coefficient is zero.

Given any desired vector $v\in V_S$ over $\kappa$, its difference
from $d\beta(\eta)$ belongs to $H^1(X,\Theta_X)$.  Choose a Stein
cover $\{U_i\}$ with exactly one member containing each singular
point and all other members avoiding the singular set.  Every overlap
lies in the locus where $\pi$ is an isomorphism.  Represent the
difference by a \v{C}ech cocycle $\{\xi_{ij}\}$ of vector fields.
Glue the trivial deformations of $\pi^{-1}(U_i)$ by the
automorphisms $1+\epsilon\xi_{ij}$ on the overlaps.  The additive
cocycle condition gives the gluing identity modulo $\epsilon^2$.
Pushing forward these trivial local deformations and their transitions
recovers the original cocycle on $X$.  Hence this locally trivial
global tangent also lifts, and adding its lift to $\eta$ gives $v$.
\end{proof}

\Needspace{7\baselineskip}
\subsection{Smoothness of a selected deformation base}

\begin{theorem}\label{thm04:smoothbase}
Suppose the projection $K_S\to R_{p,S_p}$ is nonzero at every
degree-six point $p$.  Then $\Def^S(X)$ is smooth, of dimension
\[
 h^1(X,\Theta_X)+\dim_{\QQ}K_S.
\]
Every tangent in $T_0\Def^S(X)$ is realized by a convergent curve in
this germ.  No nonzero-projection assumption is required at nodes,
degree-five or degree-seven points, or $\PP^1\times\PP^1$ cones.
\end{theorem}

\begin{proof}
Let $B$ be the completed local ring of $\Def(Z)$.  By
Lemma~\ref{lem04:unobstructed} it is a formal power-series ring, in
particular a domain.  At a degree-seven point the local base is
$\CC[[s,w]]/(w^2,sw)$, so the pullback of $w$ to $B$ vanishes.
At a degree-six point the local base is
$\CC[[s,u,v]]/(su,sv)$, with $A_1$ component $u=v=0$ and
$A_2$ component $s=0$.  If $A_1$ was selected, the nonzero
projection hypothesis and Lemma~\ref{lem04:globaltangent} imply
that the pullback of $s$ has a nonzero linear term.  It is nonzero in
$B$, and the two equations force the pullbacks of $u,v$ to vanish.
For a selected $A_2$ component, one of $u,v$ has a nonzero linear
term, forcing the pullback of $s$ to vanish.  Thus $\beta$ factors
through the full selected formal subspace.  Nodes, degree-five
cones, and original $\PP^1\times\PP^1$ cones impose no branch equations, since their
entire local bases are retained.  Its differential onto
that subspace is surjective by Lemma~\ref{lem04:globaltangent}.

Write its completed target ring minimally as
$A=\CC[[x_1,\ldots,x_r]]/I$, with $I\subset(x_1,\ldots,x_r)^2$.
Surjectivity of the differential means that the images of the $x_i$
in $B$ have independent linear terms.  Extend them to a formal
coordinate system of $B$.  The composite
$\CC[[x_1,\ldots,x_r]]\to B$ is injective, whereas every element
of $I$ maps to zero.  Hence $I=0$.  Equation~\eqref{eq04:tangent}
gives the dimension.  An analytic local ring is regular if its
completion is regular, so the analytic subgerm $\Def^S(X)$ is
smooth.  Analytic coordinates on that germ realize every tangent by
a convergent curve.  This last step does not require convergence of
the formal map $\beta$.
\end{proof}

\subsection{Necessity from local periods}

We recall the form of the period argument that will be used after a
simultaneous partial resolution.  Its hypotheses require compactness;
the intermediate space need not be projective.

\begin{lemma}[period necessity]\label{lem04:periodnecessity}
Let $T$ be a connected compact normal Gorenstein threefold with
$\omega_T\cong\cO_T$, $H^1(T,\cO_T)=0$, and a smooth crepant
resolution $\widetilde T\to T$ that is small over the nodes and the blowup of
the vertex at each cone point.  Suppose its singularities are nodes
and anticanonical cones over $\dP_5$, $\dP_6$, $\dP_7$, or
$\PP^1\times\PP^1$, and that it has a
convergent proper smoothing which uses the $A_1$ component at every
degree-six point.  Then the relation space of its selected local
subspaces in $H_2(\widetilde T;\QQ)$ contains a vector nonzero
on every rank-one summand and satisfying
$\prod_{i=1}^5\lambda_i\ne0$ at every degree-five point.
\end{lemma}

\begin{proof}
We give the argument in a form applicable to the possibly nonprojective
space $T$; the local period computations and link markings are those
of \cite[\S\S3--4]{PaperIV}, with the
$\PP^1\times\PP^1$ case supplied below and the degree-five
case supplied by Lemma~\ref{lem03:degree5periods}.
Write $T_t$ for a smooth nearby fibre.
Remove disjoint small singular neighborhoods and let $W$ be the
common exterior with boundary $\bigsqcup_pL_p$.  Smooth local
triviality along $W$ gives decompositions
\[
 T_t=W\cup\bigsqcup_p M_p,\qquad
 \widetilde T=W\cup\bigsqcup_p N_p,
\]
where $M_p$ is the local Milnor fibre and $N_p$ a resolution
neighborhood, with the canonical common link identifications.

The smoothing has a relative holomorphic volume form after shrinking
the disc.  Here are the details needed without a polarization.  The
relative dualizing sheaf is invertible because the family is flat
with Gorenstein central fibre, and is compatible with base change.
Its first Chern class restricts to zero on $W$, by smooth triviality
and $\omega_T\cong\cO_T$, and on each $M_p$, by a local
trivialization near the central singularity.  The links satisfy
$H^1(L_p;\ZZ)=0$: for nodes this follows from
$L_p\cong S^2\times S^3$, and for the cone links it follows from
the homotopy sequence of the circle bundle with nonzero Euler class
$K_E$ over the simply connected del Pezzo surface.  Mayer--Vietoris
therefore injects $H^2(T_t;\ZZ)$ into the sum of the corresponding
groups of the pieces, so $c_1(\omega_{T_t})=0$.  Semicontinuity gives
$H^1(T_t,\cO_{T_t})=0$; the exponential sequence implies
$\omega_{T_t}\cong\cO_{T_t}$.  Connectedness and Grauert base
change now give a relative section $\Omega$ generating the canonical
sheaf on every fibre.  Its restriction to a smooth fibre is closed.

At a node with local equation $xy-zw=b(t)$, the period of the
residue form on the vanishing sphere is $c\,b(t)$, where $c\ne0$.
Indeed, after writing the quadratic form as $\sum_i z_i^2$, the
sphere is $\sqrt b\,S^3$ and the residue form scales by $b$;
its period at $b=1$ is $\pi^2$ up to orientation and the linear
change of coordinates.  The global form is a holomorphic unit times
this residue form.  Uniformly on the shrinking sphere its period is
therefore
\[
 b(t)\bigl(c\,g(p,0)+o(1)\bigr),\qquad g(p,0)\ne0,
\]
which is nonzero for every sufficiently small $t\ne0$, regardless
of the vanishing order of $b(t)$.

At a $\PP^1\times\PP^1$ cone, Lemma~\ref{lem03:quadric}
identifies the local smoothing with a pullback of
\eqref{eq03:quadfamily}, say with parameter $b(t)$.
The nodal residue form on its double cover is invariant under the
involution.  The total quotient $\AAA^4/\boldsymbol\mu_2$ is
Gorenstein because $-I\in\operatorname{SL}_4$; adjunction for its
parameter function shows that the descended residue generates the
relative dualizing sheaf, including at the central vertex.
This remains true after pullback along $b(t)$.
On the noncentral fibre the covering
is free, and the antipodal involution preserves the orientation of
the vanishing $3$-sphere.  Thus the period on the quotient cycle
$\sqrt{b(t)}\,S^3/\boldsymbol\mu_2$ is half its period on the
sphere.  The global volume form differs from the descended residue
form by a holomorphic unit $g$, and the cycle contracts to the
central vertex.  Consequently its period is
\[
 b(t)\bigl(\tfrac12 c\,g(p,0)+o(1)\bigr),\qquad c\ne0.
\]
It is nonzero for all sufficiently small $t\ne0$, for any
vanishing order of $b(t)$.  Lemma~\ref{lem03:quadricmilnor}
identifies this cycle with a generator of $H_3(M_p;\QQ)$ and
identifies the local boundary map with the ruling-difference line.
The link has $H^1(L_p;\ZZ)=0$ even though its fundamental group
is $\ZZ/2$, as required in the preceding canonical-bundle argument.

For the degree-six and degree-seven branches in the lemma, the smooth affine fibre is
$M=V\setminus E$, with $V=(\PP^1)^3$ for degree six and
$V=\operatorname{Bl}_{\mathrm{pt}}\PP^3$ for degree seven, and
$E\in|-\tfrac12K_V|$.  Every such pair $(V,E)$ is isomorphic to the toric
pair described below \cite[\S3.3]{PaperIV}.  The comparison with a compact
local Milnor fibre, including its boundary marking, is
\cite[Lem.~4.2]{PaperIV}.  Weak Lefschetz makes
$H^2(V;\QQ)\to H^2(E;\QQ)$ injective; the respective Picard
ranks are $(3,4)$ and $(2,3)$.  Since $H^3(V;\QQ)=0$, the pair
sequence gives $b_3(M)=1$.  More explicitly, take the character
$w=\chi^m$ with $m=(1,1,1)$ on $(\PP^1)^3$, and $m=(1,1,-1)$
on the blowup fan with rays $e_1,e_2,e_3,-e_1-e_2-e_3,e_1+e_2+e_3$.
The closure of $w=1$ is $E$, and
\[
 \Omega_M=\frac{w}{(w-1)^2}
      \frac{dx}{x}\wedge\frac{dy}{y}\wedge\frac{dz}{z}
\]
has order zero at each toric boundary divisor: the order is
$\langle m,v\rangle-2\min(\langle m,v\rangle,0)-1=0$, since
$\langle m,v\rangle=\pm1$.  Its only pole is the double pole at
$E$, so it generates $\omega_M$.  In adjacent toric charts it is
$\pm dp\wedge dq\wedge dz/(q-pz)^2$.  The transitions are
\[
 (p',q',z')=(pz^c,qz^c,z^{-1}),\qquad
 c=0\ \text{or}\ -1
\]
in the two respective cases.  Fix $\varepsilon>0$ and $R>1$.
A closed three-cycle consists of the two caps
\[
 \{p=0,\ q=\varepsilon R e^{i\psi},\ |z|\le R\},\qquad
 \{p'=\varepsilon R^c e^{i\psi},\ q'=0,\ |z'|\le R^{-1}\}
\]
and the intervening chain
\[
 p=r\varepsilon e^{i\theta},\quad
 q=-(1-r)\varepsilon R e^{i(\theta+\varphi)},\quad
 z=R e^{i\varphi},\qquad 0\le r\le1.
\]
The boundaries agree under the stated transition, with opposite
orientations.  On the chain $q-pz=-\varepsilon R
e^{i(\theta+\varphi)}$, so it misses $E$; the caps also miss $E$.
The form vanishes on the caps and pulls back on the chain to
$\pm dr\wedge d\theta\wedge d\varphi$.  Its period is consequently
$\pm4\pi^2$, in particular nonzero.

The one-dimensional local smoothing is equivariant with positive
weights on the cone coordinates and weight one on its parameter
$b$.  Graded Nakayama gives a homogeneous generator of its relative
canonical sheaf near the vertex.  Its nonvanishing locus is invariant
and contains the vertex, hence contains the entire fibre $b=1$:
every orbit in that fibre contracts to the vertex.  Its restriction
is therefore a constant multiple of the preceding form.
The multiplier is constant because an algebraic unit on $V\setminus E$
has divisor supported on $E$, and $\operatorname{Pic}(V)$ is
torsion-free with $[E]\ne0$.  Scaling the fixed compact cycle into
the fibre over $b(t)$ contracts it to the vertex.  If the homogeneous
form has weight $a$, its global period is
\[
 b(t)^a\bigl(c\,g(p,0)+o(1)\bigr),\qquad c\ne0.
\]
It is again nonzero for small $t\ne0$.
At a degree-five point, Lemma~\ref{lem03:degree5periods}
provides five classes $\nu_{p,i}$, the images of its boundary
classes $\beta_i$, with nonzero periods.  Their dual link
functionals are the five $\lambda_i$.
Every one of these specified local cycles survives in
$H_3(T_t;\QQ)$: otherwise it would bound a rational chain and
have zero period against the closed form $\Omega_t$.

By Poincar\'e duality choose, for each specified local cycle, a
rational class $\gamma\in H_3(T_t;\QQ)$ pairing nontrivially
with it.
Its Mayer--Vietoris boundary belongs to
\[
 \ker\left\{\bigoplus_q H_2(L_q;\QQ)
 \longrightarrow H_2(W;\QQ)\oplus\bigoplus_qH_2(M_q;\QQ)\right\}.
\]
It therefore gives a relation in $H_2(\widetilde T;\QQ)$ and
belongs at each point to the selected local subspace.
Lefschetz duality identifies the local restriction of $\gamma$
with its pairing against $H_3(M_p;\QQ)$.  At a rank-one germ,
the boundary map
$H_3(M_p,L_p;\QQ)\to H_2(L_p;\QQ)$ is an isomorphism onto the
root line: its image is that line by the local Milnor-fibre
calculation, and its domain has dimension one.  The corresponding
coefficient is therefore nonzero.  At a degree-five point,
\eqref{eq03:degree5boundary} is an isomorphism onto the full
link space.  Compatibility of the boundary with the Lefschetz
pairing identifies its pairing with $\beta_i$ with the preceding
nonzero pairing against $\nu_{p,i}$.  By
Lemma~\ref{lem03:degree5periods} this is a nonzero multiple of
$\lambda_i$ on the resulting local relation vector.
Thus each required scalar functional is nonzero on some rational
global relation.  Avoiding the finitely many proper kernels gives
one relation on which all of them are nonzero.
\end{proof}

\subsection{The exact criterion}

\begin{theorem}\label{thm04:criterion}
The threefold $X$ has a convergent one-parameter smoothing inducing
$S$ if and only if $K_S$ contains a vector nonzero on every
rank-one summand, satisfying $\prod_{i=1}^5\lambda_i(\kappa_p)\ne0$
at every degree-five point, and, at every $A_2$-assigned point, with
\[
 \kappa_p=\mu_1e_1+\mu_2e_2+\mu_3e_3,\qquad
 \mu_1+\mu_2+\mu_3=0,\qquad \mu_1\mu_2\mu_3\ne0.
\]
Equivalently, none of these individual scalar functionals vanishes
identically on $K_S$.
\end{theorem}

\begin{proof}
Suppose first that such a relation exists.  It gives a vector of
$V_S$ by~\eqref{eq04:tangent} and implies the nonzero-projection
hypothesis of Theorem~\ref{thm04:smoothbase}.  Realize that tangent
by a convergent curve in $\Def^S(X)$.  At each rank-one point the
local smoothing parameter has nonzero derivative.  At each $A_2$
point all three local discriminant factors have nonzero derivatives,
by the normalized identifications of
Proposition~\ref{prop03:a2family}.  At each degree-five point
the five discriminant factors have nonzero derivatives by
Lemma~\ref{lem03:degree5marking}.  Thus the local fibres are
smooth for every sufficiently small nonzero parameter.  Properness
and openness of smoothness exclude additional singularities outside
neighborhoods of the finitely many original singularities.  The
curve therefore gives a smoothing of $X$.

Conversely, let $\mathcal X\to\Delta$ be any convergent proper
smoothing inducing $S$.  At each $A_2$ point pull back the
simultaneous incidence modification of
Proposition~\ref{prop03:a2family} along its actual analytic
classifying arc.  Since this arc is a smoothing, it meets the
central point of the local base only at $t=0$, after shrinking.
The pulled-back modification is proper and flat over $\Delta$,
is an isomorphism off the central cone point, and is an isomorphism
over every noncentral fibre.  Glue these morphisms to the identity
on the complement of those central points in $\mathcal X$.  The
Hausdorffness and properness argument of
Proposition~\ref{prop04:globalmodel} applies to this cover of the
total space; flatness is local on the source.  We obtain a proper
flat family $\mathcal T\to\Delta$ with $T_t\cong X_t$ for
$t\ne0$.

Its central fibre $T$ replaces each $A_2$ point by three nodes and
leaves all other germs unchanged.  Its smooth crepant resolution is
$Y_0$: the three new exceptional curves are $e_1,e_2,e_3$ in the
original del Pezzo surface.  Rationality and crepancy, as in
\eqref{eq04:cy}, give $\omega_T\cong\cO_T$ and
$H^1(T,\cO_T)=0$.  The remaining degree-six points use $A_1$.
Lemma~\ref{lem04:periodnecessity} therefore supplies a relation
nonzero at every node and every remaining rank-one cone subspace,
and with all five $\lambda_i$ nonzero at each degree-five point.  In
particular, the three new nodes over any one $A_2$ point have
nonzero coefficients $\mu_1,\mu_2,\mu_3$.

Pair this relation with the original exceptional divisor $E$ in
$Y_0$.  Every term belonging to another original singular point has
zero intersection with $E$, whereas
$E\cdot e_i=K_E\cdot e_i=-1$.  Hence
$\mu_1+\mu_2+\mu_3=0$.  Replacing the three nodal terms by
$\sum_i\mu_i e_i$ reconstructs a vector in $K_S$ with all the
required nonvanishing conditions.  This argument uses the entire
classifying arc; it imposes no assumption on its first nonzero order
and excludes an escape from a discriminant line at any higher order.
Finally, a finite-dimensional rational vector space is not a union
of finitely many proper linear subspaces, which proves the last
equivalence.
\end{proof}

In particular, when no degree-six points occur, every selected base
is smooth, even if a coordinate at a node, a degree-seven point,
or a $\PP^1\times\PP^1$ cone vanishes
identically on its tangent image.  Smoothness of a deformation base
and existence of a smooth fibre are distinct conclusions.  For a
degree-six point assigned $A_2$, a nonzero local vector alone does
not ensure a smoothing: each of its three scalar coefficients is
required by Theorem~\ref{thm04:criterion}.

If the singularities are only nodes and anticanonical cones over
$\dP_5$ or $\PP^1\times\PP^1$, the selected local bases are the entire
local bases.  Thus $\Def^S(X)=\Def(X)$ is smooth, of dimension
$h^1(X,\Theta_X)+\dim_\QQ K$.  Existence of a smoothing still
requires the nonzero relation of Theorem~\ref{thm04:criterion}.

\section{A nonsmoothable threefold and its deformation space}
\label{sec05:counterexample}

The three separate nonvanishing conditions on an $A_2$ summand cannot be
replaced by nonvanishing of the vector in that summand.  The example below
also shows that changing the choices of local components does not repair
this weaker condition.
Write $L$ for the $A_1$ line and $P$ for the $A_2$ plane.  The profiles
$LP$ and $PL$ specify these choices in the order of the two cone points.

\begin{theorem}\label{thm05:counterexample}
There is a normal projective Calabi--Yau threefold $X$ with exactly thirty
ordinary double points and two anticanonical $\dP_6$ cone points such that
$K_S$ contains a relation nonzero at every singular point for each of the
two mixed profiles $S=LP,PL$, but $X$ has no smoothing.  A smooth projective
crepant resolution $Y_0$ has
\[
 (h^{1,1}(Y_0),h^{2,1}(Y_0))=(23,21).
\]
\end{theorem}

We prove the nonsmoothability assertion directly
from nodal necessity and the simultaneous partial resolution of the
$A_2$ family.  Its proof is independent of the sufficiency assertion in
Theorem~\ref{thm04:criterion}.

\subsection{The threefold and its relation matrix}

Let $\Delta_{25}\subset\ZZ^4\otimes\RR$ be the convex hull of the following
vertices.  All indices in this section start at zero.
\begin{equation}\label{eq05:vertices}
\begin{gathered}
\begin{adjustbox}{max width=\textwidth,center}
\begin{tabular}{r@{\quad}c@{\qquad}r@{\quad}c@{\qquad}r@{\quad}c}
\toprule
$i$&$v_i$&$i$&$v_i$&$i$&$v_i$\\
\midrule
0&$(1,0,0,0)$&1&$(0,1,0,0)$&2&$(0,0,1,0)$\\
3&$(0,0,0,1)$&4&$(1,-1,1,-1)$&5&$(0,0,1,-1)$\\
6&$(1,-1,0,0)$&7&$(-1,1,0,0)$&8&$(0,0,-1,1)$\\
9&$(-1,1,-1,1)$&10&$(0,0,0,-1)$&11&$(0,0,-1,0)$\\
12&$(-1,1,-1,0)$&13&$(0,-1,1,0)$&14&$(-1,0,1,0)$\\
15&$(0,-1,0,1)$&16&$(-1,0,0,1)$&17&$(-1,0,1,-1)$\\
18&$(0,-1,1,-1)$&19&$(-1,0,-1,1)$&20&$(0,-1,-1,1)$\\
21&$(0,-1,-1,0)$&22&$(0,-1,0,-1)$&23&$(-1,0,0,-1)$\\
24&$(-1,0,-1,0)$&&&&\\
\bottomrule
\end{tabular}
\end{adjustbox}
\end{gathered}
\end{equation}
This is the fan polytope: the ambient fan consists of the cones over its
proper faces, and the anticanonical Newton polytope is $\Delta_{25}^\circ$.
The polytope $\Delta_{25}$ is reflexive.  It has $23$ facets, $71$ primitive edges,
and $69$ two-faces: $37$ unimodular triangles, $30$ unit parallelograms,
and two unit-edge hexagons.  The latter have cyclic orders
\begin{equation}\label{eq05:hexagons}
 F_0=(13,15,20,21,22,18),\qquad
 F_1=(14,16,19,24,23,17)
\end{equation}
and centres $(0,-1,0,0)$ and $(-1,0,0,0)$.  These two centres and the
origin are the only lattice points of $\Delta_{25}$ other than its vertices.
Every parallelogram and hexagon has dual edge of lattice length one.
The complete list of parallelograms, in the cyclic orders used below, is
\begin{equation}\label{eq05:nodes}
\begin{gathered}
\begin{adjustbox}{max width=\textwidth,center}
\begin{tabular}{r@{\quad}c@{\qquad}r@{\quad}c@{\qquad}r@{\quad}c}
\toprule
$j$&node face&$j$&node face&$j$&node face\\
\midrule
0&$(0,1,5,4)$&1&$(0,1,9,8)$&2&$(0,1,12,11)$\\
3&$(0,3,15,6)$&4&$(0,6,20,8)$&5&$(0,6,21,11)$\\
6&$(1,2,14,7)$&7&$(1,5,17,7)$&8&$(2,3,15,13)$\\
9&$(2,3,16,14)$&10&$(2,5,17,14)$&11&$(3,8,20,15)$\\
12&$(4,5,17,18)$&13&$(4,6,15,13)$&14&$(4,6,21,22)$\\
15&$(5,10,23,17)$&16&$(7,9,16,14)$&17&$(7,12,23,17)$\\
18&$(8,9,12,11)$&19&$(8,9,19,20)$&20&$(8,11,21,20)$\\
21&$(9,12,24,19)$&22&$(10,11,21,22)$&23&$(11,12,24,21)$\\
24&$(13,14,16,15)$&25&$(13,14,17,18)$&26&$(15,16,19,20)$\\
27&$(17,18,22,23)$&28&$(19,20,21,24)$&29&$(21,22,23,24)$\\
\bottomrule
\end{tabular}
\end{adjustbox}
\end{gathered}
\end{equation}
All these lattice assertions follow by computing the supporting
hyperplanes and their face incidences from \eqref{eq05:vertices}; the
listed singular faces also specify the local lattice tests.  In
particular, each parallelogram has primitive three-dimensional cone
lattice, and each triangle in the central subdivision of a hexagon is
unimodular.

Let $X$ be a general anticanonical hypersurface in the projective
Gorenstein toric Fano fourfold defined by this fan.  Nondegeneracy and
the dual-edge lengths give one node for each face in
\eqref{eq05:nodes} and one anticanonical $\dP_6$ cone for each face in
\eqref{eq05:hexagons}.  There are no other singularities.  The hypersurface
avoids the torus fixed points, and
\[
 \omega_X\cong\cO_X,\qquad H^1(X,\cO_X)=H^2(X,\cO_X)=0.
\]
Insert the two centre rays and choose a projective crepant simplicial
refinement retaining the central subdivisions of the hexagons; one exists by
pulling the two centres first and then refining regularly, as in
\cite[\S8.1]{PaperIV}.  Its
restriction gives a smooth projective crepant resolution $f:Y_0\to X$;
the remaining possible ambient singularities lie over torus fixed
points and do not meet $Y_0$.  These are the standard hypersurface and
resolution constructions of \cite[\S\S4.1--4.2]{Batyrev94}.  This choice
selects a small resolution at each node; independent arbitrary choices
of all thirty diagonals are not needed.

Write $E_i\cong\operatorname{Bl}_3\PP^2$ for the exceptional divisor over
$F_i$.  In each of the cyclic orders \eqref{eq05:hexagons}, mark its
boundary curves by
\begin{equation}\label{eq05:boundary}
 (C_0,\ldots,C_5)=
 (e_1,h-e_1-e_2,e_2,h-e_2-e_3,e_3,h-e_1-e_3).
\end{equation}
The intersection form on $(h,e_1,e_2,e_3)$ is
$\operatorname{diag}(1,-1,-1,-1)$.  With the root basis fixed in
Section~\ref{sec:deformation},
\[
 \ell=-h+e_1+e_2+e_3,\qquad A=-e_1+e_3,\qquad B=e_2-e_3,
\]
the boundary intersection rows are
\begin{equation}\label{eq05:rootrows}
\begin{aligned}
 r_\ell&=(-1,1,-1,1,-1,1),\\
 r_A&=(1,-1,0,1,-1,0),\\
 r_B&=(0,1,-1,0,1,-1).
\end{aligned}
\end{equation}
All three classes are perpendicular to $K_{E_i}$.

Define a $36\times25$ matrix $M$ as follows.  Its first thirty rows are
indexed by \eqref{eq05:nodes}: a cyclic square $(i,j,k,l)$ gives entries
$(1,-1,1,-1)$ in those columns and zero elsewhere.  These rows fix signed
generators $\gamma_j$ of the corresponding nodal link homology; replacing
a small resolution reverses the relevant sign.  The last six rows are
$r_\ell,r_A,r_B$ on $F_0$, followed by $r_\ell,r_A,r_B$ on $F_1$, placed
in the six vertex columns in the orders \eqref{eq05:hexagons}.  Equations
\eqref{eq05:nodes}--\eqref{eq05:rootrows} specify every entry of $M$.
The two additional columns for the centre rays are zero, because the
root classes pair trivially with $K_{E_i}$.

The ambient divisor intersections detect precisely the homological
relations.  Indeed, $\Delta_{25}$ has $28$ lattice points and no facet-interior
lattice points.  The only two-faces with interior points are the two
hexagons, and their dual edges have no interior lattice points.
Batyrev's divisor formula therefore has zero correction term and gives
\begin{equation}\label{eq05:h11}
 h^{1,1}(Y_0)=28-5=23.
\end{equation}
The same construction of the divisor classes shows that the $27$ ambient
boundary divisors span $H^2(Y_0,\QQ)$
\cite[Theorem~4.4.2 and Corollary~4.5.1]{Batyrev94}.  Thus, writing a
tuple of local classes as
\[
 w=(n_0,\ldots,n_{29},\lambda_0,a_0,b_0,\lambda_1,a_1,b_1),
\]
its total class
\begin{equation}\label{eq05:totalclass}
 \sum_{j=0}^{29}n_j\gamma_j+
 \sum_{i=0}^1(\lambda_i\ell_i+a_i A_i+b_i B_i)
\end{equation}
vanishes in $H_2(Y_0,\QQ)$ if and only if $wM=0$.
Consequently
\[
 K=\{w\in\QQ^{36}:wM=0\}
\]
is the full homological relation space.  Its restriction to a profile
$S$ is $K_S$.  Whenever these coefficients are used as tangent
coordinates, they are normalized by the fixed global volume form and
Lemma~\ref{lem:localhodge}; independent choices of raw local equation
parameters are not implicit in $M$.

\subsection{A nonzero relation and the obstruction to smooth fibres}

For the profile $LP$, take
\begin{equation}\label{eq05:weakwitness}
\begin{aligned}
 (\lambda_0,a_0,b_0,\lambda_1,a_1,b_1)&=(3,0,0,0,1,0),\\
 (n_0,\ldots,n_{29})={}&(-2,-1,2,2,-2,1,1,-2,2,2,\\
 &\hphantom{(}-3,6,-1,1,-2,-4,-3,2,2,-4,\\
 &\hphantom{(}5,-4,-4,6,3,3,1,1,-1,4).
\end{aligned}
\end{equation}
Multiplication by the explicitly defined matrix gives $wM=0$.
Every nodal coefficient is nonzero, and the two cone components are
$3\ell$ on $E_0$ and $A$ on $E_1$, both nonzero.  By
\eqref{eq05:totalclass}, this is a
homological relation, so the weaker nonvanishing condition holds.  An
explicit relation on $PL$ has cone coordinates $(0,1,0,3,0,0)$; its nodal
coordinates agree with \eqref{eq05:weakwitness} for $0\leq j\leq24$ and
are
\[
 (n_{25},n_{26},n_{27},n_{28},n_{29})=(-1,5,2,-2,7).
\]
It too annihilates $M$ and is nonzero at every singular point.

Let $D_i$ be the ambient divisor associated with $v_i$, restricted to
$Y_0$, and set
\begin{equation}\label{eq05:divisors}
\begin{aligned}
 H_L&=-\sum_{i\in\{0,1,2,4,5,6,7,13,14,17,18\}}D_i,\\
 H_M&=\sum_{i\in\{2,3,13,14,15,16\}}D_i,\\
 H_P&=\sum_{i\in\{0,1,2,3,4,5,6,7,13,14,15,16,17,18\}}D_i.
\end{aligned}
\end{equation}
Pairing \eqref{eq05:totalclass} with these three divisors gives, for
every $w\in K$,
\begin{equation}\label{eq05:identities}
\begin{aligned}
 n_{13}-a_0+b_0-a_1+b_1&=0,\\
 b_0+b_1&=0,\\
 n_{16}+\lambda_0+\lambda_1&=0.
\end{aligned}
\end{equation}
These equalities can be checked directly from the four-entry nodal rows
and the six-entry rows \eqref{eq05:rootrows}.  In particular they are
identities on the full relation space, before any local component is
selected.  They imply the following restrictions:
\begin{equation}\label{eq05:forced}
 \left.n_{13}\right|_{K_{LL}}=0,\qquad
 \left.b_1\right|_{K_{LP}}=0,\qquad
 \left.b_0\right|_{K_{PL}}=0,\qquad
 \left.n_{16}\right|_{K_{PP}}=0.
\end{equation}
For a plane component the interpretation of $b_i$ is intrinsic to the
marking:
\begin{equation}\label{eq05:threecoefficients}
 a_iA_i+b_iB_i=-a_i e_1+b_i e_2+(a_i-b_i)e_3.
\end{equation}
It is the coefficient of the second of the three nodal curves in the
simultaneous partial resolution.

\begin{proof}[Proof of the nonsmoothability assertion in
Theorem~\ref{thm05:counterexample}]
Suppose that $X$ has a convergent one-parameter smoothing.  At each cone
point the classifying map from the disc factors through one of the two
local components: in the domain $\CC\{t\}$, the equations
$\lambda_i a_i=\lambda_i b_i=0$ force either $\lambda_i=0$ or
$a_i=b_i=0$.  Hence the smoothing has one of the four profiles.

We use the following form of nodal necessity.  A smoothing and a chosen
ordinary double point produce a homological relation on a crepant
resolution whose coefficient at that node is nonzero; at the remaining
cone points its components belong to the selected branch subspaces.
To recall the argument, the local vanishing sphere has a nonzero period
of the Calabi--Yau volume form, and so is nonzero in the homology of a
smooth fibre.  A dual global three-cycle has a boundary tuple on the
links.  The boundary exact sequence makes its total class zero on the
resolution, and its component at the chosen node is nonzero.  The
boundary classes of the other local fillings give the stated branch
restrictions.  This is the period argument of
Lemma~\ref{lem04:periodnecessity} (compare \cite[\S1]{RT09}), with the cone
boundary description of \cite[Prop.~2.1 and Lem.~4.2]{PaperIV}; it concerns
the given smoothing, without a condition on its Kodaira--Spencer class.

For $LL$, apply this argument to node $13$.  The first identity in
\eqref{eq05:identities} forces its coefficient to be zero, a contradiction.
For $PP$, apply it to node $16$ and use the third identity.

For $LP$ or $PL$, apply the simultaneous modification of
Proposition~\ref{prop03:a2family} at the cone assigned to $P$.  More
explicitly, its local family consists of rank-at-most-one matrices with
diagonal entries $z+d_1,z+d_2,z+d_3$.  Remembering the column line and
writing the matrix as $\xi\eta^{\mathsf T}$ gives, on $\xi_k=1$,
\[
 \xi_i\eta_i-\xi_j\eta_j=d_i-d_j,
 \qquad \{i,j,k\}=\{1,2,3\}.
\]
The modification is proper and flat over the two-dimensional base.  Its
central fibre replaces the cone by three nodes with small-resolution
curves $e_1,e_2,e_3$.  It is an isomorphism over every noncentral base
point, since the zero matrix can occur only when all diagonal differences
vanish.  Pulling it back along the hypothetical smoothing and gluing to
the identity off the central cone point therefore gives a proper flat
family with the same smooth noncentral fibres.  Its central threefold
$Z$ has the thirty old nodes, these three new nodes, and the remaining
cone on its $L$ branch.  Crepancy and rationality give
$\omega_Z\cong\cO_Z$ and $H^1(Z,\cO_Z)=0$.

Apply nodal necessity to the new node represented by $e_2$.  The resulting
boundary relation on $Y_0$ has a nonzero coefficient at $e_2$.  However,
$H_M$ pairs trivially with every old nodal curve and with the remaining
$A_1$ root.  On the modified exceptional surface,
\[
 H_M|_{E_i}=e_1+(h-e_1-e_2)=h-e_2,
 \qquad (H_M\cdot e_1,H_M\cdot e_2,H_M\cdot e_3)=(0,1,0).
\]
Pairing the relation with $H_M$ forces its $e_2$ coefficient to vanish.
This contradiction excludes both mixed profiles.  The argument applies
also to smoothing arcs with arbitrarily high contact with the local
discriminant.  Together with \eqref{eq05:weakwitness}, it proves the
claimed failure of the weaker condition.  The Hodge numbers are
\eqref{eq05:h11} and Lemma~\ref{lem05:locallytrivial} below.
\end{proof}

\subsection{Tangent dimensions and the four selected deformation spaces}

\begin{lemma}\label{lem05:locallytrivial}
For this threefold,
\[
 h^1(X,\Theta_X)=h^{2,1}(Y_0)=21.
\]
\end{lemma}

\begin{proof}
The polar $\Delta_{25}^\circ$ has $26$ lattice points and no facet-interior lattice
points.  Every edge of $\Delta_{25}$ is primitive, so the correction term in
Batyrev's formula for $h^{2,1}$ vanishes.  Thus
\[
 h^{2,1}(Y_0)=26-5=21
\]
by \cite[Corollary~4.5.1]{Batyrev94}.  The counts may be obtained directly
from $\langle m,v_i\rangle\geq-1$ using \eqref{eq05:vertices}.

We verify the tangent-sheaf comparison needed for the first equality.
A vector field at an isolated cone point vanishes at the vertex, and
its local flow preserves the vertex ideal.  It therefore lifts to the
vertex blow-up $\operatorname{Tot}(K_E)$.  At a node the local flow lifts
to the blow-up and preserves each of the two ruling components of its
exceptional quadric: a one-parameter flow cannot interchange these
discrete choices.  It consequently lifts to the chosen small
resolution.  Conversely a vector field on the resolution descends on
the regular locus and extends across the point by normality, or
equivalently by reflexivity of the tangent sheaf.  Hence
\begin{equation}\label{eq05:pushforwardtheta}
 f_*\Theta_{Y_0}=\Theta_X.
\end{equation}
This is the equivariant-resolution comparison of
\cite[Definition~1.11 and Lemma~1.15]{FL22a} in these local models.

We also have $R^1f_*\Theta_{Y_0}=0$.  At a degree-six point use the
filtration by powers of the zero-section ideal.  The tangent bundle
restricts as $\Theta_E\oplus K_E$, so its graded terms are
\begin{equation}\label{eq05:gradedtheta}
 \Theta_E\otimes\cO_E(-kK_E),\qquad
 K_E\otimes\cO_E(-kK_E),\qquad k\geq0.
\end{equation}
The second term has vanishing $H^1$: use Serre duality and rationality
for $k=0$, rationality for $k=1$, and Kodaira vanishing for $k\geq2$.
For the first term, twist the toric Euler sequence
\[
 0\longrightarrow\cO_E^{\oplus4}
 \longrightarrow\bigoplus_{j=0}^5\cO_E(C_j)
 \longrightarrow\Theta_E\longrightarrow0.
\]
At $k=0$, the sequence
$0\to\cO_E\to\cO_E(C_j)\to\cO_{C_j}(-1)\to0$ gives
$H^1(E,\cO_E(C_j))=0$.  For $k\geq1$, the divisor $C_j-kK_E$ is nef:
its intersections with the six generators \eqref{eq05:boundary} of
the effective cone of curves are nonnegative, since $C_j^2=-1$ and
$(-K_E)\cdot C_l=1$ for every $l$.  Toric vanishing gives
$H^1(E,\cO_E(C_j-kK_E))=0$.  Finally
$H^2(E,\cO_E(-kK_E))=0$ for every $k\geq0$, so the twisted Euler
sequence proves the other $H^1$ vanishing in \eqref{eq05:gradedtheta}.

At a nodal small resolution the exceptional curve has normal bundle
$\cO_{\PP^1}(-1)^{\oplus2}$, and the corresponding graded tangent
terms are sums of $\cO_{\PP^1}(k+2)$ and
$\cO_{\PP^1}(k-1)$ for $k\geq0$.  Their $H^1$ groups also vanish.
Induction on the infinitesimal neighborhoods of the exceptional sets
and the theorem on formal functions now give
$R^1f_*\Theta_{Y_0}=0$; these ideal filtrations are cofinal with the
pullbacks of the maximal-ideal filtrations at the singular points.
Leray, \eqref{eq05:pushforwardtheta}, and contraction with the crepant
volume form yield
\[
 H^1(X,\Theta_X)\cong H^1(Y_0,\Theta_{Y_0})
 \cong H^1(Y_0,\Omega^2_{Y_0}),
\]
which proves the assertion.
\end{proof}

Row reduction of the matrix specified above gives
\begin{equation}\label{eq05:fullranks}
 \rk M=20,\qquad \dim K=16,\qquad
 \rk M_{\mathrm{nodes}}=19.
\end{equation}
Here $M_{\mathrm{nodes}}$ consists of the first thirty rows.  Its relation
space has dimension $11$, so projection of $K$ to the six cone
coordinates has dimension $16-11=5$.  The second identity in
\eqref{eq05:identities} therefore describes its entire image:
\begin{equation}\label{eq05:coneimage}
 \operatorname{im}(K\longrightarrow\QQ^6)
 =\{(\lambda_0,a_0,b_0,\lambda_1,a_1,b_1):b_0+b_1=0\}.
\end{equation}

\Needspace{8\baselineskip}
For a profile $S$, let $M_S$ consist of the nodal rows and the selected
root rows, and let $\rho_i:K_S\to R_{i,S_i}$ be its projection to the
$i$th cone summand.  The ranks and resulting dimensions are listed in
Table~\ref{tab05:profiles}.
In particular both cone projections are nonzero on every profile.
Theorem~\ref{thm04:smoothbase} applies and proves that all four selected
deformation spaces are smooth, with the dimensions in the final column.
These dimensions use Lemma~\ref{lem05:locallytrivial} and the exact
sequence of Theorem~\ref{thm:D}.
\begin{table}[htbp]
\caption{Relation spaces and selected deformation spaces of $X$.}
\label{tab05:profiles}
\begin{adjustbox}{max width=\textwidth,center}
\begin{tabular}{c r r r c c r}
\toprule
$S$&rows of $M_S$&$\rk M_S$&$\dim K_S$&$\rk\rho_0$&$\rk\rho_1$
 &$\dim\Def^S(X)$\\
\midrule
$LL$&32&19&13&1&1&34\\
$LP$&33&20&13&1&1&34\\
$PL$&33&20&13&1&1&34\\
$PP$&34&20&14&2&2&35\\
\bottomrule
\end{tabular}
\end{adjustbox}
\end{table}

For the unrestricted deformation space the same theorem gives
\begin{equation}\label{eq05:globaltangent}
 0\longrightarrow H^1(X,\Theta_X)
 \longrightarrow T_0\Def(X)
 \longrightarrow K\otimes_\QQ\CC\longrightarrow0.
\end{equation}
Thus $\dim T_0\Def(X)=37$.  The kernel of its projection to the two
cone tangent spaces has dimension $21+11=32$.  The second summand in
this count consists of global tangent directions whose local
components are supported at the old nodes.

\subsection{The unrestricted analytic deformation germ}

\begin{theorem}\label{thm05:germ}
For the threefold of Theorem~\ref{thm05:counterexample} there is an
isomorphism of convergent local rings
\begin{equation}\label{eq05:normalform}
 \cO_{\Def(X),0}\cong
 \frac{\CC\{z_1,\ldots,z_{32},x,y,u,v,b\}}
 {(xy,xb,uv,ub)}.
\end{equation}
The completed local ring has the same presentation with formal power
series.  The four selected deformation spaces are precisely its
irreducible components.  They are smooth, and their ideals in the
ambient smooth germ are
\begin{equation}\label{eq05:componentideals}
 \begin{aligned}
 P_{LL}&=(y,v,b),& P_{LP}&=(y,b,u),\\
 P_{PL}&=(x,v,b),& P_{PP}&=(x,u).
 \end{aligned}
\end{equation}
In particular the full deformation germ is reduced, although its
deformation functor is obstructed.
\end{theorem}

\begin{proof}
Choose coordinates on the two analytic local bases so that
\begin{equation}\label{eq05:localbases}
 \cO_{\Def(X,p_i),0}=
 \CC\{\lambda_i,a_i,b_i\}/(\lambda_i a_i,\lambda_i b_i)
\end{equation}
and their differentials are the Hodge-normalized coordinates used in
$M$.  This normalization is possible by an invertible linear change
on the line and on the plane separately; it preserves the displayed
ideal.  Fix the local classifying maps used to define $\Def^S(X)$.

By \eqref{eq05:coneimage} and \eqref{eq05:globaltangent}, the pulled-back
functions
\[
 x=\lambda_0,\quad y=a_0,\quad u=\lambda_1,\quad v=a_1,\quad b=b_0
\]
have independent differentials.  Complete them to a minimal analytic
embedding with $32$ further coordinates $z_1,\ldots,z_{32}$, and write
\[
 \cO_{\Def(X),0}=R/I,\qquad
 R=\CC\{z_1,\ldots,z_{32},x,y,u,v,b\},\qquad
 I\subseteq\mathfrak m^2,
\]
where $\mathfrak m$ is the maximal ideal of $R$.
Choose a representative $H\in R$ for the negative of the pullback of
$b_1$.  Equation~\eqref{eq05:coneimage} says exactly that
$H=b+O(\mathfrak m^2)$.  The local equations
\eqref{eq05:localbases} imply
\begin{equation}\label{eq05:localideal}
 F:=(xy,xb,uv,uH)\subseteq I.
\end{equation}
We will prove equality and then remove the higher terms of $H$ by a
convergent change of coordinates.

Let $\operatorname{in}(I)$ denote the ideal of lowest nonzero homogeneous
parts of elements of $I$ in
$\operatorname{gr}_{\mathfrak m}R=\CC[z_1,\ldots,z_{32},x,y,u,v,b]$.
The four elements in \eqref{eq05:localideal} give
\[
 J:=(xy,xb,uv,ub)\subseteq\operatorname{in}(I).
\]
On the other hand, each of the four selected germs is a smooth closed
subgerm of $\Def(X)$, by Table~\ref{tab05:profiles} and
Theorem~\ref{thm04:smoothbase}.  Its tangent space has the corresponding
ideal in \eqref{eq05:componentideals}; this follows from
\eqref{eq05:globaltangent} and the linear relation $b_1=-b_0$.
If $g\in I$, substitute a regular local parametrization of one of these
smooth germs into $g$.  Since the result vanishes, the lowest homogeneous
part of $g$ vanishes on its tangent space.  Therefore
\[
 \operatorname{in}(I)\subseteq
 P_{LL}\cap P_{LP}\cap P_{PL}\cap P_{PP}.
\]
The elementary monomial-ideal identity
\begin{equation}\label{eq05:primeintersection}
 (y,v,b)\cap(y,b,u)\cap(x,v,b)\cap(x,u)
 =(xy,xb,uv,ub)
\end{equation}
now proves
\begin{equation}\label{eq05:initialideal}
 \operatorname{in}(I)=J.
\end{equation}
For example, \eqref{eq05:primeintersection} can be checked monomial by
monomial: membership in all four coordinate primes forces divisibility
by at least one of $xy,xb,uv,ub$, and each of those four monomials
belongs to every prime.

Since $J$ is an intersection of prime ideals, the associated graded
ring of $R/I$ is reduced.  Its maximal-ideal filtration is separated by
Krull's intersection theorem.  A nonzero nilpotent would have a nonzero
initial form, and a power of that initial form would be zero, which is
impossible.  The same argument applies to the completion, whose
associated graded ring is unchanged.  Thus both local rings are
reduced.  This step uses the full equality \eqref{eq05:initialideal};
smoothness of the four selected germs by itself would not exclude an
additional nilpotent structure.

Equation~\eqref{eq05:initialideal} also proves $I=F$ in the convergent
ring.  Indeed, for $g\in I$, express its lowest homogeneous part as a
homogeneous polynomial combination of $xy,xb,uv,ub$.  Subtract the same
polynomial combination of $xy,xb,uv,uH$.  The remainder is in $I$ and
has strictly higher order.  For every integer $N$ a finite iteration
therefore gives $g\in F+\mathfrak m^N$.  Krull intersection in the
Noetherian local ring $R/F$ implies $g\in F$.  Hence
\begin{equation}\label{eq05:exactideal}
 I=(xy,xb,uv,uH).
\end{equation}
No convergence of an infinite sequence of correcting coefficients is
needed for this ideal equality.

It remains to straighten $H$.  The scheme-theoretic $LL$ restriction
imposes $y=b=v=H=0$.  By \eqref{eq05:exactideal}, its local ring is
\[
 R/(y,b,v,H)\cong
 \CC\{z_1,\ldots,z_{32},x,u\}/(H|_{y=b=v=0}).
\]
This selected germ has dimension $34$, equal to the dimension of the
convergent power-series ring on the right.  That ring is a domain, so
the displayed restricting function must vanish identically.  It follows
that
\begin{equation}\label{eq05:Hdecomposition}
 H=cb+py+qv,\qquad c(0)=1,\quad p(0)=q(0)=0,
\end{equation}
with convergent coefficient functions $c,p,q$.  In particular $c$ is a
unit.  Put $b'=cb+py=H-qv$.  Its linear part is $b$, so the analytic
inverse function theorem allows $b'$ to replace $b$ as an ambient
coordinate, with the other $36$ coordinates fixed.  The identities
\[
\begin{aligned}
 xb'&=c(xb)+p(xy),& ub'&=uH-q(uv),\\
 c(xb)&=xb'-p(xy),& uH&=ub'+q(uv)
\end{aligned}
\]
give
\[
 (xy,xb,uv,uH)=(xy,xb',uv,ub').
\]
Renaming $b'$ as $b$ proves \eqref{eq05:normalform} by a convergent
coordinate change.

The change also identifies the scheme-theoretic profile restrictions.
Their ideals before the change are
\[
 (y,b,v,H),\quad(y,b,u),\quad(x,v,H),\quad(x,u),
\]
in the order $LL,LP,PL,PP$.  Equation~\eqref{eq05:Hdecomposition} and the
invertibility of $c$ send these ideals to those in
\eqref{eq05:componentideals}.  The four coordinate primes are pairwise
incomparable, and \eqref{eq05:primeintersection} is their intersection.
They are therefore exactly the irreducible components.
Finally a tangent with $x=y=1$ and $u=v=b=0$ cannot extend to second
order, because $xy$ has a nonzero coefficient of $t^2$.  Thus the full
deformation functor is obstructed.
\end{proof}

The $32$-dimensional smooth factor in \eqref{eq05:normalform} is not the
locally trivial deformation space of $X$.  Its tangent space is the
kernel of localization at the two cone points; it contains the $21$
locally trivial directions and $11$ additional independent directions
with local components at the old nodes.  The final coordinate change
has identity linear part, so this tangent interpretation is preserved.

\begin{corollary}\label{cor05:intersections}
All scheme-theoretic intersections of the components in
Theorem~\ref{thm05:germ} are smooth.  The pairwise dimensions are
\begin{center}
\begin{adjustbox}{max width=\textwidth,center}
\begin{tabular}{c c c c c c c}
\toprule
pair&$LL,LP$&$LL,PL$&$LL,PP$&$LP,PL$&$LP,PP$&$PL,PP$\\
\midrule
dimension&33&33&32&32&33&33\\
\bottomrule
\end{tabular}
\end{adjustbox}
\end{center}
Every intersection of three or four distinct components is the same
smooth $32$-dimensional germ $x=y=u=v=b=0$.
\end{corollary}

\begin{proof}
Add the corresponding coordinate ideals in
\eqref{eq05:componentideals}.  Each sum is generated by a subset of the
five coordinates, giving the dimensions stated.  No reduction of these
intersection schemes is required.
\end{proof}

\begin{corollary}\label{cor05:integrability}
A global first-order deformation of $X$ is the tangent of a convergent
curve in $\Def(X)$ if and only if its five coordinates in
\eqref{eq05:normalform} satisfy
\[
 xy=xb=uv=ub=0.
\]
Equivalently, existence of a second-order extension of a tangent is
sufficient for existence of a convergent deformation with that tangent.
Each of the four components also contains a curve whose sufficiently
small noncentral fibres have exactly one ordinary double point and no
other singularities.
\end{corollary}

\begin{proof}
The quadratic equations give the necessity for a tangent by taking the
coefficients of $t^2$.  Conversely, a vector satisfying them belongs to
one of the four linear components in \eqref{eq05:componentideals}.
Its straight line in the convergent coordinates of
\eqref{eq05:normalform} integrates it.

For the last assertion, on each $K_S$ the coefficient in
\eqref{eq05:forced} is the only identically vanishing functional among
the old nodal coefficients, the selected line coefficients, and the
three separate plane coefficients \eqref{eq05:threecoefficients}.
This follows by row reduction of $M_S$ using the printed rows.
Avoiding the kernels of the other finitely many functionals gives a
tangent on the smooth selected germ with all these other coefficients
nonzero.  Integrate it there.  Every other local singularity then
smooths for sufficiently small nonzero parameter.  On $LL$ the remaining
point is old node $13$, and on $PP$ it is old node $16$.  On $LP$ or $PL$
the remaining plane parameter is noncentral; its only possible
singularity is the ordinary node on the corresponding discriminant
line, as in Proposition~\ref{prop03:a2family}.  Were the final point also
to smooth, the resulting curve would give a smoothing of $X$, contrary
to Theorem~\ref{thm05:counterexample}.  Compactness and openness of
smoothness exclude additional singularities away from the original
ones.
\end{proof}

The assertion about a prescribed tangent does not say that every
prescribed higher jet extends.  For example, $x=y=t^2$ and
$u=v=b=z_1=\cdots=z_{32}=0$ gives a deformation over
$\CC[t]/(t^4)$ which has no extension with those coefficients over
$\CC[t]/(t^5)$, since $xy=t^4$.  Its zero tangent nevertheless integrates.
Likewise, the four minimal equations in \eqref{eq05:normalform} do not
assert that $\operatorname{Ext}^2(\Omega_X^1,\cO_X)$ has dimension four.

The finite calculations used in this section are the face and lattice
counts for \eqref{eq05:vertices}, the divisor intersections
\eqref{eq05:identities}, and the ranks of the matrices specified by
\eqref{eq05:nodes}--\eqref{eq05:rootrows}.  Reproduction files and exact
rational outputs are supplied with the paper.\footnote{The geometric
data and intersection calculations are in
\texttt{a2\_branch\_selection\_counterexample.py}; the Hodge and selected
kernel calculations are in \texttt{branch\_smoothness\_inputs.py}; and
the unrestricted kernel and ideal calculations are in
\texttt{counterexample\_deformation\_germ.py}, all under
\texttt{certificates/}.  The independent standard-library calculation
\texttt{a2\_counterexample\_rational\_replay.py} reproduces the rational
matrix identities and ranks.}
The analytic ideal equality and coordinate change are proved above;
they do not follow from a finite truncation of the deformation equations.

\section{Germs from reflexive polytopes}\label{sec:germs}

We fix notation and recall the dictionary between $2$-faces of a reflexive
polytope and the singularities of its Calabi--Yau hypersurface.  Everything in
this section is standard; \cite{Batyrev94} and \cite{CoxLittleSchenck} are the
references.

For the toric applications write $\Xh=Y_0$ for the standard smooth
crepant resolution and retain the link and branch notation of
Section~\ref{sec:deformation}.  We recall the earlier rank-one necessity
statement because it suffices for the small-polytope classification.
Theorem~\ref{thm04:criterion} supplies the stronger criterion at $A_2$
points.

\begin{recollection}[mixed necessity; {\cite[Thm.~6.1]{PaperIV}}]\label{thm:companion}
Let $X$ be a connected compact complex threefold with $\omega_X\cong\cO_X$ and
$h^1(\cO_X)=0$ whose singular points are ordinary double points and
anticanonical cones over $\dP_6$ or $\dP_7$, let $\Xh\to X$ be a crepant
resolution that is small over the nodes and the blowup of the vertex at each
cone point, for instance $Y_0$, and suppose $X$ admits a
one-parameter smoothing.  Set
\[
  K(X)=\Bigl\{\kappa\in\ker\Bigl(\bigoplus_pH_2(L_p;\QQ)\to H_2(\Xh;\QQ)\Bigr)
  \ :\ \kappa_p\in \Rsm_p\text{ at every cone point }p\Bigr\}.
\]
Then for every node $q$ there is $\kappa\in K(X)$ with $\kappa_q\neq0$; for
every $\dP_7$ point $q$ there is $\kappa\in K(X)$ whose $q$-component is a
nonzero element of the $(-2)$-line $\Rsm_q$; and for every $\dP_6$ point $q$
smoothed through the one-parameter component there is $\kappa\in K(X)$ whose
$q$-component is a nonzero element of the $A_1$ summand $\Rsm_q$.
\end{recollection}

The step from ``no one-parameter smoothing'' to ``not smoothable'' is also the
companion's, and we quote it because Theorem~\ref{thm:C} rests on it.

\begin{recollection}[{\cite[Cor.~6.7]{PaperIV}}]\label{thm:noarc}
Let $X$ be as in Recollection~\ref{thm:companion} and admit no one-parameter
smoothing.  Then no fibre of a small representative of the miniversal
deformation of $X$ is smooth; consequently $X$ is not smoothable over any base
germ, and neither is any small deformation of $X$.  Likewise no formal arc in
the miniversal deformation meets all the local smoothing loci.
\end{recollection}

The mechanism is that a smoothing over an arbitrary base germ would, by the
curve selection lemma applied to the discriminant in the miniversal base,
produce a one-parameter one; ruling out the one-parameter case therefore rules
out all of them.

We say the criterion \emph{obstructs} $X$ when some coordinate at which
Recollection~\ref{thm:companion} asserts nonvanishing vanishes identically on
$K(X)$, so that no such $\kappa$ exists and $X$ has no one-parameter smoothing;
and that it is \emph{silent} when none does.  Theorem~\ref{thm04:criterion} strengthens this necessity statement
to an equivalence by imposing all three scalar conditions at an $A_2$ point.

\begin{convention}\label{conv:setup}
$N \cong \ZZ^4$ is a lattice, $M = \Hom(N,\ZZ)$ its dual,
$\langle \cdot,\cdot\rangle$ the pairing.  $\Delta \subset N_\RR$ is a
\emph{reflexive} polytope: a lattice polytope with the origin in its interior
whose polar
$\Delta^\circ = \{u \in M_\RR : \langle u,v\rangle \ge -1 \ \forall v \in \Delta\}$
is again a lattice polytope.  $\Sigma$ is the face fan of $\Delta$, whose
cones are $\sigma_F = \Cone(F)$ for $F$ a face, and
$\PP_\Delta = \TV(\Sigma)$ the associated toric fourfold, which is Gorenstein
Fano.  For a facet $F$ we write $u_F \in M$ for the primitive inner normal,
normalized by $\langle u_F, v\rangle = -1$ for $v \in F$; reflexivity is exactly
the statement that $u_F$ is a lattice point for every facet.  $X \subset
\PP_\Delta$ denotes an anticanonical hypersurface that is regular in the sense
of \cite[Def.~3.1.1]{Batyrev94} (Batyrev's polytope there is the Newton
polytope, here $\Delta^\circ$); it is a Calabi--Yau threefold with at worst
Gorenstein canonical singularities.
\end{convention}

\subsection{The dictionary}
Let $G$ be a $2$-face of $\Delta$ and $G^* \subset \Delta^\circ$ its dual edge.
Write $\ell(G^*)$ for the lattice length of $G^*$.  A $\Delta$-regular $X$ meets
the closed orbit of $\sigma_G$, a curve in $\PP_\Delta$, in exactly
$\ell(G^*)$ points, and near each of them
\begin{equation}\label{eq:germ}
  (X, x) \;\cong\; \bigl(\Cone(S_G),\,0\bigr),
\end{equation}
the transverse affine toric threefold defined by $\sigma_G$ in the
saturated rank-three lattice
$L_G=N\cap\operatorname{span}_{\RR}(\sigma_G)$.
The surface and its polarization are determined by this height-one
cone.  For a unit square this is the cone over
$(\PP^1\times\PP^1,\cO(1,1))$; for the pentagon and hexagon below it is
the anticanonical cone over the indicated del Pezzo surface.
Here the face polygon is a section of the ray cone, rather than the
moment polygon of its polarization.

The germ \eqref{eq:germ} is a singular point of $X$ precisely when $G$ is not a
unimodular triangle, and it is an \emph{isolated} singular point precisely when
$G$ has no lattice points on the interiors of its edges, i.e.\ when every edge
of $G$ has lattice length one.  We will only meet the following three:

\begin{center}
\begin{adjustbox}{max width=\textwidth, center}
\begin{tabular}{lll}
\toprule
$G$ & $S_G$ & germ \\
\midrule
unit square & $\PP^1 \times \PP^1$ & ordinary double point (node), $xy = zw$ \\
unit-edge pentagon, one interior point & $\dP_7 = \mathrm{Bl}_2\PP^2$ & $\Cone(\dP_7)$ \\
unit-edge hexagon, one interior point & $\dP_6 = \mathrm{Bl}_3\PP^2$ & $\Cone(\dP_6)$ \\
\bottomrule
\end{tabular}
\end{adjustbox}
\end{center}

A unit-edge pentagon with exactly one interior lattice point is reflexive after
translating that point to the origin, and its face fan is the fan of $\dP_7$;
this is the sense in which the second row is a del Pezzo cone.

\begin{definition}\label{def:framework}
$\Delta$ lies in the \emph{admissible framework} if every $2$-face of $\Delta$
is either a unimodular triangle, or a unit square, or a pentagon or hexagon with
one interior point; every edge of $\Delta$ has lattice length one; and every
singular $2$-face has $\ell(G^*) = 1$.  For such a $\Delta$, $\Sing(X)$ is a
finite set with one point per singular $2$-face, and every germ is a node or an
anticanonical del Pezzo cone of degree $6$ or $7$.
\end{definition}

This is the toric class studied in the remaining sections.  Its germs
are locally smoothable by the families described in
Section~\ref{sec03:partial} and \cite{Altmann97}.  The general
criterion of Theorem~\ref{thm04:criterion} also applies outside this
toric framework.

\subsection{Nodes, coloops, and when the nodal mechanism is unavailable}
Let $X$ have nodes $x_1,\dots,x_k$, and let $\widehat X \to X$ be a small
resolution of the nodes with exceptional curves $C_1,\dots,C_k$.  We call the
matroid of the vectors $[C_j]$ the \emph{node relation system}, and, as in
\cite[\S\S1.1 and~5.3]{PaperIV}, a node $x_j$ a \emph{coloop} when $[C_j]$ lies
in no circuit; equivalently, when every relation $\sum\lambda_i[C_i]=0$ has
$\lambda_j = 0$.

\begin{remark}\label{rem:coloop}
If some node is a coloop, no relation among the nodal exceptional curve
classes alone has every coefficient nonzero, as required by
\eqref{eq:friedman}.  The threefolds of interest here are exactly those where
this happens and a del Pezzo cone point is also present, so that a smoothing, if it exists,
must move both kinds of germ at once.  We call such an $X$ \emph{genuinely
mixed}.
\end{remark}

Both $X_{19}$ and $X_9$ are genuinely mixed:  the node system of $X_{19}$ has
rank $12$ with two coloops, and that of $X_9$ has rank $2$ with both nodes
coloops.

\section{Deformations of the ambient by Minkowski decomposition}\label{sec:mech}

We now recall the mechanism.  The reader who knows \cite{IV12} can skim to
Definition~\ref{def:singledegree}, which fixes the object the two theorems
quantify over.

\subsection{Altmann's construction, affine case}
Let $\widetilde N$ be the lattice of the germ, unrelated to the ambient $N$ of
Convention~\ref{conv:setup}, and $\widetilde M$ its dual.  Let
$\delta \subset \widetilde N_\RR$ be a pointed cone and $r \in \widetilde M$
a degree.  Altmann \cite{Altmann95,Altmann97} showed that Minkowski
decompositions
\[
  \delta_r := \delta \cap [r = 1] = \Delta_0 + \dots + \Delta_l
\]
into polyhedra with common tail cone $\sigma := \delta \cap [r=0]$ give
$l$-parameter deformations of the affine toric variety $\TV(\delta)$, with
deformation parameters in degree $-r$.  For an isolated Gorenstein toric
threefold germ, $T^1$ is concentrated in the single degree $-R^*$ and has
dimension $\#(\text{vertices}) - 3$ \cite[Thm.~6.5]{Altmann95}; for the
pentagon this is $2$, and the versal base is a line with an embedded component
\cite[(7.1.4)]{Altmann95}.  The pentagon has, up to translation, exactly one
nontrivial decomposition into lattice polyhedra, namely
\begin{equation}\label{eq:pentdecomp}
  \text{pentagon} \;=\; \text{primitive segment} \;+\; \text{unimodular triangle},
\end{equation}
and the corresponding one-parameter deformation smooths $\Cone(\dP_7)$, because
the cones over a primitive segment and over a unimodular triangle are both
smooth.

\subsection{The non-affine case: complexity-one \texorpdfstring{$T$}{T}-varieties}
Our ambient $\PP_\Delta$ is complete and, for the polytopes we need, not
simplicial, so the affine statement does not apply directly.  Ilten--Vollmert
\cite{IV12} give the required globalization, in the language of polyhedral
divisors of Altmann--Hausen \cite{AH06,AHS08}.

Fix a primitive degree $R \in M$ and \emph{downgrade} the torus action of
$\PP_\Delta$ along $R$: restrict from the full four-torus to the
three-dimensional subtorus $\ker R$.  By \cite[Rem.~1.8]{IV12} the resulting
complexity-one $T$-variety is described by a divisorial fan on $\PP^1$ whose
two nontrivial \emph{slices} are the cross-sections
\begin{equation}\label{eq:slices}
  \cS^{+} := \Sigma \cap [R = 1], \qquad \cS^{-} := \Sigma \cap [R = -1],
\end{equation}
polyhedral complexes in the rank-three lattice $R^\perp \cap N$ together with
the tail fan $\Sigma \cap [R=0]$.  The maximal cells of $\cS^{\pm}$ are the
cross-sections $\sigma_F \cap [R = \pm 1]$ of the maximal cones; a cell is
\emph{bounded} exactly when $R$ is positive (resp.\ negative) on every ray of
$\sigma_F$: a point of $\sigma_F\cap[R=1]$ is $\sum_i\lambda_iv_i$ with
$\lambda_i\ge0$ and $\sum_i\lambda_i\langle R,v_i\rangle=1$, so each
$\lambda_i\le\langle R,v_i\rangle^{-1}$ when all $\langle R,v_i\rangle>0$, and
the cell is empty when all are negative.  We call the union $\cS := \cS^+ \sqcup \cS^-$ the \emph{slice
complex} at $R$.

\begin{definition}[{\cite[Def.~2.1, 2.2, 4.1]{IV12}}]\label{def:iv}
A \emph{Minkowski decomposition of the slice complex} assigns, separately on
each of the subdivisions $\cS^+$ and $\cS^-$, to every cell $C$ a decomposition
$C=C^0+C^1$ into polyhedra with the same tail cone as $C$, subject to
\begin{enumerate}
\item[(D1)] both summands carry the full tail cone, and $C^0 + C^1 = C$;
\item[(D2)] \emph{admissibility}: for each lattice character in the dual of the
  tail cone, at most one of the corresponding faces of the summands fails to
  contain a lattice point;
\item[(D3)] if $C\cap D\ne\varnothing$, then
  $(C\cap D)^i=C^i\cap D^i$;
\item[(D4)] for nonempty collections of cells
  $\mathcal J\subseteq\mathcal I$ in the same subdivision and every
  $A\subseteq\{0,1\}$,
  \[
    \sum_{i\in A}\bigcap_{C\in\mathcal I}C^i
    \ \prec\
    \sum_{i\in A}\bigcap_{C\in\mathcal J}C^i.
  \]
  The empty polyhedron is allowed as a face; the empty index set $A$ is
  tautological.
\end{enumerate}
\end{definition}

The pairwise consequence of (D4) is the one that does the work in
Theorem~\ref{thm:A}: take $\mathcal J=\{\lambda\}$,
$\mathcal I=\{\kappa,\lambda\}$ and $A=\{i\}$ when
$\kappa\prec\lambda$.  Together with (D3), this says that the decomposition
induced on a face is a face of the decomposition of the cell above it.  In
particular, if a cell admits only homothetic summands, then so does every face
of it.

\begin{recollection}[{\cite[Thm.~4.4]{IV12}}]\label{thm:iv44}
A Minkowski decomposition of the slice complex in the sense of
Definition~\ref{def:iv} determines a flat family
$\pi \colon \cP \to B$ over a smooth base germ $B$, with one parameter for each
decomposed slice, with $\cP_0 = \PP_\Delta$ and with $T$-action on every fibre.
\end{recollection}

\begin{definition}\label{def:singledegree}
The family of Recollection~\ref{thm:iv44}, for a fixed primitive $R \in M$ and a
fixed decomposition of the slice complex at $R$, restricted to a curve germ
through the origin and assumed proper over it, is the \emph{single-degree
family at $R$}.  When only one slice is decomposed, properness holds
automatically \cite[Thm.~3.2]{HI11}; for decompositions of both slices it is
asserted in \cite[Rem.~3.3]{HI11}, and we include it in the definition.  Its fibre over $b\in B$ is denoted $\cP_b$, and the base
parameter is written $b$ throughout, the letter $t$ being reserved for edge
dilations.
\end{definition}

\begin{recollection}[{\cite[Cor.~2.12]{IV12}}]\label{thm:cor212}
Suppose the polyhedral divisor at a germ has \emph{affine locus}.  Then the
singularities of the general fibre at that germ are the cones over the summands
$C^0$, $C^1$.  In particular the germ is smoothed if and only if each summand
cone is smooth; for the bounded lattice summands used below, this means a point
or a unimodular simplex, in particular a primitive segment in dimension one.
\end{recollection}

\begin{recollection}[{\cite[Rem.~2.13]{IV12}}]\label{thm:rem213}
If a polyhedral divisor has \emph{complete locus} and its associated affine
$T$-variety is singular, no $T$-deformation obtained by decomposing that
polyhedral divisor is a smoothing.  Thus the general fibre is singular; the
remark does not assert that its analytic germ is unchanged.
\end{recollection}

The following two consequences pass from a singular ambient fibre to
its anticanonical hypersurfaces.

\begin{proposition}[Relative anticanonical ampleness]\label{prop:relativefano}
Let $\pi\colon\cP\to B$ be a single-degree family of $\PP_\Delta$, with
$\Delta$ reflexive and $B$ a curve germ.  After shrinking $B$, the morphism
$\pi$ is proper and Gorenstein of relative dimension $4$, and
\[
  \cL:=\omega_{\cP/B}^{-1}
\]
is $\pi$-ample and satisfies
$\cL|_{\cP_b}\cong\omega_{\cP_b}^{-1}$ for every $b\in B$.
\end{proposition}

\begin{proof}
The morphism $\pi$ is proper by Definition~\ref{def:singledegree}; flatness is
Recollection~\ref{thm:iv44}.  The special fibre $\PP_\Delta$ is Gorenstein.  The
fibrewise Gorenstein locus of a flat morphism of finite presentation is open
\cite[Tag~0C09]{Stacks}; it contains the whole proper fibre $\cP_0$, so
properness permits us to shrink $B$ until every fibre is Gorenstein.  Thus
$\pi$ is Gorenstein, and relative duality gives the asserted invertible
dualizing sheaf and its base-change isomorphisms \cite{Conrad00}.  Finally,
$\cL_0\cong\cO_{\PP_\Delta}(-K_{\PP_\Delta})$ is ample because $\Delta$ is
reflexive.  Relative ampleness is open for a line bundle under a proper
morphism, so another shrinking makes $\cL$ relatively ample
\cite[Tag~0D2S]{Stacks}.
\end{proof}

\begin{lemma}[Ample sections and singular curves]
\label{lem:singcurve}
Let $Z$ be a normal projective variety, let $L$ be ample, and suppose
$\Sing(Z)$ contains a complete curve $C$.  Every effective Cartier divisor
$D\in|L|$ is singular.
\end{lemma}

\begin{proof}
Pass to the normalization $\widetilde C$ of an irreducible component of $C$.
If the section defining $D$ vanishes identically on $C$, then $D$ contains
$C$.  Otherwise it restricts to a nonzero section of the positive-degree line
bundle $L|_{\widetilde C}$, and hence has a zero.  In either case choose a
closed point $x\in C\cap D$.

Put $A=\cO_{Z,x}$ and let $g\in A$ define $D$.  Since $D$ is Cartier, $g$ is a
nonzerodivisor.  If $A/(g)$ were regular, then $A$ would be regular: indeed,
if $g\notin\mathfrak m_A^2$, both embedding dimension and dimension rise by
one from $A/(g)$ to $A$; while if $g\in\mathfrak m_A^2$, regularity of
$A/(g)$ would give
$\operatorname{edim}A=\dim(A/(g))=\dim A-1$, impossible.  This contradicts
$x\in\Sing(Z)$, so $D$ is singular at $x$.
\end{proof}

For a cone $\sigma_G$ in the face fan the dichotomy is combinatorial.  The
locus is affine when $\langle R, \cdot\rangle$ is nonnegative or nonpositive
on the rays of $\sigma_G$ (we then say $\sigma_G$ is \emph{one-sided} at
$R$), and complete when $\sigma_G$ has rays strictly on both sides of
$R^\perp$.  If $R$ vanishes on the whole cone, neither slice acts on its
chart.  The following lemma gives the resulting persistence statement.

\section{Persistent ambient singularities and Minkowski rigidity}\label{sec:obstructions}

Throughout this section $\Delta$ is in the admissible framework and $G$ is a
singular $2$-face.

\subsection{Degrees taking both signs on a cone}

\begin{lemma}[Persistence for degrees taking both signs]\label{lem:reach}
Let $R \in M$ be primitive.  If $\sigma_G$ has rays strictly on both sides of
$R^{\perp}$, then the corresponding affine chart of the general fibre of every
single-degree family at $R$ has positive-dimensional singular locus.
\end{lemma}

\begin{proof}
The rays of $\sigma_G$ are the vertices of $G$.  If $\langle R,\cdot\rangle$
takes both signs on them, then in the downgrade along $R$ the polyhedral divisor
attached to $\sigma_G$ has nonempty coefficients at both $0$ and $\infty$, hence
complete locus.  By Definition~\ref{def:framework} the germ is a node or a del
Pezzo cone, in particular singular, so Recollection~\ref{thm:rem213} applies.

It remains to justify the asserted dimension.  The tail cone
$\tau=\sigma_G\cap R^\perp$ has dimension $2$ in the rank-three acting lattice
$N\cap R^\perp$.  Choose a nonzero character $u$ annihilating $\tau$.  Both
$u$ and $-u$ lie in $\tau^\vee$.  For the proper polyhedral divisor $\cD_b$ of
the general fibre, semiampleness gives
$\deg\cD_b(u),\deg\cD_b(-u)\ge0$.  Coefficient by coefficient,
\[
 \min_{v\in(\cD_b)_y}\langle u,v\rangle
 +\min_{v\in(\cD_b)_y}\langle-u,v\rangle\le0,
\]
so both degrees are zero, equality holds at every coefficient, and
$\cD_b(-u)=-\cD_b(u)$.  After multiplying $u$ to clear denominators, this
degree-zero semiample divisor on $\PP^1$ is principal.  The coordinate ring
therefore contains a nonconstant homogeneous unit $f\chi^u$ and its inverse.
No point can be fixed by the acting torus while this unit is nonzero there.
The singular locus is torus-invariant, so every one of its nonempty orbits has
positive dimension.

Altmann's theorem gives an independent check of the first step.  The chart is
$Y_G\times\mathbb G_m$, whose $T^1$ is concentrated in degrees restricting to
$-R_G^*$ on $L_G$ \cite[Thm.~(6.5)]{Altmann95}.  The parameters displacing
the two slices have opposite full-torus weights, so every homogeneous term
in the formal expansion has degree $jR$ for some $j\in\ZZ$.  No such degree
restricts to $-R_G^*$, because $R$ takes both signs on the rays of $\sigma_G$.
Consequently, at each formal order the possible local deformation class
vanishes on this chart.
\end{proof}

\begin{corollary}\label{cor:reach}
If for every primitive $R \in M$ some singular $2$-face of $\Delta$ is
two-sided at $R$, then no relative anticanonical effective Cartier divisor in
any single-degree family has smooth general fibre.
\end{corollary}

\begin{proof}
By Lemma~\ref{lem:reach}, the singular locus of the general ambient fibre has
a positive-dimensional irreducible component.  Its closure in the projective
fibre is a positive-dimensional projective subvariety contained in the
singular locus, and hence contains a complete curve.  Apply
Proposition~\ref{prop:relativefano} and Lemma~\ref{lem:singcurve}.
\end{proof}

\begin{remark}\label{rem:reachscope}
The quantifier in Corollary~\ref{cor:reach} must range over \emph{all} primitive
degrees, not over the degrees $R = u_F$ dual to a facet $F$.  With the
normalization $\langle u_F,v\rangle=-1$ these are exactly the vertices of the
polar polytope $\Delta^\circ$, and we call them \emph{vertex degrees} for that
reason; at such an $R$ the cell of $F$ is the unique maximal bounded cell of
the level-$(-1)$ slice, and is $F$ itself, which is why we call it the
\emph{lattice facet}.  On
$\Delta_{19}$ the vertex degrees leave at most $13$ of the $15$ singular
$2$-faces one-sided, while a general primitive degree, for instance
$R = (-3,0,0,1)$, leaves all $15$ one-sided.  Thus the sign condition in Lemma~\ref{lem:reach} does not obstruct
$X_{19}$; its ambient obstruction follows instead from Minkowski
rigidity, as proved in Theorem~\ref{thm:A}.  A test restricted to
vertex degrees would miss this distinction.
\end{remark}

\subsection{Restriction to the local Gorenstein degree}

Fix a unit-edge pentagonal $2$-face $G$ and put
\[
 L_G=N\cap\operatorname{span}_{\RR}(\sigma_G),\qquad
 Y_G=\TV_{L_G}(\sigma_G).
\]
Then $L_G$ is saturated of rank $3$, $Y_G$ is the isolated Gorenstein
threefold cone over $\dP_7$, and a choice of splitting gives
\[
 U_{\sigma_G}\cong Y_G\times\mathbb G_m.
\]
Let $R_G^*\in L_G^\vee$ be the Gorenstein degree, characterized by
$\langle R_G^*,v\rangle=1$ on every ray of $\sigma_G$.

\begin{definition}\label{def:crosscut}
If $r\in M$ is nonnegative and nonzero on $\sigma_G$, its \emph{crosscut} is
\[
 Q_G(r):=\sigma_G\cap[r=1].
\]
It is a polyhedron with tail $\sigma_G\cap r^\perp$.  When $r$ is positive on
all five rays it is the rational polygon
$\conv\{v/\langle r,v\rangle:v\in V(G)\}$.
\end{definition}

\begin{lemma}[Restriction to the Gorenstein degree]\label{lem:degreegate}
Let $R\in M$ be primitive and suppose $\sigma_G$ is one-sided at $R$ but
$R|_{L_G}\ne0$.  Choose $\epsilon\in\{1,-1\}$ so that
$r=\epsilon R$ is nonnegative on $\sigma_G$.  For every single-degree family
at $R$:
\begin{enumerate}
\item if $r|_{L_G}\ne R_G^*$, the induced deformation of
  $U_{\sigma_G}$ is trivial;
\item if $r|_{L_G}=R_G^*$, then
  $\langle r,v\rangle=1$ for all $v\in V(G)$ and every admissible induced
  decomposition of $G=Q_G(r)$ is a decomposition into lattice polygons.
  If that decomposition is trivial, the induced deformation of
  $U_{\sigma_G}$ is trivial.
\end{enumerate}
Thus a one-sided single-degree family can move the transverse $\dP_7$ cone only
on the affine lattice
\[
 \cA_G:=\{r\in M:\langle r,v\rangle=1\text{ for every }v\in V(G)\}.
\]
\end{lemma}

\begin{proof}
The deformation datum lives in the cone
$\sigma_G\subset (L_G)_\RR$.  After choosing the splitting
$U_{\sigma_G}=Y_G\times\mathbb G_m$, the restriction of the
Ilten--Vollmert construction to this chart is the Laurent extension of the
homogeneous Altmann deformation of $Y_G$ associated with the same crosscut
\cite[\S3]{IV12}.  Its Kodaira--Spencer class in $T^1(Y_G)$ has degree
$-r|_{L_G}$ \cite[(5.3)]{Altmann95}.  Altmann's theorem for an isolated
three-dimensional Gorenstein toric singularity says that $T^1(Y_G)$ is
concentrated in degree $-R_G^*$ \cite[Thm.~(6.5)]{Altmann95}.  Hence the
Kodaira--Spencer map is zero unless $r|_{L_G}=R_G^*$.

Zero Kodaira--Spencer class is enough here, although it is not enough for an
arbitrary deformation.  Altmann's classification of the transverse admissible
deformation datum with zero Kodaira--Spencer map gives only types A and B
\cite[Prop.~(5.6)]{Altmann95}.  Type A yields a trivial family
\cite[(5.7)]{Altmann95}; and a nontrivial type B datum does not exist when the
toric variety is smooth in codimension two \cite[(5.8)]{Altmann95}.
The transverse variety $Y_G$ is smooth in codimension two because its
singularity is isolated.  Thus its Altmann family is trivial, and so is its
Laurent extension to $U_{\sigma_G}$.  This proves (1).

In the remaining degree the crosscut is $G$ itself.  Altmann's Gorenstein
specialization says precisely that admissibility is equivalent to all
summands being lattice polytopes \cite[Thm.~(6.6)]{Altmann95}.  A trivial
decomposition is type A, so (5.7) again makes the induced family trivial.
\end{proof}

\subsection{Locking: a combinatorial rigidity criterion}

A Minkowski summand cone is computed from one cross-section.  We use a
rigidity criterion depending only on the face lattice, which therefore applies
to all bounded cross-sections at once.

\begin{definition}\label{def:rigid}
A positive-dimensional polytope $C$ is \emph{Minkowski-rigid} if every
Minkowski summand is a homothet of $C$ up to translation.  Equivalently,
its \emph{Minkowski summand cone} has dimension one.  This cone is the
solution cone of the \emph{edge-cycle system}: one variable $t_e\ge0$
for each edge, and for each $2$-face with cyclic edge vectors
$e_1,\ldots,e_k$ the closing condition
\begin{equation}\label{eq:closing}
  \sum_{i=1}^k t_i e_i=0.
\end{equation}
\end{definition}

The all-ones vector represents homothety and satisfies every closing
condition.  Thus the summand cone has dimension at least one.
Only bounded positive-dimensional cross-sections are used below.

\begin{remark}\label{rem:dimmoves}
A computation of the Minkowski summand cone at one degree concerns that
cross-section only, and the degrees form an infinite set.  The criterion below
depends only on the face lattice of the facet and on which of its $2$-faces are
triangles, so by Lemma~\ref{lem:crosssection} it applies uniformly to all
bounded cross-sections.
\end{remark}

The reason such a criterion exists is elementary.

\begin{lemma}[Cross-section]\label{lem:crosssection}
Let $\sigma \subset N_\RR$ be a pointed cone of dimension $n$ and let
$R \in \operatorname{int}\sigma^\vee$.  Then $C_R := \sigma \cap [R=1]$ is a
polytope of dimension $n-1$, its faces are the cross-sections of the nonzero
faces of $\sigma$, and its face lattice is that of $\sigma$ with the apex
removed.  In particular, for any two $R, R' \in \operatorname{int}\sigma^\vee$
the polytopes $C_R$ and $C_{R'}$ are combinatorially isomorphic, by the
isomorphism matching cross-sections of the same face.
\end{lemma}

\begin{proof}
$R$ is positive on $\sigma \smallsetminus \{0\}$, so every ray meets $[R=1]$ in
exactly one point and $C_R$ is bounded; the correspondence
$\sigma' \mapsto \sigma' \cap [R=1]$ between nonzero faces of $\sigma$ and faces
of $C_R$ is inclusion-preserving in both directions.
\end{proof}

The edge \emph{directions} of $C_R$ do change with $R$: the vertex of $C_R$ on
the ray through $v$ is $v/\langle R,v\rangle$, so an edge from $v$ to $v'$ has
direction $v'/\langle R,v'\rangle - v/\langle R,v\rangle$.  The point is that
the forcing rules below never use the directions, only the following two facts,
both of which hold in every cross-section:
\begin{enumerate}
\item[(a)] the edge vectors of a bounded $2$-face sum to zero around the cycle;
\item[(b)] two consecutive edges of a convex polygon are linearly independent.
\end{enumerate}

\begin{proposition}[The forcing rules]\label{prop:rules}
Let $C$ be a polyhedron, $f$ a bounded $2$-face of $C$ with edge vectors
$e_1,\dots,e_k$ in cyclic order, and $t_1,\dots,t_k$ a solution of the
edge-cycle system.
\begin{enumerate}
\item \textbf{Triangle.} If $k = 3$ then $t_1 = t_2 = t_3$.
\item \textbf{Single.} If $t_i = \lambda$ for every $i$ except one, say $i=k$,
  then $t_k = \lambda$.
\item \textbf{Adjacent.} If $t_i = \lambda$ for every $i$ except two
  \emph{consecutive} indices, say $i = 1, 2$, then $t_1 = t_2 = \lambda$.
\end{enumerate}
None of the three refers to the lattice, to the degree, or to the edge
directions beyond (b).
\end{proposition}

\begin{proof}
(1) By (a), $e_3 = -e_1 - e_2$, so \eqref{eq:closing} reads
$(t_1-t_3)e_1 + (t_2-t_3)e_2 = 0$, and $e_1, e_2$ are independent by (b).
(2) By (a), $\sum_{i<k} e_i = -e_k$, so \eqref{eq:closing} reads
$(t_k - \lambda)e_k = 0$ and $e_k \ne 0$.
(3) By (a), $\sum_{i \ge 3} e_i = -(e_1+e_2)$, so \eqref{eq:closing} reads
$(t_1-\lambda)e_1 + (t_2-\lambda)e_2 = 0$, and $e_1, e_2$ are consecutive,
hence independent by (b).
\end{proof}

\begin{remark}\label{rem:adjacent}
The hypothesis that the edges in (3) are consecutive is necessary.
Two opposite
edges of a quadrilateral may be parallel, and then the same computation gives
only one equation for two unknowns.  Applying rule (3) to opposite
pairs would incorrectly imply rigidity of the cube; the cube is the Minkowski sum of three
segments and has Minkowski summand cone of dimension three.
\end{remark}

\begin{definition}\label{def:locked}
Let $F$ be a $3$-polytope and $S$ a set of vertices of $F$.  Maintain a
partition of the edge set of $F$, initially into singletons, and call two edges
\emph{tied} when they lie in a common class.  Using only those $2$-faces all of
whose vertices lie in $S$, apply the following three rules until none applies.
\begin{enumerate}
\item[(T)] \emph{Triangle.}  Merge the three edges of a triangular $2$-face into
  one class.
\item[(S)] \emph{Single edge.}  If all but one edge of a $2$-face are tied to
  one another, merge the remaining edge into their class.
\item[(A)] \emph{Adjacent pair.}  If all but two edges of a $2$-face are tied to
  one another and those two are consecutive, merge both into their class.
\end{enumerate}
Each application either changes nothing or decreases the number of classes, so
the propagation terminates.  Its outcome does not depend on the order in which
the rules are applied: each rule's hypothesis is preserved under coarsening the
partition, so the three rules define a monotone closure operator and the
terminal partition is its least fixed point.  Say that a set $E$ of edges is \emph{locked in $(F,S)$} if at
termination all of $E$ lies in a single class; that a facet $F$ is
\emph{locked} if its whole edge set is locked in $(F,V(F))$; and that a
$2$-face $G$ lying in the facets $F_1,F_2$ is \emph{locked} if both $F_1$ and
$F_2$ are locked.  This last is hypothesis~{\rm(L)} of
Theorem~\ref{thm:Aprime}, and it is what the census of
Section~\ref{sec:census} counts.
\end{definition}

The requirement in (S) and (A) that the other edges be tied \emph{to one
another}, and not merely each already forced, is what makes the rules valid:
what they propagate is the value of a common dilation parameter, and two edges
carrying forced but unequal parameters license nothing.  A partition, rather
than a set of forced edges, is what records that.

The propagation needs a seed, and (T) is the only rule that fires from
singletons, so a $3$-polytope with no triangular $2$-face is never locked.

\begin{proposition}\label{prop:locked}
Let $F$ be a facet of $\Delta$ and $R \in M$ with
$\langle R, \cdot\rangle > 0$ on every vertex of $F$.  If $F$ is locked then
the cell $\sigma_F \cap [R=1]$ is Minkowski-rigid.
\end{proposition}

\begin{proof}
By Lemma~\ref{lem:crosssection} the cell is a polytope combinatorially
isomorphic to $F$, with the same triangular $2$-faces and the same
edge-adjacency.  By Proposition~\ref{prop:rules} every step of the propagation
is a valid deduction about any solution of the edge-cycle system, so all
$t_e$ are equal and the Minkowski summand cone has dimension one.
\end{proof}

\begin{corollary}\label{cor:lockedrel}
Let $G$ be a unit-edge pentagonal $2$-face, let $r\in\cA_G$, and let $F$ be a
facet containing $G$.  If $r$ is positive on every vertex of $F$ and $F$ is
locked, then every decomposition of the slice complex induces the trivial
lattice Minkowski decomposition on $G$.  Consequently the induced deformation
of $U_{\sigma_G}$ is trivial.
\end{corollary}

\begin{proof}
Proposition~\ref{prop:locked} makes
$C=\sigma_F\cap[r=1]$ Minkowski-rigid, so in a decomposition
$C=C^0+C^1$ we have $C^i=\lambda_iC+t_i$ with
$\lambda_0+\lambda_1=1$.  Condition (D3) and the pairwise consequence of (D4)
then give
$G^i=\lambda_iG+t'_i$.  Lemma~\ref{lem:degreegate} says that both $G^i$ are
lattice polytopes.  Every edge of $G$ is primitive, so the lattice length of a
corresponding edge of $G^i$ is $\lambda_i$; hence
$\lambda_i\in\ZZ_{\ge0}$.  Their sum is one, so one summand is a translate of
$G$ and the other is a point.  This is the trivial decomposition, and the last
assertion is Lemma~\ref{lem:degreegate}(2).
\end{proof}

\subsection{The sign lemma}

The affine lattice $\cA_G$ is a line.  The following determines how the two
facets containing $G$ behave along it; it is a statement about reflexive
polytopes and does not involve deformation theory.

\begin{lemma}[Sign lemma]\label{lem:sign}
Let $\Delta$ be reflexive, $G$ a $2$-face, and $F_1, F_2$ the two facets
containing $G$.  The annihilator of $\mathrm{span}(G)$ in $M$ has rank one; let
$d$ generate it, normalized so that $u_{F_1} - u_{F_2} = c\,d$ with $c > 0$.
Then
\[
  d(w) \le -\tfrac1c < 0 \ \text{ for } w \in V(F_1)\smallsetminus V(G),
  \qquad
  d(w') \ge \tfrac1c > 0 \ \text{ for } w' \in V(F_2)\smallsetminus V(G).
\]
\end{lemma}

\begin{proof}
$\mathrm{span}(G)$ is three-dimensional, so its annihilator has rank one.  Both
$u_{F_1}$ and $u_{F_2}$ take the value $-1$ on $G$, so their difference kills
$\mathrm{span}(G)$ and is an integer multiple of $d$; this defines $c \ne 0$ and
fixes the sign of $d$.  For $w \in V(F_1) \smallsetminus V(G)$ we have
$u_{F_1}(w) = -1$, and $w \notin F_2$ with $\Delta$ reflexive gives
$u_{F_2}(w) \ge 0$, whence $c\,d(w) = -1 - u_{F_2}(w) \le -1$.  Symmetrically
$c\,d(w') = u_{F_1}(w') + 1 \ge 1$.
\end{proof}

\begin{corollary}\label{cor:thresholds}
With the notation of Lemma~\ref{lem:sign}, suppose $R_0\in M_\RR$ is
strictly positive on every vertex of $G$, and put
$R_\tau=R_0+\tau d$ for $\tau\in\RR$.  The nonempty positive-slice cell
$\sigma_{F_1}\cap[R_\tau=1]$ is bounded exactly when $\tau<T_1$, and
$\sigma_{F_2}\cap[R_\tau=1]$ is bounded exactly when $\tau>T_2$, where
\[
 T_1=\min_{w\in V(F_1)\setminus V(G)}\frac{-R_0(w)}{d(w)},
 \qquad
 T_2=\max_{w\in V(F_2)\setminus V(G)}\frac{-R_0(w)}{d(w)}.
\]
If $m$ is the number of distinct thresholds $-R_0(w)/d(w)$ for
$w\in(V(F_1)\cup V(F_2))\setminus V(G)$, then the sign pattern on
$V(F_1)\cup V(F_2)$ takes at most $2m+1$ values, including the patterns
at the thresholds.  On the complementary open intervals it takes at
most $m+1$ values.
\end{corollary}

\begin{proof}
The pairings on $G$ are constant and positive.  Each slice is therefore
nonempty, and it is bounded if and only if the functional is strictly
positive on every ray of its facet cone.  Lemma~\ref{lem:sign} turns
these inequalities into the two stated threshold inequalities.
All signs are constant between successive thresholds.  The $m$
threshold points can contribute $m$ additional patterns, giving the
bound $2m+1$.  In particular, restricting to lattice degrees does not
justify omitting the patterns with a zero pairing.
\end{proof}

\section{Obstructions to single-degree ambient smoothings}\label{sec:thmA}

We can now state the obstruction for a class of polytopes rather than for one.

\begin{theorem}\label{thm:Aprime}
Let $\Delta$ be a reflexive $4$-polytope in the admissible framework and let
$G$ be a unit-edge pentagonal $2$-face, with containing facets $F_1,F_2$.
Suppose that
\begin{enumerate}
\item[\rm(L)] both $F_1$ and $F_2$ are locked;
\item[\rm(C)] for every $r\in\cA_G$, the functional $r$ is positive on every
  vertex of $F_1$ or on every vertex of $F_2$.
\end{enumerate}
Then, for every primitive $R\in M$, the general ambient fibre of every
single-degree family at $R$ has positive-dimensional singular locus.  After
shrinking the base, no relative anticanonical effective Cartier divisor in
such a family has smooth general fibre.

The hypothesis is decidable: $\cA_G$ is an affine lattice line, and condition
{\rm(C)} is a finite comparison of the thresholds in
Corollary~\ref{cor:thresholds}.
\end{theorem}

\begin{proof}
Let $R$ be primitive.  If $\sigma_G$ is two-sided at $R$, Lemma~\ref{lem:reach}
already gives a singular curve in the corresponding chart.  If $R$ vanishes
on $L_G$, the coefficients of the polyhedral divisor of $\sigma_G$ are
empty or equal to its tail, and the decomposition is trivial on them; the
corresponding chart of the general fibre is the product of $Y_G$ with an open
curve, and is singular along a curve.

It remains to consider the one-sided case with nonzero restriction.  Choose
$r=R$ or $-R$ nonnegative on $\sigma_G$.  If
$r|_{L_G}\ne R_G^*$, Lemma~\ref{lem:degreegate}(1) makes the local family
trivial.  If $r|_{L_G}=R_G^*$, then $r\in\cA_G$; by (C), one containing facet
is positive at $r$, and by (L) it is locked.  Corollary~\ref{cor:lockedrel}
again makes the local family trivial.  Thus in every one-sided case the general
fibre contains the singular curve $\Sing(Y_G)\times\mathbb G_m$ in this chart.

Its closure in the proper ambient fibre is a complete curve contained in the
singular locus in the one-sided case.  In the two-sided case,
Lemma~\ref{lem:reach} gives a positive-dimensional singular subvariety; its
projective closure contains a complete curve in the singular locus, as in
Corollary~\ref{cor:reach}.  Proposition~\ref{prop:relativefano} and
Lemma~\ref{lem:singcurve} now give the assertion about anticanonical Cartier
divisors.
\end{proof}

\subsection{The ambient obstruction for \texorpdfstring{$X_{19}$}{X19}}
Let $\Delta_{19}$ be the polytope of \cite[\S5.3]{PaperIII}:
\begin{equation}\label{eq:V19}
\begin{aligned}
\Delta_{19} = \conv\bigl\{
&(1,0,0,0),\,(0,1,0,0),\,(0,0,1,0),\,(0,0,-1,0),\,(0,-1,0,0),\,(1,-1,-1,0),\\
&(0,0,0,1),\,(-1,1,1,-1),\,(0,-1,0,1),\,(0,0,-1,1),\,(-1,1,0,-1),\,(-1,0,1,-1),\\
&(-1,0,1,0),\,(-1,1,0,0),\,(0,-1,-1,0),\,(0,-1,-1,1),\,(-1,0,0,-1),\\
&(-2,1,1,-1),\,(-1,0,0,1)\bigr\},
\end{aligned}
\end{equation}
and label the displayed vertices $v_0,\ldots,v_{18}$ in that order.  This is
a reflexive polytope with $19$ vertices, $27$ facets and $64$ edges, all of
lattice length one.  It lies in the admissible framework: its singular
$2$-faces are fourteen unit squares and one pentagon
$P = \{v_{14},v_{15},v_{16},v_{17},v_{18}\}$ with one interior point, each with
dual edge of length one.  So $X_{19}$ has fourteen nodes and one $\dP_7$ cone
point, all isolated.  The pentagon lies in exactly two facets, which we
denote by $F_{21}$ and $F_{24}$.  Their
supporting equations $\langle u,\cdot\rangle=1$ and vertex sets are
\[
\begin{aligned}
F_{21}&:\ u=(-1,0,-1,0), &&\{v_3,v_9,v_{10},v_{13},v_{14},v_{15},v_{16},v_{17},v_{18}\},\\
F_{24}&:\ u=(-1,-1,0,0), &&\{v_4,v_8,v_{11},v_{12},v_{14},v_{15},v_{16},v_{17},v_{18}\},
\end{aligned}
\]

\begin{theorem}\label{thm:A}
$\Delta_{19}$ satisfies the hypothesis of Theorem~\ref{thm:Aprime} at $P$.
Consequently, for every primitive $R\in M$ the general fibre of every
single-degree family of $\PP_{\Delta_{19}}$ at $R$ has a singular curve.  No
relative anticanonical effective Cartier divisor in such a single-degree
ambient family, with central fibre $X_{19}$, has smooth general fibre.
\end{theorem}

\begin{proof}
Both facets are locked: each has $9$ vertices, $16$ edges and nine $2$-faces,
of degrees $(3,3,3,3,3,4,4,4,5)$ summing to $32=2\cdot16$.
For $F_{21}$, the five triangular faces, written by vertex indices, are
\[
 \{9,13,18\},\quad\{9,15,18\},\quad\{10,13,17\},\quad
 \{10,16,17\},\quad\{13,17,18\}.
\]
Rule (T) ties their eleven edges to one another.  Rule (A) on
$\{3,9,10,13\}$ then ties the edges $\{3,9\}$ and $\{3,10\}$ to this
class.  Applying (A) on $\{3,9,14,15\}$ adds $\{3,14\}$ and
$\{14,15\}$; applying (S) on $\{3,10,14,16\}$ adds $\{14,16\}$.
These are all sixteen edges.  Interchanging the second and third
coordinates preserves $\Delta_{19}$ and carries $F_{21}$ to $F_{24}$,
so the same argument applies there.

It remains to check (C), and here there is no enumeration.  Solving
$\langle r,v_i\rangle=1$ for $i=14,\ldots,18$ gives exactly
\[
  \cA_P=\{r_s=(-1,s,-1-s,0):s\in\ZZ\}.
\]
The five pentagon vertices pair to $1$.  Each of the other four vertices of
$F_{21}$ pairs to $s+1$, and each of the other four vertices of $F_{24}$ pairs
to $-s$.  Hence $r_s$ is positive on every vertex of $F_{21}$ when $s\ge0$ and
on every vertex of $F_{24}$ when $s\le-1$.  These two ranges cover every
integer $s$, proving (C).  Theorem~\ref{thm:Aprime} completes the proof.
\end{proof}

\begin{corollary}\label{cor:x19smooth}
The full semiuniversal deformation space of a $\Delta_{19}$-regular
anticanonical hypersurface $X_{19}$ is smooth of dimension $30$.
The threefold $X_{19}$ has a convergent smoothing.
\end{corollary}

\begin{proof}
We give the finite relation calculation explicitly.  Order its nodal
faces by the following vertex sets, using \eqref{eq:V19}:
\[
\begin{array}{rrrr@{\qquad}rrrr@{\qquad}rrrr}
0&1&3&5&0&2&4&5&6&8&9&15\\
7&10&11&16&1&6&13&18&1&7&13&17\\
2&6&12&18&2&7&12&17&3&9&10&13\\
4&8&11&12&3&9&14&15&3&10&14&16\\
4&8&14&15&4&11&14&16&&&&
\end{array}
\]
For each set take the primitive relation on its four rays whose
first nonzero coefficient is positive, and extend it by zero on the
other rays.  Let $N$ be the resulting $14\times19$ matrix.  The
degree-seven root has divisor row (the pentagon has cyclic order
$(v_{17},v_{16},v_{14},v_{15},v_{18})$, with $v_{14}$ the middle $(-1)$-curve)
\[
 r=(0,0,0,0,0,0,0,0,0,0,0,0,0,0,0,1,-1,1,-1).
\]
Exact row reduction gives
\[
 \rk N=12,\qquad \rk\begin{pmatrix}N\\r\end{pmatrix}=12.
\]
A relation in the latter matrix, nonzero in all fifteen coordinates, is
\[
 (-2,2,-1,1,1,-3,-2,4,-2,2,3,1,-1,-3,-1).
\]
Replace $r$ by a basis of the full two-dimensional relation space
on the pentagon rays $v_{14},\ldots,v_{18}$.  The resulting
$16\times19$ matrix has rank $13$.  Thus both the selected and the
full local relation kernels have dimension three; they are equal.
These matrices are reconstructed from the printed vertices and face
sets in the exact computation listed in Appendix~\ref{app:cert}.

The rows detect homology relations.  Indeed, the correction term for
non-toric divisor classes on a smooth crepant resolution vanishes:
every singular two-face has a dual edge of lattice length one and
every other two-face is a unimodular triangle.  Toric divisors
therefore span $H^2$ by Batyrev's formula \cite{Batyrev94}.  The
inserted divisor over the pentagon pairs to zero with $K_E^\perp$.
Consequently the full link kernel $K$ is already $K_S$, and the
displayed vector proves smoothing by Theorem~\ref{thm04:criterion}.

For the dimension, the polar has $32$ lattice points, no
facet-interior lattice points, and zero correction term in Batyrev's
$h^{2,1}$ formula, since the original edges are primitive.  Hence
$h^{2,1}(Y_0)=32-5=27$.  One also has
\[
 f_{0*}\Theta_{Y_0}=\Theta_{X_{19}},\qquad
 R^1f_{0*}\Theta_{Y_0}=0.
\]
Vector fields lift to the vertex blow-up of a cone and to a chosen
nodal small resolution.  For the higher direct image, formal
functions reduces the degree-seven assertion to the vanishings
\[
 H^1(E,\Theta_E\otimes\cO_E(-kK_E))=0,\qquad
 H^1(E,\cO_E((1-k)K_E))=0\qquad(k\ge0).
\]
The first follows from
$\Theta_E\cong\Omega^1_E\otimes\cO_E(-K_E)$ and toric Bott
vanishing.  The second follows from rationality and Serre duality
for $k=0,1$, and Kodaira vanishing for $k\ge2$.  At a node the
graded terms are copies of $\cO(k+2)$ and $\cO(k-1)$ on $\PP^1$,
whose first cohomology vanishes.  Leray and contraction with the
volume form give $h^1(X_{19},\Theta_{X_{19}})=27$.
Theorem~\ref{thm:D} now gives $\dim\mathbb T^1_{X_{19}}=27+3=30$.

The construction in the proof of Theorem~\ref{thm04:smoothbase}
gives a smooth formal source whose blow-down differential covers
$V_S$.  Here $V_S=\mathbb T^1_{X_{19}}$ because $K_S=K$.
The same independent-coordinate argument therefore makes the
\emph{full} deformation ring regular, not only its selected
quotient.  Regularity of the completed analytic ring implies
smoothness of the analytic germ, with the dimension just computed.
\end{proof}

Theorem~\ref{thm:A} and Corollary~\ref{cor:x19smooth} thus concern
different deformation constructions: $X_{19}$ is smoothable, while
every single-degree ambient family in the stated class retains a
singular anticanonical fibre.

\begin{remark}\label{rem:whatfailedfirst}
The affine line in the proof contains every one-sided degree that can
move the transverse cone.  This follows from Altmann's concentration
of $T^1$ in the Gorenstein degree and his classification of deformation
data with zero Kodaira--Spencer class.  Smoothness of cones over
summands alone does not impose the admissibility condition (D2), and
therefore does not yield this necessary degree restriction.
\end{remark}

\section{A census of the facet rigidity condition}\label{sec:census}

We apply the combinatorial criterion to the Kreuzer--Skarke
classification \cite{KS00}, restricted here to at most nine vertices.
The input files belong to the complete database copy whose reflexivity,
absence of repetitions up to lattice equivalence and completeness are
verified in \cite[\S6.1]{PaperI}; the verification report and input
manifest accompany the ancillary programs.  The broad census tests locking, condition
{\rm(L)} of Theorem~\ref{thm:Aprime}; it does not test the degree
coverage condition {\rm(C)}.  It also records rigidity of the
containing lattice facets when the locking criterion is inconclusive.
These counts concern facet rigidity.  The separate smoothability
classification in Section~\ref{sec:classify} uses the homological
obstruction of \cite{PaperIV} and the explicit family constructed below.

\subsection{Enumeration and rigidity classes}
For each unit-edge pentagonal $2$-face with one interior lattice point,
we consider its two containing facets.  Definition~\ref{def:locked}
tests rigidity by propagation of equalities among edge dilations.
For either facet where that criterion is inconclusive, the
edge-cycle system determines its Minkowski summand cone directly.
This gives three disjoint classes of pentagonal faces:
\begin{enumerate}
\item \emph{locked}: both containing facets satisfy
Definition~\ref{def:locked};
\item \emph{rigid, not locked}: both containing lattice facets are
Minkowski-rigid, but at least one does not satisfy that definition;
\item \emph{decomposable}: at least one containing lattice facet has
a nonhomothetic Minkowski summand.
\end{enumerate}
The first condition implies rigidity for all bounded cross-sections
of both facet cones.  The second asserts rigidity only at their
lattice cross-sections; in the range searched it does not occur.

\begin{table}[ht]
\caption{Facet rigidity in the Kreuzer--Skarke classification.
Each pentagonal face is counted once.  ``Locked'' means both facets
satisfy the edge-propagation criterion; ``rigid, not locked'' means
both lattice facets are rigid but the criterion is inconclusive;
``decomposable'' means at least one has a nonhomothetic summand.}\label{tab:census}
\centering
\begin{adjustbox}{max width=\textwidth, center}
\begin{tabular}{lrrrrr}
\toprule
vertices & polytopes & pentagons & locked & rigid, not locked & decomposable \\
\midrule
$5$ & $1\,561$ & $0$ & n/a & n/a & n/a \\
$6$ & $24\,189$ & $0$ & n/a & n/a & n/a \\
$7$ & $177\,446$ & $30$ & $30$ & $0$ & $0$ \\
$8$ & $834\,638$ & $882$ & $763$ & $0$ & $119$ \\
$9$ & $2\,867\,955$ & $11\,596$ & $8\,710$ & $0$ & $2\,886$ \\
\midrule
total & $3\,905\,789$ & $12\,508$ & $9\,503$ & $0$ & $3\,005$ \\
\bottomrule
\end{tabular}
\end{adjustbox}
\end{table}

The classes exhaust the pentagonal faces:
$8\,710+2\,886=11\,596$ at nine vertices, and
$9\,503+3\,005=12\,508$ in total.  The middle class is empty.

Exactly $9\,503$ of the $12\,508$ pentagonal faces satisfy
{\rm(L)}, and these are also exactly the faces whose two containing
lattice facets are Minkowski-rigid.  By Proposition~\ref{prop:locked},
locking gives rigidity at every bounded cross-section of both facet cones.
The count is nevertheless not an all-degree immovability count,
because the broad census does not impose {\rm(C)}.
The $77$ polytopes of Section~\ref{sec:classify} are counted separately.

\begin{remark}\label{rem:flatset}
Definition~\ref{def:locked} records a partition of the edges because
rules (S) and (A) require equality of the other dilation parameters.
Knowing only which parameters are constrained does not suffice.
\end{remark}

For a construction at a lattice facet degree, a decomposing facet is
necessary but not sufficient, because by
Definition~\ref{def:iv}(D4) what matters is the decomposition
\emph{induced} on the pentagon, and by \eqref{eq:pentdecomp} that must be
segment~$+$~triangle with all dilation parameters nonnegative across the facet.
Of the $119$ at eight vertices: all $119$ induce a nontrivial decomposition of
the pentagon, $49$ attain the segment-plus-triangle datum with nonnegative
dilations, and $0$ lie in a polytope whose whole germ inventory lies in the
framework of Definition~\ref{def:framework}.

\begin{remark}\label{rem:censuscorrection}
A filter that classified a $2$-face by its numbers of vertices and interior
points alone would admit polytopes at eight vertices that are not in the
framework, because it misses the lattice smoothness condition.  A triangle with
no interior point need not be smooth: it is smooth only if \emph{unimodular},
and one with an edge of lattice length two gives a non-isolated germ.
\end{remark}

\subsection{Nine vertices, and \texorpdfstring{$\Delta_9$}{Delta9}}
At nine vertices the database contains $11\,596$ pentagonal $2$-faces.  Of
these, $63$ lie in a polytope inside the admissible framework, and $2\,886$ lie
in a facet that decomposes.  Exactly one does both, and it is the pentagon of
$\Delta_9$.  Section~\ref{sec:thmB} constructs a smoothing from the compatible
Minkowski decomposition of this polytope.

Across all five vertex counts the number of pentagons in a framework polytope
is
$77$, and the number lying in a decomposing facet is $3\,005$; the single
polytope in both is $\Delta_9$.  Thus this range contains one candidate satisfying both combinatorial
conditions.

\subsection{The searched range, classified}\label{sec:classify}
Restricting the enumeration to polytopes that lie in the admissible framework and
carry a pentagonal $2$-face returns exactly $77$ of them, one pentagon apiece:
one at seven vertices, thirteen at eight, sixty-three at nine.  The negative classification in \cite[Prop.~9.2]{PaperIV}, together
with the smoothing below, decides all $77$ cases.

\begin{theorem}\label{thm:C}
Let $\Delta$ be a reflexive $4$-polytope with at most nine vertices lying in the
admissible framework and carrying a pentagonal $2$-face.  There are $77$ such
$\Delta$, and:
\begin{enumerate}
\item for $76$ of them the criterion of \cite{PaperIV} obstructs, so $X$ admits no
  one-parameter smoothing \cite[Prop.~9.2]{PaperIV} and hence, by
  \cite[Cor.~6.7]{PaperIV}, is \emph{not smoothable} over any base germ; $67$ of these have a single
  singular point, always a lone $\dP_7$ cone point, and $9$ have one node in
  addition;
\item for the remaining one, which is $\Delta_9$, the criterion is silent, and
  $X_9$ \emph{is} smoothable by Theorem~\ref{thm:B};
\item exactly $76$ of the $77$ pentagons are locked, and the unlocked one is
  $\Delta_9$'s.
\end{enumerate}
\end{theorem}

\begin{proof}
The classification contains one framework polytope at seven vertices,
thirteen at eight, and sixty-three at nine, each with one pentagonal
face.  The nonsmoothability of the $76$ cases in part~(1), including
their division into $67$ single-cone and $9$ node-plus-cone cases,
is \cite[Prop.~9.2]{PaperIV}.  Its Corollary~6.7 gives the assertion
over an arbitrary base germ.  Part~(2) follows from
Proposition~\ref{prop:mixed9} and Theorem~\ref{thm:B}.
The edge partitions of Definition~\ref{def:locked} identify precisely
those $76$ pentagons as locked, proving part~(3).
The exact enumeration, relation matrices and partitions are
reproduced by the programs listed in Appendix~\ref{app:cert}.
\end{proof}

\begin{remark}
The $76$ nonsmoothability assertions in part~(1) are proved in
\cite[Prop.~9.2]{PaperIV}.  The explicit family constructed here
treats the remaining polytope and completes the classification.
Its abstract smoothability also follows from
Theorem~\ref{thm04:criterion}.  Part~(3) is a separate combinatorial
observation: locking agrees with the negative classification in
this range.  The agreement does not extend to $X_{19}$, which is
locked and smoothable by Corollary~\ref{cor:x19smooth}.
\end{remark}

\section{An ambient smoothing of \texorpdfstring{$X_9$}{X9}}\label{sec:thmB}

\subsection{The polytope and its germs}
Let $\Delta_9$ be the convex hull of the following vertices:
\[
\begin{aligned}
&v_0 = (1,0,0,0), && v_1 = (0,1,0,0), && v_2 = (-1,-1,0,0), \\
&v_3 = (0,0,1,0), && v_4 = (0,0,0,1), && v_5 = (-4,-2,-1,0), \\
&v_6 = (-4,-2,0,-1), && v_7 = (3,1,1,1), && v_8 = (2,1,1,1).
\end{aligned}
\]
It is reflexive, with $12$ facets and $27$ two-faces, and it lies in the
admissible framework.

\begin{proposition}\label{prop:sing9}
$\Sing(X_9)$ consists of exactly three points: two ordinary double points, at
the $2$-faces $\{v_2,v_3,v_6,v_7\}$ and $\{v_2,v_4,v_5,v_7\}$, and one $\dP_7$
cone point, at the pentagon $\{v_3,v_4,v_5,v_6,v_8\}$.
\end{proposition}

\begin{proof}
Every vertex of $\Delta_9$ is primitive, so every ray of $\Sigma$ is a smooth
cone.  Every edge has lattice length one and unit greatest common divisor of
its $2 \times 2$ minors, so every two-dimensional cone is smooth; hence
$\PP_{\Delta_9}$ is smooth along the orbits of all cones of dimension $\le 2$.
Of the $27$ two-faces exactly the three listed have a singular cone, and each
has dual edge of lattice length one, so contributes exactly one point.
Finally, $\#(\Delta_9^\circ \cap M) = 162$ and this set contains every vertex of
$\Delta_9^\circ$, so a generic anticanonical section is nonzero at every
torus-fixed point and $X_9$ misses the zero-dimensional strata.
\end{proof}

\begin{proposition}\label{prop:mixed9}
The node relation system of $X_9$ has rank $2$ and both nodes are coloops.  The
criterion of \cite{PaperIV} is silent on $X_9$: at the $\dP_7$ point the branch
subspace is a single line, since $K_E^\perp$ has discriminant $7$ and carries a
unique $(-2)$-class up to sign, so the branch restriction there is one linear
condition; the branch-restricted kernel is one-dimensional, and no coordinate
vanishes identically on it.
\end{proposition}

\begin{proof}
The two nodal rows are independent.  Adding the degree-seven root
row gives a matrix with one-dimensional relation kernel, generated
by $(-2,-2,2)$, with entries indexed by the two nodes and the
$(-2)$-line at the cone point.  Every entry is nonzero, so the
criterion of \cite{PaperIV} imposes no forced-zero condition.
The vertex and ray-relation computations are reproduced in
Appendix~\ref{app:cert}.
\end{proof}

Thus $X_9$ is genuinely mixed in the sense of Remark~\ref{rem:coloop}.
The displayed relation and Theorem~\ref{thm04:criterion} prove that it
is smoothable.  We now construct a smoothing in an ambient family.

\subsection{The family}
Take the vertex degree
\[
  R \;=\; u_{11} \;=\; (0,1,-1,-1),
  \qquad \langle R,v\rangle=-1 \text{ for every } v \in F_{11},
\]
dual to the facet $F_{11} = \{v_2,v_3,v_4,v_5,v_6,v_7,v_8\}$, which contains all
three singular $2$-faces.  Downgrading along $R$, the bounded cell of the
level-$(-1)$ slice is the lattice facet $F_{11}$ itself.  We use the
following integral dilation vector; its two summands have lattice
vertices, which will give admissibility.

The Minkowski summand cone of $F_{11}$ has dimension $2$.  We take the decomposition
given by the dilation vector
\[
  t = (0,0,1,1,0,0,1,0,0,0,0)
\]
on the eleven edges of $F_{11}$, listed in lexicographic order of their vertex
index pairs as
\[
  (2,5),\,(2,6),\,(2,7),\,(3,6),\,(3,7),\,(3,8),\,(4,5),\,(4,7),\,(4,8),\,
  (5,6),\,(7,8),
\]
so the dilation is $1$ on the three edges $v_2v_7$, $v_3v_6$, $v_4v_5$ and $0$
on the other eight.  Figure~\ref{fig:F11} shows $F_{11}$, its three singular
$2$-faces and the three dilated edges.

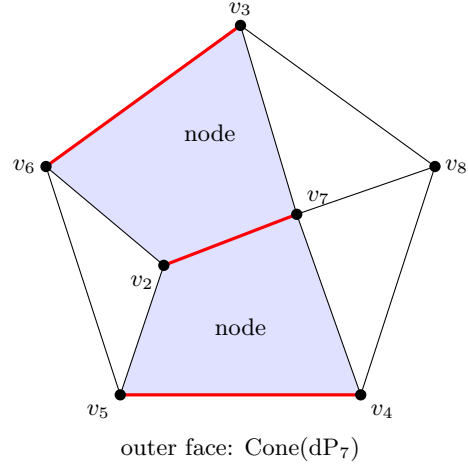
\begin{figure}[ht]
\centering
\begin{adjustbox}{max width=\textwidth, center}
\begin{tikzpicture}[scale=1.35,
  vtx/.style={circle,fill=black,inner sep=1.5pt},
  dil/.style={very thick,red},
  lbl/.style={font=\footnotesize}]
% outer pentagon 3-6-5-4-8, the dP_7 face
\coordinate (v3) at (0,2);
\coordinate (v6) at (-1.902,0.618);
\coordinate (v5) at (-1.176,-1.618);
\coordinate (v4) at (1.176,-1.618);
\coordinate (v8) at (1.902,0.618);
\coordinate (v2) at (-0.75,-0.35);
\coordinate (v7) at (0.55,0.15);
% shade the two nodal squares; the triangles stay white, the pentagon is the
% outer face and so has no bounded region of its own
\fill[blue!12] (v2) -- (v6) -- (v3) -- (v7) -- cycle;
\fill[blue!12] (v2) -- (v5) -- (v4) -- (v7) -- cycle;
% edges, dilation 0
\draw (v3) -- (v8) -- (v4);
\draw (v6) -- (v5);
\draw (v2) -- (v5); \draw (v2) -- (v6);
\draw (v7) -- (v3); \draw (v7) -- (v4); \draw (v7) -- (v8);
% edges, dilation 1
\draw[dil] (v3) -- (v6);
\draw[dil] (v5) -- (v4);
\draw[dil] (v2) -- (v7);
\foreach \i in {2,3,4,5,6,7,8} \node[vtx] at (v\i) {};
\node[lbl,above]      at (v3) {$v_3$};
\node[lbl,left]       at (v6) {$v_6$};
\node[lbl,below left] at (v5) {$v_5$};
\node[lbl,below right]at (v4) {$v_4$};
\node[lbl,right]      at (v8) {$v_8$};
\node[lbl,below left] at (v2) {$v_2$};
\node[lbl,above right]at (v7) {$v_7$};
\node[lbl] at (0,-0.95) {node};
\node[lbl] at (-0.30,0.95) {node};
\node[lbl] at (0,-2.15) {outer face: $\Cone(\dP_7)$};
\end{tikzpicture}
\end{adjustbox}
\caption{Schlegel diagram of the facet $F_{11}$, projected from beyond its
pentagonal $2$-face, which becomes the outer boundary.  The two shaded faces
are the squares $\{v_2,v_3,v_6,v_7\}$ and $\{v_2,v_4,v_5,v_7\}$ carrying
the nodes; the third singular face is the pentagon
$\{v_3,v_4,v_5,v_6,v_8\}$ carrying the $\dP_7$ cone point, which the projection
sends to the outer boundary.  The three unshaded faces are triangles, and the
germ there is smooth.  The three heavy edges are those on which the dilation
vector $t$ is $1$; it is $0$ on the other eight.}\label{fig:F11}
\end{figure}

A dilation vector determines a summand through the \emph{displacement maps}
$\varphi_0,\varphi_1$ on the vertices of $F_{11}$: fix the base vertex $v_2$,
set $\varphi_0(v_2)=0$, and along each edge $v_av_b$ put
\[
  \varphi_0(v_b)=\varphi_0(v_a)+t_{ab}\,(v_b-v_a),
\]
which is well defined precisely because $t$ solves the edge-cycle system, that
being what makes the value independent of the path from $v_2$.  Put
$\varphi_1(v_a)=(v_a-v_2)-\varphi_0(v_a)$.  Then
$\conv\varphi_0(V(F_{11}))+\conv\varphi_1(V(F_{11}))=F_{11}-v_2$, so the two
maps exhibit the summands.  Their restrictions to the vertices of each
face, with its tail cone added, define the cellwise decomposition after
the affine identification specified below.  The three singular
$2$-faces inherit the following decompositions:

\begin{table}[ht]
\caption{The decomposition induced at each germ.}\label{tab:decomp}
\begin{adjustbox}{max width=\textwidth, center}
\begin{tabular}{lll}
\toprule
$2$-face & germ & induced decomposition \\
\midrule
$\{v_2,v_3,v_6,v_7\}$ & node & primitive segment $+$ primitive segment \\
$\{v_2,v_4,v_5,v_7\}$ & node & primitive segment $+$ primitive segment \\
$\{v_3,v_4,v_5,v_6,v_8\}$ & $\Cone(\dP_7)$ & primitive segment $+$ unimodular triangle \\
\bottomrule
\end{tabular}
\end{adjustbox}
\end{table}

To specify the affine identifications and the gluing, use the
saturated basis
\[
(1,0,0,0),\qquad(0,1,1,0),\qquad(0,1,0,1)
\]
of $\ker R$, with splitting vector $w=v_1$.  In these coordinates put
\[
\begin{gathered}
a_0=(-1,0,0),\quad a_1=(3,1,1),\\
b_0=(-3,-1,0),\quad b_1=(-3,0,-1),\quad b_2=(-1,0,0),\quad b_3=0.
\end{gathered}
\]
The two summands in the negative slice are the convex hulls of
the images $A(v)=\varphi_0(v)+v_2+w$ and $B(v)=\varphi_1(v)$,
with the tail cone added.  The vertex images, in the order
$v_2,\ldots,v_8$, are
\[
A:\ (a_0,a_1,a_1,a_0,a_0,a_1,a_1),\qquad
B:\ (b_3,b_1,b_0,b_0,b_1,b_3,b_2).
\]
The tail rays are
\[
\begin{aligned}
q_0&=(-4,-1,0),&q_1&=(-4,0,-1),&q_2&=(0,0,1),\\
q_3&=(0,1,0),&q_4&=(1,0,0),&q_5&=(2,1,1),&q_6&=(3,1,1).
\end{aligned}
\]
Table~\ref{tab:x9slices} lists all maximal original charts.  An
index set in the $A$ or $B$ column means the convex hull of those
vertices plus the indicated tail.  The marking records complete
locus.  Write $A_j,B_j$ for the summands in row $j$.
All remaining cells are obtained from the nonzero face
cones of the twelve listed facets by the same slice construction.

\begin{table}[ht]
\centering
\caption{The two displaced slices and their markings.}
\label{tab:x9slices}
\begin{tabular}{rccccc}
\toprule
facet & vertices & tail rays & $A$ & $B$ & marked \\
\midrule
$0$ & $0136$ & $134$ & $01$ & $1$ & yes \\
$1$ & $0137$ & $346$ & $1$ & $13$ & yes \\
$2$ & $0145$ & $024$ & $01$ & $0$ & yes \\
$3$ & $0147$ & $246$ & $1$ & $03$ & yes \\
$4$ & $0156$ & $014$ & $0$ & $01$ & yes \\
$5$ & $02367$ & $4$ & $01$ & $13$ & no \\
$6$ & $02457$ & $4$ & $01$ & $03$ & no \\
$7$ & $0256$ & $4$ & $0$ & $013$ & no \\
$8$ & $134568$ & $01235$ & $01$ & $012$ & yes \\
$9$ & $1378$ & $356$ & $1$ & $123$ & yes \\
$10$ & $1478$ & $256$ & $1$ & $023$ & yes \\
$11$ & $2345678$ & $\varnothing$ & $01$ & $0123$ & no \\
\bottomrule
\end{tabular}
\end{table}

\begin{proposition}\label{prop:admissible}
The displayed assignment is an admissible Minkowski decomposition of
the two-sided slice complex in the sense of Definition~\ref{def:iv}.
\end{proposition}

\begin{proof}
The summands in Table~\ref{tab:x9slices} have lattice vertices.
Their minimizing faces therefore contain lattice points, proving
(D2).  Taking their sums and intersections verifies (D1) and (D3)
on the $69$ nonempty negative cells obtained from the displayed
facet data.

For (D4), form the finite set of ordered pairs
\[
\mathcal I=\left\{\left(\bigcap_{C\in J}C^0,
\bigcap_{C\in J}C^1\right):
\varnothing\ne J\subseteq\mathcal S^-\right\}.
\]
It is obtained from the $69$ pairs $(C^0,C^1)$ by closing under
$(U,V)\mapsto(U\cap D^0,V\cap D^1)$ for each cell $D$.
For every resulting pair and every $D$, the displayed rational
polyhedra satisfy
\[
U\cap D^0\prec U,\qquad V\cap D^1\prec V,\qquad
(U\cap D^0)+(V\cap D^1)\prec U+V.
\]
These three face relations suffice for arbitrary inclusions of
collections by transitivity.

One case illustrates why disjoint original cells must be included.
For facets $0$ and $2$, let $S=[a_0,a_1]$ and let
$\ell(x,y,z)=y-z$.  The functional is nonnegative on $A_0$ and
nonpositive on $A_2$, with
\[
A_0\cap A_2=S+\operatorname{pos}(q_4)
=A_0\cap\{\ell=0\}=A_2\cap\{\ell=0\}.
\]
On $B_0$ it is at least $1$, whereas on $B_2$ it is at most $-1$.
Thus $B_0\cap B_2=\varnothing$.  The first intersection is a
proper face and the other two required partial sums are empty
faces, although the original cells are disjoint.

The positive slice uses $C=C+\tail(C)$; its tail summand contains
zero in every minimizing face.  The same face comparisons verify
(D1)--(D4) there.  The complete rational intersection enumeration
from the vertex and dilation data printed here\footnote{The program
\path{slice_admissible.sage} reproduces this enumeration.}
gives $86$ distinct pairs
and $5\,934$ extensions on the negative slice, and $29$ pairs
and $841$ extensions on the positive slice.
\end{proof}

\begin{corollary}\label{cor:ambientsmooth}
By Recollection~\ref{thm:iv44} the decomposition determines a flat family
$\pi \colon \cP \to B$ with $\cP_0 = \PP_{\Delta_9}$, and by
Recollection~\ref{thm:cor212} the general fibre $\cP_b$ is smooth at all three germs
of $X_9$: the cones over a primitive segment and over a unimodular simplex are
smooth.
\end{corollary}

All three cones $\sigma_G$ are one-sided at $R$ with bounded cells, so
Lemma~\ref{lem:reach} does not intervene and the affine-locus hypothesis of
Recollection~\ref{thm:cor212} is satisfied at each.

\subsection{The general ambient fibre}

The general fibre has a larger torus action than the one used to
construct the family.  We describe its fan explicitly; this also
determines its singular locus.

\begin{proposition}\label{prop:toricfibre}
For $b\ne0$, the fibre $\cP_b$ is the toric fourfold with the following
primitive rays in $\ZZ\oplus\ker R$:
\begin{equation}\label{eq:x9rays}
\begin{aligned}
r_0&=(-1,-3,-1,0),&r_1&=(-1,-3,0,-1),\\
r_2&=(-1,-1,0,0),&r_3&=(-1,0,0,0),\\
r_4&=(0,1,0,0),&r_5&=(1,-1,0,0),\\
r_6&=(1,3,1,1).
\end{aligned}
\end{equation}
Its maximal cones, written as sets of ray indices, are
\begin{equation}\label{eq:x9maxcones}
\begin{gathered}
1456,\quad1346,\quad0456,\quad0346,\quad0145,\\
0134,\quad01256,\quad1236,\quad0236,\quad0123.
\end{gathered}
\end{equation}
This is a Gorenstein Fano fourfold.
\end{proposition}

\begin{proof}
The vertex levels are $R(v_1)=1$, $R(v_0)=0$, and $R(v_i)=-1$
for $i\ge2$.  With splitting vector $w=v_1$, each nonempty
positive coefficient is its tail cone, with vertex at zero.
A face cone with empty positive coefficient has tail either zero
or $\operatorname{pos}(v_0)$.  Both already occur as positive
coefficients, from $\operatorname{pos}(v_1)$ and
$\operatorname{pos}(v_0,v_1)$ respectively.  Thus the positive slice
is the entire tail fan.  Only the two displaced negative slices can
be nontrivial.  By \cite[Cor.~1.9]{Suess08}, the three-dimensional
torus action extends to a four-dimensional torus action.

Here is the fan, including the gluing data.  Move the two nontrivial
support points to $0$ and $\infty$.  For each marked row of
Table~\ref{tab:x9slices}, take the cone generated by
$(1,A)$, $(-1,B)$ and $(0,\tau)$.  For each unmarked row, take the
two cones generated separately by $(1,A),(0,\tau)$ and
$(-1,B),(0,\tau)$.  Removing proper faces gives
\eqref{eq:x9rays}--\eqref{eq:x9maxcones}.
The sections at first coordinate $1$ and $-1$ recover the two
subdivisions in the table.  Their markings also agree: the first
has eight maximal cells, all marked; the second has ten maximal
cells, eight marked.  The generic tail subdivision has eight
maximal cells, all marked.  Hence these are the marked subdivisions
of the given fibre, which determine the variety
\cite[Rem.~1.6 and Thm.~1.8]{Suess08}.

For completeness, the anticanonical Cartier characters on the ten
cones, in their listed order, are
\begin{equation}\label{eq:x9cartier}
\begin{gathered}
(-2,-1,-2,6),\ (1,-1,-2,3),\ (-2,-1,6,-2),\
(1,-1,3,-2),\ (-2,-1,6,6),\\
(1,-1,3,3),\ (0,1,-2,-2),\ (1,0,-2,0),\
(1,0,0,-2),\ (1,0,0,0).
\end{gathered}
\end{equation}
Each pairs to $-1$ with the rays of its cone and strictly greater
than $-1$ with every other ray.  The seven inequalities
$\langle m,r_i\rangle\ge-1$ define a bounded polytope with exactly
these vertices.  Consequently the displayed cones are its normal
fan, and $-K$ is Cartier and ample.  All the data needed for these
finite comparisons are printed above and in Table~\ref{tab:x9slices}.
\footnote{The program \path{toric_fibre.sage} reproduces the fan and
marked subdivisions.}
\end{proof}

\begin{proposition}\label{prop:charts}
The singular locus of $\cP_b$, for $b\ne0$, consists of one
torus-fixed point.
\end{proposition}
\begin{proof}
Nine of the maximal cones in \eqref{eq:x9maxcones} are unimodular.
The remaining cone $01256$ has five extremal rays in dimension
four, with relation
\[
r_0+r_1+r_6=2r_2+r_5.
\]
Its six facets have ray sets
\[
012,\quad015,\quad026,\quad126,\quad056,\quad156.
\]
For each facet, the greatest common divisor of the maximal minors
of its three primitive ray vectors is one.  Thus every proper
face is smooth, while the full cone is singular.  The corresponding
fixed point is therefore the only singular point of the fourfold.
In the original cover, Table~\ref{tab:x9slices} has eight
complete-locus charts and four affine-locus charts.  The only
singular full cone occurs in the complete-locus chart $F_8$;
every lower-dimensional cone arising from an affine-locus chart
is smooth as well.
\end{proof}

\subsection{The relative anticanonical system}\label{sec:h0}

By Proposition~\ref{prop:relativefano}, after shrinking $B$ the
family is proper and Gorenstein and
$\cL=\omega_{\cP/B}^{-1}$ is relatively ample and commutes with
base change.  Constancy and generation of its sections follow
from the special fibre.

\begin{proposition}\label{prop:h0}\label{prop:gg}
After shrinking $B$, the sheaf $E=\pi_*\cL$ is locally free of
rank $162$ and commutes with base change.  The evaluation map
$\pi^*E\to\cL$ is surjective.  In particular
$h^0(\cP_b,\cL_b)=162$ and $\cL_b$ is globally generated for every
$b$ near zero.
\end{proposition}
\begin{proof}
The line bundle $\cL_0$ is ample Cartier on the complete toric
variety $\PP_{\Delta_9}$.  Toric vanishing gives
$H^i(\cP_0,\cL_0)=0$ for $i>0$
\cite[Thm.~9.2.3]{CoxLittleSchenck}, and
$h^0(\cP_0,\cL_0)=\#(\Delta_9^\circ\cap M)=162$.
Cohomology and base change therefore makes $E$ a vector bundle of
rank $162$ commuting with base change
\cite[Tag~07VK]{Stacks}.  Its evaluation map is surjective on
the special fibre, since $\cL_0$ is globally generated.  The
cokernel has closed support whose proper image in $B$ misses
zero.  Removing that image proves the assertion.  After a further
shrinking, trivializing $E$ gives a common $162$-dimensional
space of sections surjecting onto every fibre.
\end{proof}

The general fan provides a direct check of the section count.
Writing a character as $(s,a,b,c)$, its anticanonical inequalities
are
\[
a\ge-1,\qquad a-1\le s\le\min(1,1-a),\qquad
b,c\le1-s-3a,\qquad b+c\ge-1-s-3a.
\]
Only $a=-1,0,1$ occur at lattice points.  At fixed $a,s$ there are
$\binom{5-s-3a}{2}$ choices of $(b,c)$.  Summing gives
\[
(45+36+28+21)+(15+10+6)+1=130+31+1=162.
\]

\subsection{The general anticanonical member}

\begin{theorem}[Explicit ambient smoothing of $X_9$]\label{thm:B}
$X_9$ is smoothable.  Precisely, let $g \in H^0(\cP,\cL)$ be general and
$\cX := \{g = 0\}$.  Then $\cX \to B$ is flat, $\cX_0$ is a generic
anticanonical hypersurface of $\PP_{\Delta_9}$, and $\cX_b$ is a smooth
Calabi--Yau threefold for general $b\in B$.
\end{theorem}

\begin{proof}
Use the trivialization of $\pi_*\cL$ from Proposition~\ref{prop:h0}, and choose
$g$ from a general constant vector in its $162$-dimensional fibre.  A nonzero
constant vector restricts to a nonzero section $g_b$ on every fibre.  The
fibres $\cP_b$ are integral normal $T$-varieties, so $g_b$ is a
nonzerodivisor.  The fibrewise regular-section criterion therefore makes
$\cX$ a relative effective Cartier divisor: locally, multiplication by $g$
gives an exact sequence
\[
 0\longrightarrow\cL^{-1}\xrightarrow{\ g\ }\cO_{\cP}
 \longrightarrow\cO_{\cX}\longrightarrow0
\]
which remains exact after restriction to every fibre.  Since the first two
terms are flat over $B$, the local criterion for flatness shows that
$\cO_{\cX}$ is flat over $B$.

$\cX_0$ is a general member of $|-K_{\PP_{\Delta_9}}|$: by
Proposition~\ref{prop:h0} the restriction
$H^0(\cP,\cL) \to H^0(\cP_0,\cL_0)$ is
surjective, so a general $g$ restricts to a general section on the special
fibre.  Hence $\cX_0$ has the singularities of Proposition~\ref{prop:sing9}.

Fix one $b_1\ne0$.  By Proposition~\ref{prop:gg} the system $\cL_{b_1}$ is
globally generated, so Bertini makes its general member smooth away from
$\Sing(\cP_{b_1})$.  Proposition~\ref{prop:charts} says that this singular
locus is finite; global generation also makes a general section nonzero at
each of its points.  Thus smooth sections form a nonempty open subset of
$H^0(\cP_{b_1},\cL_{b_1})$.  Under the trivialization of $\pi_*\cL$, its inverse
image in the constant section space $V\cong\CC^{162}$ is open and nonempty.
The sections whose restriction to $\cP_0$ is generic form another nonempty
open subset of the same irreducible space $V$, so a general $g$ has both
properties.  Smoothness is open in the flat family $\cX\to B$; hence
$\cX_b$ is smooth for $b$ in a nonempty open subset of $B$.

Finally relative adjunction gives
\[
 \omega_{\cX/B}
 \cong(\omega_{\cP/B}\otimes\cL)|_{\cX}
 \cong\cO_{\cX}.
\]
On the special fibre the sequence is
\[
 0\longrightarrow\omega_{\cP_0}\longrightarrow\cO_{\cP_0}
 \longrightarrow\cO_{\cX_0}\longrightarrow0.
\]
A complete toric variety has $H^i(\cP_0,\cO_{\cP_0})=0$ for $i>0$; Serre
duality gives the corresponding vanishings for $\omega_{\cP_0}$ below degree
$4$.  The long exact sequence therefore gives
$H^1(\cX_0,\cO_{\cX_0})=H^2(\cX_0,\cO_{\cX_0})=0$.  Upper semicontinuity for
the proper flat family $\cX\to B$ gives the same vanishings on nearby fibres.
Thus every smooth good fibre has trivial canonical bundle and
$h^1(\cO)=h^2(\cO)=0$, so it is a Calabi--Yau threefold.
\end{proof}

\begin{remark}\label{rem:sevencells}
The deformation is small in a precise sense: of the cells of the slice complex,
only seven change.  Three are the cells of the singular $2$-faces, where
Table~\ref{tab:decomp} and Corollary~\ref{cor:ambientsmooth} show that the cones
over the summands are smooth.  The other four are higher-dimensional cells.
Their orbit strata cannot be inferred from cell dimension in a complete-locus
chart; their contribution to the singular locus is covered instead by the fan and singular-locus calculation in
Proposition~\ref{prop:charts}.
\end{remark}

\subsection{The relation supplied by the ambient family}
The two nodal exceptional curves are independent when considered
alone.  Nevertheless the local tangents induced by the ambient family
give a relation involving both nodes and the degree-seven point.
The jump formula of \cite[Prop.~8.2]{PaperIV}, applied to the dilation
vector above, gives three classes whose unique relation has every
coefficient nonzero.  The degree-seven component lies on its
$(-2)$-root line.  Thus the construction realizes the relation required
by Theorem~\ref{thm04:criterion} and moves all three germs
simultaneously; a relation involving only the nodes cannot do so.

\appendix

\section{Other ambient deformation constructions}\label{sec:empty}

We compare two further ambient constructions with the examples above.
The non-polynomial correction term vanishes throughout the toric
framework used here.  The global Minkowski construction has no
nontrivial input on $\Delta_{19}$.

\subsection{Mavlyutov's non-polynomial deformations}
For a smooth crepant Batyrev hypersurface, Mavlyutov
\cite{Mav04,Mav05} identifies deformations not induced by the ambient
linear system.  Their contribution to the dimension is
\begin{equation}\label{eq:mav}
  \sum_{\operatorname{codim}\Gamma = 2} \ell^*(\Gamma)\,\ell^*(\Gamma^*),
\end{equation}
where $\ell^*$ is the number of relative-interior lattice points and $\Gamma^*$
runs over the dual faces, which for a codimension-two face $\Gamma$ of
$\Delta^\circ$ is an \emph{edge} of $\Delta$.

\begin{remark}\label{prop:mavempty}
If every edge of $\Delta$ has lattice length one, in particular if $\Delta$
is in the admissible framework, then \eqref{eq:mav} vanishes: $\ell^*(\Gamma^*)=0$ for every
edge $\Gamma^*$.  This is a sufficient condition for the vanishing,
not a characterization.  The formula concerns the smooth crepant
hypersurface and does not classify all deformations of the singular $X$.
\end{remark}

\subsection{Global Minkowski decomposition}
The construction in \cite[\S7, Thm.~7.1]{Mav09} takes as input a Minkowski
decomposition of $\Delta$ itself.

\begin{proposition}\label{prop:globalindec}
$\Delta_{19}$ is Minkowski-indecomposable: its only weak Minkowski summands are
its own homothets.
\end{proposition}

\begin{proof}
The weak Minkowski summands of a polytope are the solutions of the edge-cycle
system \eqref{eq:closing} imposed simultaneously on all of its $2$-faces, one
vector equation per $2$-face in the edge dilations.  $\Delta_{19}$ has $f$-vector
$(19,64,72,27)$, so the system has $64$ unknowns and $72$ vector equations, and
its solution space is one-dimensional, spanned by the all-ones vector, which is
the homothety direction.  A polytope whose summand space is a single ray is
indecomposable.  The rank computation is exact over $\QQ$ and
the $f$-vector satisfies Euler's relation.\footnote{The program
\texttt{global\_decomp.py} reproduces these two calculations.}
\end{proof}

Thus the global Minkowski construction has no nontrivial input on
$\Delta_{19}$.

\subsection{Why one cannot simply refine to a simplicial ambient}\label{sec:mpcp}
It is natural to try to avoid the non-simplicial ambient by passing to a
maximal projective crepant partial (MPCP) desingularization, which cones off
every boundary lattice point of $\Delta$.  The specified Minkowski
decomposition need not descend to the resulting subdivision.

Over a singular $2$-face $G$, an MPCP subdivision refines $\sigma_G$ using the
interior lattice points of $G$; at a vertex degree the slice cell \emph{is} $G$,
so the cell is cut into triangles.  A triangle has only homothetic Minkowski summands.  For a
unimodular triangle, lattice admissibility forces any such decomposition
to be trivial: its primitive edges cannot split into two nonzero
integral dilations.  On $\Delta_9$ a full MPCP triangulation also
triangulates the two node squares, although they have no interior
lattice points.  The pentagon has one interior point and its star
subdivision consists of five triangles.  The original
segment-plus-triangle datum is therefore no longer a decomposition
of a cell in the refined complex.

This observation concerns the specified Minkowski datum on the slice
complex.  The partial-resolution method of Section~\ref{sec04:smoothing}
uses deformations of a different space and subsequent blow-down; it
does not require that datum to survive on an MPCP refinement.

\section{Exact computations and reproducibility}\label{app:cert}

The following programs reproduce the finite equations, lattice data
and relation matrices used in the proofs.  Their inputs are the stated
lattice and local-equation data, together with the Kreuzer--Skarke
database for the enumeration.  The table records which
mathematical assertions each program checks.

{\small
\begin{longtable}{@{}p{0.43\linewidth}p{0.43\linewidth}r@{}}
\toprule
statement & script & checks \\
\midrule
\endhead
Rem.~\ref{rem:reachscope}
  & \texttt{one\_facet.py} & $7$ \\
Prop.~\ref{prop:rules}, Def.~\ref{def:locked}, Prop.~\ref{prop:locked}
  & \texttt{locking.py} & $12$ \\
Thm.~\ref{thm:A}, the line $\cA_P$, its two ranges, and cell rigidity
  & \texttt{slice\_rigidity.py} & $49$ \\
\quad independent reimplementation
  & \texttt{slice\_rigidity.sage} & $24$ \\
\quad face compatibility in Corollary~\ref{cor:lockedrel}
  & \texttt{def41\_check.sage} & $16$ \\
Prop.~\ref{prop:globalindec}
  & \texttt{global\_decomp.py} & $12$ \\
Section~\ref{sec:census}, the enumeration
  & \texttt{dp7\_facet\_sweep.py} & n/a \\
\quad induced pentagon decompositions
  & \texttt{candidates.sage} & n/a \\
Appendix~\ref{sec:mpcp}
  & \texttt{refine.py} & $5$ \\
Props.~\ref{prop:sing9}, \ref{prop:mixed9}
  & \texttt{v09\_candidate.py}, \texttt{sing\_locus.py} & $6$ \\
Prop.~\ref{prop:admissible}, Cor.~\ref{cor:ambientsmooth}
  & \texttt{slice\_admissible.sage} & $36$ \\
Cor.~\ref{cor:ambientsmooth}, completeness of the general fibre
  & \texttt{general\_fibre.sage} & $6$ \\
Prop.~\ref{prop:h0}; independent semiampleness diagnostic
  & \texttt{hzero.sage} & n/a \\
Prop.~\ref{prop:toricfibre}, \ref{prop:charts}
  & \texttt{toric\_fibre.sage} & n/a \\
Prop.~\ref{prop:charts}
  & \texttt{charts.sage} & $3$ \\
Rem.~\ref{rem:sevencells}
  & \texttt{final\_check.sage} & $3$ \\
Local branch tangents correspond to $\Rsm_p$
  & \texttt{branch\_locus.py} & $17$ \\
\quad comparison with cyclic period coordinates
  & \texttt{smoothing\_locus.py} & $13$ \\
Thm.~\ref{thm:C}, and \cite[Prop.~9.2]{PaperIV}
  & \texttt{classification.py}$^{\dagger}$ & $10$ \\
the relation classes of $X_9$ and $X_{19}$
  & \texttt{relation\_class.py}$^{\dagger}$ & $36$ \\
\quad from Theorem~\ref{thm:B}'s family
  & \texttt{theoremB\_class.py} & $10$ \\
\bottomrule
\end{longtable}
}

The degree restriction is proved in Lemma~\ref{lem:degreegate}.
Its finite input is the affine degree line and the two pairing
formulas, reproduced by \path{slice_rigidity.py}.  Smoothness of
cones over summands alone does not impose admissibility (D2) and is
not used as a necessary degree test.

The two marked $\dagger$ certify statements of \cite{PaperIV} as well as
statements here, and live in that paper's certificate directory of the shared
repository, from which the scripts of this paper already import their exact
toric modules.  Two further scripts there,
\texttt{kernel\_equality.py} and \texttt{lone\_germ.py}, certify
\cite[Prop.~8.3]{PaperIV} and \cite[\S9.2]{PaperIV}, used in the companion computations.

\texttt{slice\_rigidity.py} and \texttt{slice\_rigidity.sage} implement the same
statement in different systems, pure Python with integer edge equalities and
Sage with convex hulls over $\QQ$ and then over $\QQ(s)$, and are
compared against each other and against an independent facet computation.

\texttt{hzero.sage} implements the invariant-divisor formula of
\cite{PS11} and is validated on that source's worked example,
reproducing the printed value $h^0(-K_{\PP(\Omega_{\PP^2})}) = 27$, and then
cross-checked on the special fibre against the toric lattice-point count.

On the special fibre, \path{charts.sage} recovers the
positive-dimensional singular loci in the four charts containing
singular $2$-faces.  On the general fibre, it verifies the smoothness
of every cone corresponding to a positive-dimensional orbit.

The jump formula of \cite[Prop.~8.2]{PaperIV}, applied to the
Minkowski data of Theorem~\ref{thm:B}, gives the three local tangent
classes and their relation with every coefficient nonzero.
The file \path{theoremB_class.py} reconstructs this relation and
checks that the degree-seven component lies on the $(-2)$-line.

The file \path{locking.py} validates the rules of
Proposition~\ref{prop:rules} on polytopes with known summand cones, checks the
necessity of consecutive edges, and verifies for the two facets of
$\Delta_{19}$ containing the pentagon that their cross-sections at lattice
degrees positive on the facet are Minkowski-rigid, as
Proposition~\ref{prop:locked} predicts.  Its edge partitions agree with
those computed by \path{dp7_facet_sweep.py} on all $27$ facets
of $\Delta_{19}$.

\subsection{Certificates for the smoothing criterion and deformation spaces}
The following additional computations use exact arithmetic.  Their
geometric implications are proved in the corresponding sections;
the scripts verify the finite local equations, face data, and
relation matrices used there.
\begin{itemize}
\item \path{a2_simultaneous_partial_resolution.py} checks the incidence
charts, nodal Hessians, and inverse maps of
Proposition~\ref{prop03:a2family} ($33$ assertions).
\item \path{dp6_nodal_partial_resolutions.py} checks the two degree-six
subdivisions, normal bundles, root classes, and local parameter maps
of Proposition~\ref{prop03:localmodels} ($43$ assertions).
\item \path{a2_branch_selection_counterexample.py} reconstructs the
$25$-vertex polytope, its faces, divisor rows, four relation kernels,
and three obstruction identities in Section~\ref{sec05:counterexample}
($272$ assertions).
\item \path{branch_smoothness_inputs.py} verifies the Hodge data,
degree-six vanishing inputs, nonzero local projections, and the
relation-space comparisons under nodal replacement ($52$ assertions).
\item \path{counterexample_deformation_germ.py} checks the full
relation matrix, cone image, component tangent ideals, their
intersection, and the resulting dimensions ($57$ assertions).
\item \path{a2_counterexample_rational_replay.py} independently
repeats the counterexample's rational linear algebra and Hodge
input with Python's standard-library rational arithmetic
($50$ assertions).
\item \path{flop_weight_zero.py} computes the weight-zero Hodge cohomology
of the degree-seven cone, its two crepant resolutions and their common open
set, used in the proof of Lemma~\ref{lem04:globaltangent} ($12$ assertions).
\item \path{d19_quadric_partial_resolution.py} checks the $X_{19}$
partial fan, local quadric-cone cohomology, Hodge input, and the
matrices printed in Corollary~\ref{cor:x19smooth}.
\end{itemize}
The Sage-dependent files run with \texttt{sage -python}; the rational
replay runs with \texttt{python3}.  They resolve their inputs relative
to the repository.  The finite computations do not replace the
analytic lifting, blow-down, or necessity proofs.

The additional files in \path{figures/code/} run with Python's
standard library.  The file \path{compute01_degree5_local.py}
verifies the five conic-pencil classes and their $A_4$ coordinates,
the Cremona permutation, the Grassmannian Pieri integrals and
Chern classes, and the two exterior-product ranks used in
Section~\ref{sec03:degree5}.  The file
\path{compute01_degree8_local.py} verifies the Cayley lattice,
projection monomials, and link lattice in the quadric-cone proof.

\bigskip
\noindent{\sc Bernd Johannes Wuebben, New York, NY,}
\texttt{wuebben@gmail.com}

\end{document}